\documentclass[9pt]{amsart}
\usepackage{amsmath}
\usepackage{amsthm}
\usepackage{amsbsy}
\usepackage{amssymb}
\usepackage{amsfonts}
\usepackage[dvips]{lscape}
\usepackage{xcolor}
\usepackage{amscd}
\usepackage[all,cmtip]{xy}
\usepackage{euscript}
\usepackage{parskip}
\usepackage{enumerate}
\usepackage[all,cmtip]{xy}
\usepackage{tikz-cd}
\usepackage{libertine}
\usepackage[libertine]{newtxmath}

\makeatletter
\newcommand{\manuallabel}[2]{\def\@currentlabel{#2}\label{#1}}
\makeatother

\usepackage[margin=1in]{geometry}
\usepackage[expansion=false]{microtype}

\usepackage[pdfusetitle,unicode,hidelinks]{hyperref}

\theoremstyle{plain}

\theoremstyle{plain}
\newtheorem{theorem}{Theorem}[section]

\newtheorem{proposition}[theorem]{Proposition}
\newtheorem{lemma}[theorem]{Lemma}
\newtheorem{corollary}[theorem]{Corollary}

\newtheorem*{THMA}{Theorem~\textup{A}}
\newtheorem*{THMB}{Theorem~\textup{B}}

\theoremstyle{remark}

\theoremstyle{definition}
\newtheorem{definition}[theorem]{Definition}

\newcommand{\cal}{\EuScript}

\let\lim=\relax

\DeclareMathOperator*{\lim}{lim}

\renewcommand{\contentsname}{}

\newcommand{\vol}{\textup{vol}}

\newcommand{\grad}{\textup{grad}}

\begin{document}
\title{Stiefel manifolds and upper bounds for spherical codes and packings}
\author{Masoud Zargar}\footnotetext{This project was financially supported partly by the University of Southern California.}
\email{mzargar1225@gmail.com}
\renewcommand{\contentsname}{}
\date{\today}
\maketitle
\vspace{-1cm}
\begin{center}
\end{center}
\begin{abstract}
\begin{small}
We improve upper bounds on sphere packing densities and sizes of spherical codes in high dimensions. In particular, we prove that the maximal sphere packing densities $\delta_n$ in $\mathbb{R}^n$ satisfy
\[\delta_n\leq \frac{1+o(1)}{e}\cdot \delta^{\textup{KL}}_{n}\]
for large $n$, where $\delta^{\textup{KL}}_{n}$ is the best bound on $\delta_n$ obtained essentially by Kabatyanskii and Levenshtein from the 1970s with improvements over the years. We also obtain the same improvement factor for the maximal size $M(n,\theta)$ of $\theta$-spherical codes in $S^{n-1}$: for angles $0<\theta<\theta'\leq\frac{\pi}{2}$,
\[M(n,\theta)\leq \frac{1+o(1)}{e}\cdot \frac{M_{\textup{Lev}}(n-1,\theta')}{\mu_n(\theta,\theta')}\]
for large $n$, where $\mu_n(\theta,\theta')$ is the mass of the spherical cap in the unit sphere $S^{n-1}$ of radius $\frac{\sin(\theta/2)}{\sin(\theta'/2)}$, and $M_{\textup{Lev}}(n-1,\theta')$ is Levenshtein's upper bound on $M(n-1,\theta')$ when applying the Delsarte linear programming method to Levenshtein's optimal polynomials. In fact, we prove that there are no analytic losses in our arguments and that the constant $\frac{1}{e}=0.367...$ is optimal for the class of functions considered. Our results also show that the improvement factor does not depend on the special angle $\theta^*=62.997...^{\circ}$, explaining the numerics in~\cite{SZ}. In the spherical codes case, the above inequality improves the Kabatyanskii--Levenshtein bound by a factor of $0.2304...$ on geometric average. Along the way, we construct a general class of functions using Stiefel manifolds for which we prove general results and study the improvement factors obtained from them in various settings.
\end{small}
\end{abstract}
\vspace{-0.75cm}
\setcounter{tocdepth}{1}
\tableofcontents
\vspace{-0.5cm}
\section{History, methods, and main results}\label{intro}
\subsection{Sphere packings}\label{intro:packings}
Maximal sphere packing densities $\delta_n$ in $\mathbb{R}^n$ and related quantities have been of interest for at least a few centuries, partly for being simple to define yet very difficult to calculate, and partly for their appearance in different areas ranging from coding theory and physics to even counting $\ell$-adic sheaves on curves over finite fields. Though they have been extensively studied using different ideas and methods from harmonic analysis, number theory, combinatorics, optimization, and probability theory, very little is known about sphere packings. In this paper, we improve upper bounds on sphere packing densities and sizes of spherical codes, and study the connection between multiple projections and such bounds. Before providing relevant historical facts, methods, and the main results of this paper, we begin by defining the basic objects of study.\\
\\
Maximal sphere packing densities $\delta_n$ are defined as follows. Denote by $B_n(\boldsymbol{x},r)$ the open ball in $\mathbb{R}^n$ of radius $r>0$ centered at $\boldsymbol{x}\in\mathbb{R}^n$, distance measured using the Euclidean norm $|\cdot|$. For every (discrete) set $S$ of points in $\mathbb{R}^n$ such that any two distinct $\boldsymbol{x},\boldsymbol{y}\in S$ satisfy $|\boldsymbol{x}-\boldsymbol{y}|\geq 2$, we obtain a packing
\[\mathcal{P}:=\bigcup_{\boldsymbol{x}\in S}B_n(\boldsymbol{x},1),\]
a disjoint union of open unit balls. This is called a \textit{sphere packing} in $\mathbb{R}^n$; $S$ may vary. Define for each $r>0$ and each sphere packing $\mathcal{P}$ the quantity
\[\delta_{\mathcal{P}}(r):=\frac{\vol(\mathcal{P}\cap B_n(\boldsymbol{0},r))}{\vol(B_n(\boldsymbol{0},r))},\]
where the volume $\vol(.)$ is with respect to the Euclidean metric. The \textit{packing density} of $\mathcal{P}$ is then given by
\[\delta_{\mathcal{P}}:=\limsup_{r\rightarrow+\infty}\delta_{\mathcal{P}}(r)\in [0,1].\]
\begin{definition}The \textit{maximal sphere packing density} in $\mathbb{R}^n$ is
\[\delta_n:=\sup_{\mathcal{P}\subset\mathbb{R}^n}\delta_{\mathcal{P}},\]
where the supremum is over all sphere packings $\mathcal{P}$ in $\mathbb{R}^n$.
\end{definition}
$\delta_n$ is exactly known only in dimensions $n=1,2,3,8,24$. Clearly, $\delta_1=1$. $\delta_2=\frac{\pi}{2\sqrt{3}}=0.9068...$ was proved by Fejes T\'oth in 1942~\cite{FT}; this density comes from the hexagonal lattice. For $n=2$, it was essentially already proved by Lagrange in 1773 that among all \textit{lattice} packings, that is, packings where the centers of the unit balls are chosen to be from points in a lattice, the hexagonal lattice produces the densest possible packing. For $n=3$, Kepler conjectured in 1611 that $\delta_3=\frac{\pi}{3\sqrt{2}}=0.7404...$~\cite{Kepler}. This was proved by T. Hales~\cite{H} in 1998 by using extensive computerized proof by exhaustion. In this case, Gauss~\cite{Gauss} had proved in 1831 the $n=3$ analogue of Lagrange's theorem. The cases $n=8,24$ were resolved by M. Viazovska and collaborators~\cite{Via1},~\cite{Via2} recently using modular forms and the Cohn--Elkies linear programming method. They were shown to be realized by the $E_8$ and Leech lattices, respectively. In some sense, the few dimensions $n=1,2,3,8,24$ in which we know the precise value of $\delta_n$ are very special, and precise calculations of $\delta_n$ for any other given dimension $n$ is extremely difficult. They are instead bounded from above using linear and semi-definite programming methods, and from below either using probabilistic methods or through explicit constructions.\\
\\
For lower bounds, a simple greedy algorithm gives $\delta_n\geq 2^{-n}$, while Minkowski~\cite{Minkowski} obtained a constant improvement $\delta_n\geq (2+o(1))\cdot 2^{-n}$. In 1947, Rogers~\cite{Rogers} proved the first asymptotically growing improvement by showing that $\delta_n\geq (1+o(1))cn\cdot 2^{-n}$ with $c=\frac{2}{e}$. Over the years, the constant $c$ was enlarged to $c=1.68$ by Davenport--Rogers~\cite{DR}, then $c=2$ by Ball~\cite{Ball}, $c=\frac{6}{e}$ by Vance~\cite{Vance}, and $c=65963$ by Venkatesh~\cite{Venkatesh}. Venkatesh used Siegel's mean value theorem~\cite{Siegel} to obtain lower bounds by averaging over the moduli space of unimodular lattices. The best asymptotic lower bound to date is the recent inequality $\delta_n\geq (1-o(1))n\log(n)2^{-n-1}$ due to Campos--Jenssen--Michelsen--Sahasrabudhe~\cite{CJMS}, improving previous lower bounds by a factor of $\log(n)$ using a new graph-theoretic input.\\
\\
Though semi-definite programming is more powerful than linear programming for upper bounds, it is computationally feasible only in low dimensions. The best asymptotic upper bounds to date on $\delta_n$ are obtained using linear programming methods, the best known exponential bound being
\begin{equation}\label{KLsphbound}\delta_n\leq 2^{-n(0.599+o(1))}\textup{ as }n\rightarrow\infty,\end{equation}
dating back to Kabatyanskii--Levenshtein~\cite{KL} from 1978. Though it is expected by some that $\delta_n$ is close to the obtained lower bounds, this seems out of reach at this point. Aside from not having better lower bounds, Torquato--Stillinger~\cite{ST} have proved that the existing linear programming methods for upper bounds on $\delta_n$ cannot give a bound better than $\delta_n\leq 2^{-n\left(\log_2\left(\sqrt{\frac{8}{e}}\right)+o(1)\right)}$, where $\log_2\left(\sqrt{\frac{8}{e}}\right)=0.7786...$. In an unreleased paper, Sardari and Zargar have obtained a different proof of this dual linear programming bound which may lead to greater insight into the limits of the linear programming method. Nevertheless, the numerics of Afkhami-Jeddi--Cohn--Hartman--de Laat--Tajdini~\cite{AJCH} suggest that linear programming should be able to improve inequality~\eqref{KLsphbound} at least by a small exponential quantity. In fact, they conjecture that $\delta_n\leq 2^{-n(0.6044+o(1))}$ should be attainable from the current linear programming methods.\\
\\
Given that the lower and upper bounds on sphere packing densities are exponentially far apart, improvements to sphere packing density bounds have been sought for a long time. Improvements on upper bounds have generated many papers, \cite{Lev79},~\cite{CZ},~\cite{SZ} to name a few. However, progress regarding even linear improvements has been slower than that for lower bounds. The first explicit and uniform constant improvement to previous upper bounds was recently obtained by the author and Sardari~\cite{SZ}. In that paper, we proved a more effective version of the following theorem: for sufficiently large $n$,
\begin{equation}\label{SZbound}\delta_n\leq 0.4325\cdot\delta_n^{\textup{KL}},\end{equation}
where $\delta_n^{\textup{KL}}$ was the best previous bound obtained from the Delsarte linear programming method by Kabatyanskii--Levenshtein with some improvements over the years. More precisely, $\delta_n^{\textup{KL}}$ is the upper bound obtained using inequality~\eqref{CohnZhao} of Cohn--Zhao below proved in~\cite{CZ} combined with Levenshtein's~\cite{Lev79}. See~\eqref{Levdef},~\eqref{dKL}, and~\eqref{dKLstar} below for a definition of $\delta_n^{\textup{KL}}$. However, in~\cite{SZ}, we conjectured based on numerics that a sharp analysis of the functions that we constructed should lead to the improvement factor $\frac{1}{e}$. However, there were obstructions that we could not resolve. We prove that the improvement factor $\frac{1}{e}(1+o(1))$ is possible and optimal; see Theorem~\ref{THMone} below. We explain this after discussing all relevant concepts and methods, including the closely related problem of spherical codes.\\
\\
$\delta_n$ are often bounded from above by applying the Cohn--Elkies linear programming method, an incarnation of which is encapsulated in the following theorem, proved using the Poisson summation formula on tori.
\begin{theorem}[Cohn--Elkies~\cite{CE}]\label{CohnElkiesthm}
Suppose $f:\mathbb{R}^n\rightarrow\mathbb{R}$ is a continuous, positive semi-definite, and integrable function satisfying $f(\boldsymbol{x})\leq 0$ whenever $|\boldsymbol{x}|\geq 1$, and $\widehat{f}(\boldsymbol{0})>0$. Then
\begin{equation}\label{CEineqgeneral}
\delta_n\leq \frac{\vol(B_n(\boldsymbol{0},1)) f(\boldsymbol{0})}{2^n\widehat{f}(\boldsymbol{0})}.
\end{equation}
\end{theorem}
In this theorem, $f:\mathbb{R}^n\rightarrow\mathbb{R}$ is positive-semidefinite if for every $N$ and every $\boldsymbol{x}_1,\hdots,\boldsymbol{x}_N\in\mathbb{R}^n$, the matrix $[f(\boldsymbol{x}_i-\boldsymbol{x}_j)]_{ij}$ is positive semi-definite. The Fourier transform of $f$ is given by
\begin{equation*}
\widehat{f}(\boldsymbol{\xi}):=\int_{\mathbb{R}^n}f(\boldsymbol{x})e^{2\pi i\left<\boldsymbol{\xi},\boldsymbol{x}\right>}d\boldsymbol{x},
\end{equation*}
where $\left<\cdot,\cdot\right>$ is the Euclidean inner product on $\mathbb{R}^n$. By a theorem of Bochner~\cite{Bochner}, continuous positive semi-definite functions are exactly those that are Fourier transforms of finite Borel measures. This implies that when $f$ and $\widehat{f}$ are both integrable, then $f$ is positive semi-definite if and only if $\widehat{f}\geq 0$ everywhere. The Cohn--Elkies linear programming method is the minimization of the right hand side of inequality~\eqref{CEineqgeneral} as $f$ varies over functions satisfying conditions in Theorem~\ref{CohnElkiesthm}. To be precise, the way the above theorem is formulated corresponds to considering sphere packings by spheres of radius $\frac{1}{2}$, hence the normalization factor $\frac{1}{2^n}$. Functions satisfying the Cohn--Elkies linear programming conditions could be constructed in different ways, the most powerful of which in high dimensions to date is through spherical codes.
\subsection{Spherical codes}\label{intro:codes}
Upper bounds on sphere packing densities are intimately related to the maximal sizes of spherical codes. In fact, the limitations in our understanding of spherical codes are comparable to those of sphere packings. Maximal sizes of spherical codes are, in turn, bounded from above using yet another linear programming method known as the Delsarte linear programming method~\cite{Delsarte},~\cite{Delsarte2},~\cite{KL} predating and motivating the Cohn--Elkies linear programming method. In this subsection, we recall the definition of spherical codes, how they are studied using the Delsarte linear programming method, and previous results. In Subsection~\ref{mainresults}, we explain how spherical codes and packings are related to each other, state our main results, and give an overview of the methods used in their proofs.\\
\\
Given an angle $0<\theta\leq \pi$, a \textit{$\theta$-spherical code} in the unit sphere $S^{n-1}=\{\boldsymbol{x}\in\mathbb{R}^n:|\boldsymbol{x}|=1\}$ is any (finite) set $A$ of points in $S^{n-1}$ such that any two distinct points are separated by angular distance at least $\theta$, that is, for every $\boldsymbol{x}\neq \boldsymbol{y}$ in $A$, $\left<\boldsymbol{x},\boldsymbol{y}\right>\leq\cos\theta$.
\begin{definition}
$M(n,\theta)$ is the largest possible size of $\theta$-spherical codes in $S^{n-1}$. 
\end{definition}
For example, $M(n,\frac{\pi}{3})$ are the kissing numbers with $M(3,\frac{\pi}{3})=12$ the topic of a famous discussion in the year 1694 between Newton and Gregory.\\
\\
Upper bounds on $M(n,\theta)$ may be obtained using the Delsarte linear programming method in the following manner. Consider a probability measure $\mu$ on $[-1,1]$. This gives the inner product
\begin{equation*}\label{L2}\left<f,g\right>_{L^2(\mu)}:=\int_{-1}^1f(t)g(t)\mu(dt)\end{equation*}
on the $\mathbb{R}$-vector space $\mathbb{R}[t]$ of real polynomials. Suppose now that $p_0,p_1,\hdots\in\mathbb{R}[t]$, with $\deg p_i=i$ for every $i$, is a basis orthogonal with respect to this inner product and $p_0=1$. Furthermore, assume that each $p_k$ is positive semi-definite on $S^{n-1}$, that is, for every $N\in\mathbb{N}$ and every finite subset $\{\boldsymbol{x}_1,\hdots,\boldsymbol{x}_N\}\subset S^{n-1}$ of arbitrary size and every $a_1,\hdots,a_N\in\mathbb{R}$,
\begin{equation*}\label{possemi}\sum_{i,j}a_ia_jp_k\left(\left<\boldsymbol{x}_i,\boldsymbol{x}_j\right>\right)\geq 0.\end{equation*}
Equivalently, the matrix $\begin{bmatrix}p_k\left(\left<\boldsymbol{x}_i,\boldsymbol{x}_j\right>\right)\end{bmatrix}_{ij}$ is positive semi-definite for any $N$ and any $\{\boldsymbol{x}_1,\hdots,\boldsymbol{x}_N\}\subset S^{n-1}$.\\
\\
Given a function $f=\sum_{k=0}^{\infty}f_kp_k$, with $f_k\in\mathbb{R}$ for every $k$, such that
\begin{enumerate}[i)]
\item $f(t)\leq 0$ for every $t\in [-1,\cos\theta]$, and
\item $f_k\geq 0$ for every $k$ and $f_0>0$,
\end{enumerate}
we obtain for any given $\theta$-spherical code $\{\boldsymbol{x}_1,\hdots,\boldsymbol{x}_N\}\subset S^{n-1}$ that
\begin{equation}\label{Delsarteineq}Nf(1)\geq \sum_{i,j}f\left(\left<\boldsymbol{x}_i,\boldsymbol{x}_j\right>\right)=\sum_{k=0}^{\infty}f_k\sum_{i,j}p_k\left(\left<\boldsymbol{x}_i,\boldsymbol{x}_j\right>\right)\geq N^2f_0,\end{equation}
from which we obtain that
\begin{equation}\label{Delsartebnd}M(n,\theta)\leq\frac{f(1)}{f_0}.\end{equation}
Note that since $p_0=1$ and $\{p_k\}_{k\geq 0}$ are orthogonal, $f_0=\int_{-1}^1f(t)\mu(dt)$. The equality in~\eqref{Delsarteineq} is a trace formula comparing a geometric quantity, the left hand side, to a spectral quantity, the right hand side. In order to bound $M(n,\theta)$ from above, we want to minimize the right hand side of~\eqref{Delsartebnd} over all functions $f$ satisfying conditions i) and ii) above. This is the Delsarte linear programming method with respect to suitable measures $\mu$. Note the similarity to the Cohn--Elkies linear programming method.\\
\\
A probability measure $\mu$ on $[-1,1]$ that gives polynomials $p_k^{\frac{n-3}{2},\frac{n-3}{2}}$ that are positive semi-definite on $S^{n-1}$ is the Jacobi measure
\begin{equation}\label{Jacmeasure}\mu(dt)=\frac{(1-t^2)^{\frac{n-3}{2}}dt}{\int_{-1}^1(1-t^2)^{\frac{n-3}{2}}dt},\end{equation}
obtained from a projection of the uniform measure on $S^{n-1}$ onto an axis. $p_k^{\frac{n-3}{2},\frac{n-3}{2}}$ in this case are called Gegenbauer polynomials, a special case of Jacobi polynomials. Throughout this paper, we implicitly apply the Delsarte linear programming method with respect to this measure, and refer to conditions i) and ii) in this case as the Delsarte linear programming conditions for $M(n,\theta)$.\\
\\
Levenshtein~\cite{Lev79}, continuing joint work with Kabatyanskii in~\cite{KL}, constructed certain polynomials $g_{n,\theta}$ that satisfy the Delsarte linear programming conditions for $M(n,\theta)$. See equation~\eqref{gnthetageneral} and surrounding discussion for the definition. We call these polynomials Levenshtein's optimal polynomials. An application of the Delsarte linear programming method for $M(n,\theta)$ to $g_{n,\theta}$ gives the bound
\begin{equation}\label{Levtrivial}M(n,\theta)\leq M_{\textup{Lev}}(n,\theta):=\frac{g_{n,\theta}(1)}{\frac{\int_{-1}^1g_{n,\theta}(t)(1-t^2)^{\frac{n-3}{2}}dt}{\int_{-1}^1(1-t^2)^{\frac{n-3}{2}}}dt}.\end{equation}
We briefly describe the value of $M_{\textup{Lev}}(n,\theta)$ assuming properties of Jacobi polynomials to be recalled in Section~\ref{Jacobiestimates}. Explicitly, its value is determined as follows. Throughout this paper, let $t_{1,d}^{\alpha,\beta}>\hdots>t_{d,d}^{\alpha,\beta}$ be the roots of the Jacobi polynomial $p_d^{\alpha,\beta}$, where $p_d^{\alpha,\beta}$, as the degree $d$ varies, are polynomials orthogonal with respect to the inner product
\begin{equation*}\label{L2ab}
\left<f,g\right>_{L^2(\mu_{\alpha,\beta})}:=\int_{-1}^1f(t)g(t)\mu_{\alpha,\beta}(dt)
\end{equation*}
on $\mathbb{R}[t]$ with $\alpha,\beta>-1$ and
\begin{equation*}
\label{Jacmeasureab}\mu_{\alpha,\beta}(dt):=\frac{(1-t)^{\alpha}(1+t)^{\beta}dt}{\int_{-1}^1(1-t)^{\alpha}(1+t)^{\beta}dt}.
\end{equation*}
As is well-known from the theory of orthogonal polynomials, the roots of $p_d^{\alpha,\beta}$ are all real and simple. Let $\alpha=\frac{n-3}{2}$, and let $d$ be uniquely determined by $t_{1,d-1}^{\alpha+1,\alpha+1}<\cos(\theta)\leq t_{1,d}^{\alpha+1,\alpha+1}$. By Lemmas 5.29 and 5.30 of~\cite{Lev98}, $t_{1,d-1}^{\alpha+1,\alpha+1}<t_{1,d}^{\alpha+1,\alpha}<t_{1,d}^{\alpha+1,\alpha+1}$. Then
\begin{equation}\label{Levdef}M_{\textup{Lev}}(n,\theta)=\begin{cases}2\binom{d+n-1}{n-1}&\textup{if }t_{1,d}^{\alpha+1,\alpha}<\cos(\theta)\leq t_{1,d}^{\alpha+1,\alpha+1}\\ \binom{d+n-1}{n-1}+\binom{d+n-2}{n-1}&\textup{if }t_{1,d-1}^{\alpha+1,\alpha+1}<\cos(\theta)\leq t_{1,d}^{\alpha+1,\alpha}.\end{cases}\end{equation}
Sidelnikov's inequality~\cite{Sidelnikov} states that
\begin{equation}\label{Sidelnikovineq}
M(n,\theta)\leq \frac{M(n+1,\theta')}{\mu_n(\theta,\theta')}\textup{ when }\theta<\theta',
\end{equation}
where
\begin{equation}\label{massformula}\mu_n(\theta,\theta')=\frac{\int_{\sqrt{\frac{\cos(\theta)-\cos(\theta')}{1-\cos(\theta')}}}^1(1-t^2)^{\frac{n-3}{2}}dt}{\int_{-1}^1(1-t^2)^{\frac{n-3}{2}}dt}\end{equation}
is the mass of the spherical cap of radius $\frac{\sin(\theta/2)}{\sin(\theta'/2)}$ in the unit sphere $S^{n-1}$. By combining $M(n+1,\theta')\leq M_{\textup{Lev}}(n+1,\theta')$ with inequality~\eqref{Sidelnikovineq}, Kabatyanskii--Levenshtein~\cite{KL} and further work of Levenshtein~\cite{Lev79} showed that for $0<\theta\leq\theta^*=62.997...^{\circ}$,
\begin{equation}\label{KLcodebound}M(n,\theta)\leq \frac{M_{\textup{Lev}}(n+1,\theta^*)}{\mu_n(\theta,\theta^*)}.\end{equation}
Here, $\theta^*$ is the unique point in the interval $(0,\frac{\pi}{2})$ satisfying
\begin{equation*}
\cos(\theta^*)\log\left(\frac{1+\sin(\theta^*)}{1-\sin(\theta^*)}\right)-(1+\cos(\theta^*))\sin(\theta^*)=0.
\end{equation*}
Inequality~\eqref{KLcodebound} is an upper bound that is exponentially better than inequality~\eqref{Levtrivial}; see Corollary 3.9 of~\cite{SZ}. In~\cite{SZ}, the author and Sardari proved that a stronger version of inequality~\eqref{KLcodebound} could be obtained by an application of the Delsarte linear programming method to a single function that we constructed, thus incorporating the geometric idea of comparing $\theta$-spherical codes to $\theta^*$-spherical codes into a single function. In fact, we proved the stronger bound 
\begin{equation}\label{SZbound2}M(n,\theta)\leq 0.4325\cdot \frac{M_{\textup{Lev}}(n-1,\theta^*)}{\mu_n(\theta,\theta^*)}\end{equation}
for $0<\theta<\theta^*$ and sufficiently large $n$, the first such linear improvement. In this paper, we also prove a stronger version of this inequality. We are now ready to explain the relation between spherical codes and packings and state our main results.

\subsection{Relations and Main results}\label{mainresults}
Spherical codes and sphere packing densities are related to each other by the inequality
\begin{equation*}\label{Sid}
\delta_n\leq\sin^n\left(\frac{\theta}{2}\right)M(n+1,\theta)
\end{equation*}
for any $0<\theta\leq\pi$~\cite{Sidelnikov}. There are other such inequalities. For example, for $\frac{\pi}{3}\leq\theta\leq\pi$, Cohn--Zhao~\cite{CZ} proved using elementary geometry that
\begin{equation}\label{CohnZhao}\delta_n\leq\sin^n\left(\frac{\theta}{2}\right)M(n,\theta).\end{equation}
We define for $\frac{\pi}{3}\leq\theta\leq\pi$
\begin{equation}\label{dKL}\delta_n^{\textup{KL}}(\theta):=\sin^n\left(\frac{\theta}{2}\right)M_{\textup{Lev}}(n,\theta),\end{equation}
which is an upper bound to $\delta_n$ by inequalities~\eqref{Levtrivial} and~\eqref{CohnZhao}. We write 
\begin{equation}\label{dKLstar}\delta_n^{\textup{KL}}:=\delta_n^{\textup{KL}}(\theta^*)\end{equation}
which is the minimum value of $\delta_n^{\textup{KL}}(\theta)$ as $\theta$ varies between $(0,\frac{\pi}{2})$. This inequality $\delta_n\leq  \delta_n^{\textup{KL}}$ of Cohn--Zhao gave an improvement by a factor of $0.7915...$ on (geometric) \textit{average} to bounds preceding it. Note however, that the bounds in~\cite{CZ} were obtained by combining inequality~\eqref{CohnZhao} with the polynomials of~\cite{KL}, giving bounds weaker than those of~\cite{Lev79} on $M(n,\theta)$. Due to the seminal nature of the work of Kabatyanskii--Levenshtein~\cite{KL}, we use the notation above. All improvements to sphere packing densities discussed in this paper are relative to this bound.\\
\\
In addition to inequality~\eqref{CohnZhao}, Cohn and Zhao observed that any bound on $\delta_n$ obtained by combining inequality~\eqref{CohnZhao} with an upper bound on $M(n,\theta)$ from the Delsarte linear programming method applied to a function $g$ may be obtained by applying the Cohn--Elkies linear programming method to a function $H:\mathbb{R}^n\rightarrow\mathbb{R}$ constructed from $g$ and satisfying the Cohn--Elkies linear programming conditions. For sphere packing densities, we have from the argument of Cohn--Zhao the following proposition which is a minor generalization of what they proved.
\begin{proposition}[Prop. 3.8 of~\cite{SZ}]\label{38}Fix an angle $0<\theta\leq\pi$. Let $g:[-1,1]\rightarrow\mathbb{R}$ be a continuous function satisfying the Delsarte linear programming conditions for $M(n,\theta)$, and suppose $F:[-1,1]\rightarrow\mathbb{R}$ is an integrable function such that 
\begin{equation*}
H(T):=\int_{\mathbb{R}^n}F(|\boldsymbol{x}-\boldsymbol{z}|)F(|\boldsymbol{y}-\boldsymbol{z}|)g\left(\left<\frac{\boldsymbol{x}-\boldsymbol{z}}{|\boldsymbol{x}-\boldsymbol{z}|},\frac{\boldsymbol{y}-\boldsymbol{z}}{|\boldsymbol{y}-\boldsymbol{z}|}\right>\right)d\boldsymbol{z}
\end{equation*}
for $T=|\boldsymbol{x}-\boldsymbol{y}|$ ($H$ is a point-pair invariant function) satisfies $H(T)\leq 0$ for every $|T|\geq 1$. Here, we think of $H$ as a radial function on $\mathbb{R}^n$ with $H(\boldsymbol{x})=H(|\boldsymbol{x}|)$. Then
\begin{equation}\label{delce}
\delta_n\leq \frac{\vol(B_n(\boldsymbol{0},1))\|F\|_{L^2(\mathbb{R}^n)}^2}{2^n\|F\|_{L^1(\mathbb{R}^n)}^2}\cdot \frac{g(1)\int_{-1}^1(1-t^2)^{\frac{n-3}{2}}dt}{\int_{-1}^1g(t)(1-t^2)^{\frac{n-3}{2}}dt}.
\end{equation}
\end{proposition}
In the above, $H$ is assumed to satisfy the Cohn--Elkies linear programming conditions for $\delta_n$, and inequality~\eqref{delce} is obtained by applying Theorem~\ref{CohnElkiesthm} to $H$. In~\cite{CZ}, $F$ was taken to be the characteristic function $\chi_{[0,r]}$ of the interval $[0,r]$ with $r=\frac{1}{2\sin(\theta/2)}$ and $\frac{\pi}{3}\leq\theta\leq\pi$. In this situation, the geometric argument proving inequality~\eqref{CohnZhao} also implies that the integrand of $H(T)$ is everywhere non-positive, and so the Cohn--Elkies linear programming conditions are satisfied (positive semi-definiteness of $H$ on $\mathbb{R}^n$ is a simple consequence of the definition of $H$). Furthermore, setting $g=g_{n,\theta}$, one obtains from Proposition~\ref{38} the inequality $\delta_n\leq\frac{M_{\textup{Lev}}(n,\theta)}{(2r)^n}=\delta_n^{\textup{KL}}(\theta)$ which is optimal when $\theta=\theta^*$.\\
\\
The main insight we had in~\cite{SZ} was that in contrast to~\cite{CZ}, we could choose $F=\chi_{[0,r+\delta]}$ for some $\delta>0$ for which the integrand of $H(T)$ is no longer everywhere non-positive but $H$ still satisfies the Cohn--Elkies linear programming conditions. Determining large $\delta$ for which this is true was the main difficulty in proving inequality~\eqref{SZbound} stating that
\[\delta_n\leq 0.4325\cdot\delta_n^{\textup{KL}}\]
for sufficiently large $n$. However, the numerics of Section 6 of~\cite{SZ} suggested the following stronger result that we prove in this paper.
\begin{THMA}\manuallabel{THMone}{A}Given $\frac{\pi}{3}\leq \theta\leq\frac{\pi}{2}$,
\begin{equation}\label{constantimprovement}\delta_n\leq \frac{1+o(1)}{e}\cdot\delta^{\textup{KL}}_n(\theta)\textup{ as }n\rightarrow\infty.\end{equation}
Furthermore, $\frac{1}{e}$ is the optimal constant obtainable from Proposition~\ref{38} applied to $g=g_{n,\theta}$ with $F$ a characteristic function $\chi_{[0,r+\delta]}$. This improved $\delta_n\leq \delta_n^{\textup{KL}}$ by a factor of $0.2304...$ on geometric average.
\end{THMA}
This captures the two conjectures suggested by the numerics of~\cite{SZ}: a) the improvement factor is independent of the angle $\theta$ and b) $\frac{1}{e}$ is optimal if the analytic losses of~\cite{SZ} are eliminated. In this paper, we also prove the following strengthening of inequality~\eqref{SZbound2}. 
\begin{THMB}\manuallabel{THMtwo}{B}For every $0<\theta<\theta'\leq\frac{\pi}{2}$,
\begin{equation*}M(n,\theta)\leq \frac{1+o(1)}{e}\cdot\frac{M_{\textup{Lev}}(n-1,\theta')}{\mu_n(\theta,\theta')}\textup{ as }n\rightarrow\infty,\end{equation*}
where $\mu_n(\theta,\theta')$ is the mass of the spherical cap of radius $\frac{\sin(\theta/2)}{\sin(\theta'/2)}$ in the unit sphere $S^{n-1}$ (see~\eqref{massformula}). Furthermore, $\frac{1}{e}$ is optimal for the function constructed in~\cite{SZ}.\end{THMB}
Compare Theorem~\ref{THMtwo} to Theorem 1.1 of~\cite{SZ} where the improvement constant depended on $\theta'$ and the comparison was made to angle $\theta^*$.\\
\\
Theorem~\ref{THMone} is not simply a combination of Theorem~\ref{THMtwo} with inequality~\eqref{CohnZhao}. Though the proof of Theorem~\ref{THMone} uses Proposition~\ref{38}, the proof of Theorem~\ref{THMtwo} uses its spherical codes variant based on the Delsarte linear programming method. Indeed, we will use Proposition 3.6 of~\cite{SZ}. In the following, for each $\boldsymbol{z}\in S^{n-1}$, $\Pi_{\boldsymbol{z}}:S^{n-1}\setminus\{\pm\boldsymbol{z}\}\rightarrow S^{n-1}$ is projection onto the orthogonal complement of $\boldsymbol{z}$ followed by normalization to a unit vector, almost everywhere defined on $S^{n-1}$.
\begin{proposition}[Prop. 3.6 of~\cite{SZ}]\label{36} Suppose $0<\theta<\theta'\leq\frac{\pi}{2}$ are fixed angles and $g$ satisfies the Delsarte linear programming conditions for $M(n-1,\theta')$, and suppose $F:[-1,1]\rightarrow\mathbb{R}$ is a function such that
\begin{equation*}
H(t):=\int_{S^{n-1}}F(\left<\boldsymbol{x},\boldsymbol{z}\right>)F(\left<\boldsymbol{y},\boldsymbol{z}\right>)g\left(\left<\Pi_{\boldsymbol{z}}(\boldsymbol{x}),\Pi_{\boldsymbol{z}}(\boldsymbol{y})\right>\right)d\boldsymbol{z}
\end{equation*}
with $t=\left<\boldsymbol{x},\boldsymbol{y}\right>$ satisfies $H(t)\leq 0$ for every $t\in[-1,\cos\theta]$. Then
\begin{equation*}
M(n,\theta)\leq \frac{g(1)\int_{-1}^1(1-t^2)^{\frac{n-4}{2}}dt}{\int_{-1}^1g(t)(1-t^2)^{\frac{n-4}{2}}dt}\cdot\frac{\int_{-1}^1F(t)^2(1-t^2)^{\frac{n-3}{2}}dt\int_{-1}^1(1-t^2)^{\frac{n-3}{2}}dt}{\left(\int_{-1}^1F(t)(1-t^2)^{\frac{n-3}{2}}dt\right)^2}.
\end{equation*} 
\end{proposition}
We explain the geometry behind the proofs of Theorems~\ref{THMone} and~\ref{THMtwo} in Section~\ref{general}, and sketch the main ingredients in their proofs. We will explain the method in greater detail in Subsections~\ref{m1} and~\ref{m1b}. Aside from these propositions, we will use refined estimates of Jacobi polynomials to obtain sharp estimates of Levenshtein's optimal polynomials near their largest roots. See the next subsection for more on this.
\subsection{General results, proofs, and connection to Stiefel manifolds}\label{intro:general}
Given angles $0<\theta<\theta'\leq\frac{\pi}{2}$, Proposition~\ref{36} takes a function $g$ satisfying the Delsarte linear programming conditions for $M(n-1,\theta')$ and produces a function $H$ satisfying the Delsarte linear programming conditions for $M(n,\theta)$. As suggested by the definition of $H$, the integrand was produced by projecting a spherical code onto the orthogonal complement of vectors $\boldsymbol{z}\in S^{n-1}$, and then averaging over suitable $\boldsymbol{z}$. Similarly, Proposition~\ref{38} took a function $g$ satifying the Delsarte linear programming conditions for $M(n,\theta)$ and produced a function $H$ satisfying the Cohn--Elkies linear programming conditions for $\mathbb{R}^n$. A natural question that arises is the following: what happens if instead of projecting onto the orthogonal complement of $1$-dimensional subspaces, we project onto the orthogonal complement of a subspace spanned by orthonormal $\boldsymbol{z}_1,\hdots,\boldsymbol{z}_m\in S^{n-1}$ and then average over suitable such ordered families of $m$ orthonormal vectors? This corresponds to averaging over a subset of the Stiefel manifold $V_{m}(\mathbb{R}^n)$. There are three questions that we answer regarding this generalization.
\vspace{5mm}
\begin{enumerate}
\item How do we produce from a function $g$ satisfying the Delsarte linear programming conditions for $M(n-m,\theta')$ a function $H$ potentially satisfying the Delsarte linear programming conditions for $M(n,\theta)$?
\item Given the general construction, how do we verify that the function $H$ satisfies the Delsarte linear programming conditions, that is, what are the required estimates and calculations?
\item What upper bounds on spherical codes are obtainable if we apply such a general procedure to Levenshtein's optimal polynomials and suitably chosen regions of integration? How do the improvement factors behave when $m\geq 2$?
\end{enumerate}
\vspace{5mm}
The general construction is given in Subsection~\ref{mhigher}. The method of~\cite{SZ} corresponds to averaging over $V_1(\mathbb{R}^n)=S^{n-1}$. Regarding the second question, note that $V_m(\mathbb{R}^n)$ has large dimension $nm-\frac{m(m+1)}{2}$. However, we will apply changes of coordinates to reduce the dimensions of the domains of integration of our integrals, in fact to integrals over subsets of $\mathbb{B}^m\times\mathbb{B}^m$, where $\mathbb{B}^m$ is the closed unit ball in $\mathbb{R}^m$ centered at $\boldsymbol{0}$. After proving general results about our functions, we investigate certain constructions obtained by averaging special functions over $V_m(\mathbb{R}^n)$ for $m$ and $n$ related to each other in several ways. This part is an investigation of the connection between upper bounds for sizes of spherical codes and geometric constructions involving multiple orthogonal projections. These general results and ideas will be applied in future work to other classes of functions.\\
\\
Aside from the new class of functions that we construct, an important ingredient in our proofs is a different approach to the estimation of Jacobi polynomials compared to~\cite{SZ} to which we alluded in the previous subsection. Our estimate relies on Krasikov's estimate of ultraspherical polynomials obtained from the WKB method~\cite{IK}. We use an estimate of Krasikov~\cite{IK} for $(1-x^2)^{\frac{n-1}{2}}(p_d^{\frac{n-3}{2},\frac{n-3}{2}}(x))^2$ in a sub-interval $[-\sqrt{1-q},\sqrt{1-q}]$ of $[-1,1]$ including all its zeros. In this sub-interval, this function is nearly equioscillatory. On $[t_{1,d}^{\frac{n-3}{2},\frac{n-3}{2}},\sqrt{1-q}]$, we will show that this function is concave, implying that it lies below the tangent line at $t_{1,d}^{\frac{n-3}{2},\frac{n-3}{2}}$. Furthermore, we prove that the function is decreasing to the right of $\sqrt{1-q}$. Therefore, to the right of $t_{1,d}^{\frac{n-3}{2},\frac{n-3}{2}}$, we bound $(1-x^2)^{\frac{n-1}{2}}(p_d^{\frac{n-3}{2},\frac{n-3}{2}}(x))^2$ from above using its tangent line at $t_{1,d}^{\frac{n-3}{2},\frac{n-3}{2}}$. Furthermore, we also use Krasikov's result to estimate $(1-x^2)^{\frac{n-1}{2}}(p_d^{\frac{n-3}{2},\frac{n-3}{2}}(x))^2$ on an interval of the form $[t_{1,d}^{\frac{n-3}{2},\frac{n-3}{2}}-o(n^{-2/3}),t_{1,d}^{\frac{n-3}{2},\frac{n-3}{2}}]$. These estimates allow us to not only refine our estimates of Levenshtein's optimal polynomials, but also circumvent more complicated versions of the conditional density estimates of~\cite{SZ} in our proofs. In this paper, for example, $\theta^*$ does not play a role in the proofs of Theorems~\ref{THMone} and~\ref{THMtwo}. These lead to uniform improvement constants that are sharp when Levenshtein's optimal polynomials are used. Though the full strength of the estimates that we prove are not necessary for the results for this paper, we include them for future work along these lines. That being said, the estimates of Jacobi polynomials proved in~\cite{SZ} are insufficient for the results of this paper, including for Theorems~\ref{THMone} and~\ref{THMtwo}.\\
\\
That we cannot do better than $\frac{1}{e}$ if using Propositions~\ref{38} and~\ref{36} applied to Levenshtein's optimal polynomials will be a consequence of upper bounds on the sup-norms of Jacobi polynomials of the form found in~\cite{EMN} or~\cite{IK2}. See Lemma~\ref{Krasikovupperbound} below.  
\subsection{Outline of paper}
In Section~\ref{general}, we recall the constructions of~\cite{SZ}, and then proceed to give the construction of the general functions of this paper constructed using Stiefel manifolds. We then roughly explain the proofs of Theorems~\ref{THMone} and~\ref{THMtwo}. In Section~\ref{changeofvar}, we primarily rewrite the general function~\eqref{hStiefel} constructed in Subsection~\ref{mhigher} by introducing a convenient coordinate system that allows us to reduce our integrals to integrals over (subsets of) $\mathbb{B}^m\times\mathbb{B}^m$, where $\mathbb{B}^m$ is the closed unit $m$-dimensional ball in $\mathbb{R}^m$. Along the way, we prove results relevant to calculating the bounds obtained on $M(n,\theta)$ assuming that the general function $H$ satisfies the Delsarte linear programming conditions. We also quickly recall estimates and changes of coordinates from~\cite{SZ}. In Section~\ref{Jacobiestimates}, we recall the precise definition of Jacobi polynomials, and some of their properties useful to us. In particular, we state Krasikov's estimates of Jacobi polynomials, using which we prove sharp estimates of Levenshtein's optimal polynomials near their maximal roots after defining them. These estimates on Levenshtein's optimal polynomials are crucial to our proofs. In Section~\ref{section:m1b}, we prove Theorem~\ref{THMone}, while in Section~\ref{section:m1} we give two proofs of Theorem~\ref{THMtwo}, one using conditional density estimates and one avoiding them. In Section~\ref{section:m2}, we determine the improvement factors from $V_2(\mathbb{R}^n)$, that is, $m=2$. In Section~\ref{section:mhigher}, we analyze bounds obtained when $m=3$, $m=4$, and when $m$ is large but $m=o(n)$. Finally, in Section~\ref{section:constant}, we analyze the case where $n-m$ is constant and $n$ is large. Though the author has analyzed other situations, they are not included in this paper for brevity.
\section{General functions and overview}\label{general}
In Subsection~\ref{m1}, we recall the construction of~\cite{SZ} for spherical codes and explain the main ingredients in this paper that prove the improvement constant $\frac{1}{e}$ when applying the methods of~\cite{SZ}. We also explain why $\frac{1}{e}$ is the optimal constant in that setup if using Levenshtein's optimal polynomials. In Subsection~\ref{m1b}, we do the same for sphere packings. In the final Subsection~\ref{mhigher}, we explain the construction of functions obtained by averaging over Stiefel manifolds.
\subsection{Averaging over $S^{n-1}$ for spherical codes}\label{m1}
In~\cite{SZ}, the author and Sardari constructed general functions for spherical codes that we describe in this subsection. This simplified version clarifies the general construction using Stiefel manifolds given in Subsection~\ref{mhigher}.\\
\\
The construction was as follows. Suppose we are given angles $0<\theta<\theta'\leq\frac{\pi}{2}$, a point $\boldsymbol{z}\in S^{n-1}$ in the unit sphere, and integrable functions $F:[-1,1]\rightarrow\mathbb{R}$ and $g:[-1,1]\rightarrow\mathbb{R}$. Associated to such a data, we consider the function
\[h(\cdot,\cdot;\boldsymbol{z}):(S^{n-1}\setminus\{\pm\boldsymbol{z}\})\times (S^{n-1}\setminus\{\pm\boldsymbol{z}\})\rightarrow\mathbb{R}\]
given by
\begin{equation*}\label{preh}
h(\boldsymbol{x},\boldsymbol{y};\boldsymbol{z}):=F\left(\left<\boldsymbol{x},\boldsymbol{z}\right>\right)F\left(\left<\boldsymbol{y},\boldsymbol{z}\right>\right)g\left(\left<\Pi_{\boldsymbol{z}}(\boldsymbol{x}),\Pi_{\boldsymbol{z}}(\boldsymbol{y})\right>\right),
\end{equation*}
where $\Pi_{\boldsymbol{z}}:S^{n-1}\setminus\{\pm\boldsymbol{z}\}\rightarrow S^{n-2}$ is the projection map onto the orthogonal complement of $\boldsymbol{z}$ in $\mathbb{R}^n$ followed by normalization to a unit vector. Up to a subset of measure zero, $h$ is defined almost everywhere on $S^{n-1}\times S^{n-1}$. Under the situations in which the function
\begin{equation}\label{h}H(\boldsymbol{x},\boldsymbol{y}):=\int_{S^{n-1}}h(\boldsymbol{x},\boldsymbol{y};\boldsymbol{z})d\boldsymbol{z}\end{equation}
obtained by averaging over $\boldsymbol{z}\in S^{n-1}$ with respect to the uniform measure satisfies the Delsarte linear programming conditions for $M(n,\theta)$, we obtain bounds on $M(n,\theta)$ from Proposition~\ref{36}. Note that $H$ is a point pair invariant function, that is, it is a function that depends only on $t:=\left<\boldsymbol{x},\boldsymbol{y}\right>$. It is easy to see that $H$ is positive semi-definite on $S^{n-1}$ as soon as $g$ is positive semi-definite on $S^{n-2}$. Assuming that we choose the data appropriately to have $H(t)\leq 0$ for every $t\in[-1,\cos\theta]$, we obtain from Proposition~\ref{36} the bound
\begin{equation*}
M(n,\theta)\leq \frac{g(1)\int_{-1}^1(1-t^2)^{\frac{n-4}{2}}dt}{\int_{-1}^1g(t)(1-t^2)^{\frac{n-4}{2}}dt}\cdot\frac{\int_{-1}^1F(t)^2(1-t^2)^{\frac{n-3}{2}}dt\int_{-1}^1(1-t^2)^{\frac{n-3}{2}}dt}{\left(\int_{-1}^1F(t)(1-t^2)^{\frac{n-3}{2}}dt\right)^2}.
\end{equation*}
In~\cite{SZ}, we proved inequality~\eqref{SZbound2} above by taking $g=g_{n-1,\theta'}$ to be Levenshtein's optimal polynomial bounding $M(n-1,\theta')\leq M_{\textup{Lev}}(n-1,\theta')$. Given angles $0<\theta<\theta'\leq\frac{\pi}{2}$, real numbers $r,R\in(0,1]$ such that $r<R$, $\delta\in(0,r)$, we took $F:[-1,1]\rightarrow\mathbb{R}_{\geq 0}$ to be $F=\chi_{[r-\delta,R]}$, the characteristic function on the interval $[r-\delta,R]$. We chose $r,R$ such that for $\left<\boldsymbol{x},\boldsymbol{z}\right>,\left<\boldsymbol{y},\boldsymbol{z}\right>\in [r,R]$, if $\left<\boldsymbol{x},\boldsymbol{y}\right>\leq\cos\theta$ then $\left<\Pi_{\boldsymbol{z}}(\boldsymbol{x}),\Pi_{\boldsymbol{z}}(\boldsymbol{y})\right>\leq\cos\theta'$. Geometrically, this means that for any two points $\boldsymbol{x},\boldsymbol{y}$ in the strip
\begin{equation}\label{def:strip}\textup{Str}_{\theta,\theta'}(\boldsymbol{z})=\left\{\boldsymbol{u}\in S^{n-1}:r\leq\left<\boldsymbol{u},\boldsymbol{z}\right>\leq R\right\}\end{equation}
that are of angular distance at least $\theta$, their projections $\Pi_{\boldsymbol{z}}(\boldsymbol{x}),\Pi_{\boldsymbol{z}}(\boldsymbol{y})\in S^{n-2}$ are of angular distance at least $\theta'$ from each other, leading to the inequality $g_{n-1,\theta'}\left(\left<\Pi_{\boldsymbol{z}}(\boldsymbol{x}),\Pi_{\boldsymbol{z}}(\boldsymbol{y})\right>\right)\leq 0$. In fact, we let
\begin{equation}\label{goodR}r:=\sqrt{\frac{\cos(\theta)-\cos(\theta')}{1-\cos(\theta')}}\end{equation}
and
\begin{equation}\label{bigRfirst}R:=\cos\left(2\arctan\left(\frac{\cos(\theta)}{\sqrt{(1-\cos(\theta))(\cos(\theta)-\cos(\theta'))}}\right)+\arccos(r)-\pi\right).\end{equation}
However, the idea of~\cite{SZ} giving our constant improvement was that we do not need pointwise non-positivity of the integrand of~\eqref{h} to ensure that $H(t)\leq 0$ for every $t\in[-1,\cos\theta]$. In fact, we chose $\delta>0$ such that we still had this non-positivity condition after integration. In~\cite{SZ}, we proved that one could choose $\delta=O(1/n)$, leading to the constant improvement of~\eqref{SZbound2} over previous bounds. Note that the volume of the strip~\eqref{def:strip} concentrates around the edge of larger radius $\sqrt{1-r^2}$. Therefore, we cannot expect $\delta$ to be larger than $O(1/n)$. An ingredient in finding a large $\delta$ was the estimation of Levenshtein's optimal polynomials $g_{n-1,\theta'}$ using a Taylor expansion at its largest root. We also used an analogue of the above for sphere packing densities based on the Cohn--Elkies linear programming method, giving us inequality~\eqref{SZbound} above. We describe this in Subsection~\ref{m1b} below.\\
\\
The estimate proved and used in Proposition 5.1 of~\cite{SZ} on Jacobi polynomials near their largest roots was crude in two ways. To the right of their largest zeros, in retrospect, a crude exponential upper bound was given. It was obtained by using the differential and interlacing properties of Jacobi polynomials. From the crude bound, the reverse triangle inequality was used to bound Levenshtein's optimal polynomials to the left of their largest zeros. To the left of the largest zero, the bound we obtained was only valid in a neighbourhood of length $O(1/n)$ around the largest root, with multiplicative constant absolutely determined. These lead to a restriction on the amount of negative contribution of the integrand of $H$ that we could use and overestimated the positive contributions of the integrand of~\eqref{h}, thus limiting how large of a $\delta$ we could choose. However, in this paper, we do not use the same estimate on Levenshtein's optimal polynomials. As explained in Subsection~\ref{mhigher}, we use a result of Krasikov (see Lemma~\ref{Krasikovlemma}) giving us a very good estimate of Jacobi polynomials on $[0,\sqrt{1-q}]$ where $q$ is such that $\sqrt{1-q}$ is $O(n^{-2/3})$ to the right of their largest roots, which is more than we need. This estimate gives the best possible constant improvement for functions obtained by averaging Levenshtein's optimal polynomials as above. In addition to its statement, the proof of Lemma~\ref{Krasikovlemma} given by Krasikov is also important to us in proving our estimates in Section~\ref{Jacobiestimates}.\\
\\
In addition to leading to stronger results, there are two additional features of the estimates proved and used in this paper. The first is that they allow us to circumvent the complicated conditional density estimates proved in~\cite{SZ}. The second is that the proof given in this paper, in contrast to that given in~\cite{SZ}, does not use $\theta'=\theta^*$ and its value from the beginning. Instead, we work with general angles $0<\theta<\theta'\leq\frac{\pi}{2}$, and obtain a general improvement factor from our methods. It is only in the final stage that we specialize to $\theta'=\theta^*$ if $\theta<\theta^*$ for reasons explained in the introduction. As we will see, this also explains the apparent convergence of linear improvement factors to inequality~\eqref{CohnZhao} observed in the numerics of~\cite{SZ} to $\frac{1}{e}$ \textit{independent of the angle chosen}. Therefore, aside from leading to stronger bounds, the strengthened estimates in this paper lead to cleaner arguments that prove the numerical observations of~\cite{SZ}.\\
\\
In Subsection~\ref{mhigher}, we explain a generalization of this construction. Instead of averaging over $S^{n-1}$, we average over Stiefel manifolds $V_m(\mathbb{R}^n)$. In $m\geq 2$, in a sense to be made precise, instead of just thickening in a radial direction as in~\cite{SZ} (using $[r-\delta,R]$ instead of $[r,R]$), we also thicken in an angular direction, giving us more flexibility. We will see that in the situations that we will consider for $m\geq 2$, the improvement factors will no longer be independent of $\theta'$ and converge to $0$ as $\theta'$, and so also $\theta$, converges to $0$.
\subsection{Averaging over $\mathbb{R}^n$ for sphere packings}\label{m1b}
In~\cite{CZ}, Cohn and Zhao proved their inequality~\eqref{CohnZhao}. The linear programming incarnation of that argument was given by starting with a function $g:[-1,1]\rightarrow\mathbb{R}$ satisfying the Delsarte linear programming conditions for $M(n,\theta)$ where $0<\theta\leq\pi$, and $F:[-1,1]\rightarrow\mathbb{R}$ an appropriately chosen characteristic function so that if $F(|\boldsymbol{x}-\boldsymbol{z}|)F(|\boldsymbol{y}-\boldsymbol{z}|)\neq 0$ for $\boldsymbol{x},\boldsymbol{y}\neq \boldsymbol{z}$ in $\mathbb{R}^n$ such that $|\boldsymbol{x}-\boldsymbol{y}|\geq 1$, then $\left<\frac{\boldsymbol{x}-\boldsymbol{z}}{|\boldsymbol{x}-\boldsymbol{z}|},\frac{\boldsymbol{y}-\boldsymbol{z}}{|\boldsymbol{y}-\boldsymbol{z}|}\right>\in[-1,\cos\theta]$ and so
\begin{equation*}
g\left(\left<\frac{\boldsymbol{x}-\boldsymbol{z}}{|\boldsymbol{x}-\boldsymbol{z}|},\frac{\boldsymbol{y}-\boldsymbol{z}}{|\boldsymbol{y}-\boldsymbol{z}|}\right>\right)\leq 0.
\end{equation*}
In fact, $F=\chi_{[0,\ell]}$ with $\ell:=\frac{1}{2\sin(\theta/2)}$ and $\frac{\pi}{3}\leq\theta\leq\pi$ was chosen. Then the function
\begin{equation*}
H(\boldsymbol{x},\boldsymbol{y}):=\int_{\mathbb{R}^n}F(|\boldsymbol{x}-\boldsymbol{z}|)F(|\boldsymbol{y}-\boldsymbol{z}|)g\left(\left<\frac{\boldsymbol{x}-\boldsymbol{z}}{|\boldsymbol{x}-\boldsymbol{z}|},\frac{\boldsymbol{y}-\boldsymbol{z}}{|\boldsymbol{y}-\boldsymbol{z}|}\right>\right)d\boldsymbol{z}
\end{equation*}
is a point-pair invariant function. The choice of function $F$ ensures that if $|\boldsymbol{x}-\boldsymbol{y}|\geq 1$, then the integrand of the integral is not positive. This is a consequence of the geometric argument given in proving inequality~\eqref{CohnZhao}. If we abuse notation and write $H(T)=H(\boldsymbol{x},\boldsymbol{y})$ if $|\boldsymbol{x}-\boldsymbol{y}|=T$, we then have that $H(T)\leq 0$ whenever $|T|\geq 1$. It is also easy to see that $g$ being positive semi-definite on $S^{n-1}$ implies that $H$ is positive semi-definite on $\mathbb{R}^n$. In that case, $H$ satisfies the Cohn--Elkies linear programming conditions of Theorem~\ref{CohnElkiesthm} for $\mathbb{R}^n$. An application of Theorem~\ref{CohnElkiesthm} then implies inequality 
\begin{equation*}\delta_n\leq\sin^n(\theta/2)\frac{g(1)\int_{-1}^1(1-t^2)^{\frac{n-3}{2}}dt}{\int_{-1}^1g(t)(1-t^2)^{\frac{n-3}{2}}dt},\end{equation*}
specializing to $\delta_n\leq\delta_n^{\textup{KL}}(\theta)$ when $g=g_{n,\theta}$.\\
\\
In~\cite{SZ}, similar to the spherical codes case, the author and Sardari observed that it is not necessary for the integrand to be everywhere non-positive when $|T|\geq 1$ if we require $H(T)\leq 0$ for $|T|\geq 1$. We could instead take $F=\chi_{[0,\ell+\delta]}$ for some $\delta=O(n^{-1})$ leading to some positive contributions to the integrand while still satisfying $H(T)\leq 0$ whenever $|T|\geq 1$. Obtaining a large $\delta$ was a delicate analytic problem that we dealt with. There, we took $g=g_{n,\theta}$ to be Levenshtein's optimal polynomial for $M(n,\theta)$. As in the case of spherical codes, we used an estimation of Jacobi polynomials to prove that
\begin{equation*}
\delta_n\leq 0.4325\cdot\delta_n^{\textup{KL}}
\end{equation*}
for sufficiently large $n$. However, by using our refined estimation of Levenshtein's optimal polynomials described in the previous subsection, we obtain Theorem~\ref{THMone}.
\subsection{Averaging over Stiefel manifolds for spherical codes}\label{mhigher}
A generalization of the construction of Subsection~\ref{m1} is the following. Instead of averaging over $\boldsymbol{z}\in S^{n-1}$, we may average over ordered families $\boldsymbol{z}_1,\hdots,\boldsymbol{z}_m$ of $m$ orthonormal vectors in $S^{n-1}$. This is achieved by integrating over the Stiefel manifold $V_m(\mathbb{R}^n)$. Our setup is as follows.\\
\\
Suppose $1\leq m\leq n$ is a natural number, and that we are given angles $0<\theta<\theta'\leq\frac{\pi}{2}$, \textit{orthonormal} vectors $\boldsymbol{z}_1,\hdots,\boldsymbol{z}_m\in S^{n-1}$, an integrable function $F:[-1,1]^{m}\rightarrow\mathbb{R}$, and an integrable function $g:[-1,1]\rightarrow\mathbb{R}$. Associated to such a data, we consider the function
\[h(\cdot,\cdot;\boldsymbol{z}_1,\hdots,\boldsymbol{z}_m):(S^{n-1}\setminus\{\pm\boldsymbol{z}_1,\hdots,\pm\boldsymbol{z}_m\})\times (S^{n-1}\setminus\{\pm\boldsymbol{z}_1,\hdots,\pm\boldsymbol{z}_m\})\rightarrow\mathbb{R}\]
given by
\begin{equation*}
h(\boldsymbol{x},\boldsymbol{y};\boldsymbol{z}_1,\hdots,\boldsymbol{z}_m):=F\left(\left<\boldsymbol{x},\boldsymbol{z}_1\right>,\hdots,\left<\boldsymbol{x},\boldsymbol{z}_m\right>\right)F\left(\left<\boldsymbol{y},\boldsymbol{z}_1\right>,\hdots,\left<\boldsymbol{y},\boldsymbol{z}_m\right>\right)g\left(\left<\Pi_{m}\hdots\Pi_1(\boldsymbol{x}),\Pi_{m}\hdots\Pi_1(\boldsymbol{y})\right>\right),
\end{equation*}
where for each $1\leq j\leq m$, $\Pi_j:S^{n-1}\setminus\{\pm\boldsymbol{z}_j\}\rightarrow S^{n-1}$ is the projection operator onto the orthogonal complement of $\boldsymbol{z}_j$ in $\mathbb{R}^n$ followed by normalization to a unit vector. $h$ is defined almost everywhere for almost every choice of $\boldsymbol{z}_1,\hdots,\boldsymbol{z}_m$. If the function
\begin{equation}\label{hStiefel}H(\boldsymbol{x},\boldsymbol{y}):=\int_{V_m(\mathbb{R}^n)}h(\boldsymbol{x},\boldsymbol{y};\boldsymbol{z}_1,\hdots,\boldsymbol{z}_m)d\boldsymbol{z}_1\cdots d\boldsymbol{z}_m\end{equation}
obtained by averaging over the Stiefel manifold $V_m(\mathbb{R}^n)$ of ordered $m$-tuples $(\boldsymbol{z}_1,\hdots,\boldsymbol{z}_m)$ of orthonormal vectors in $\mathbb{R}^n$ satisfies the Delsarte linear programming conditions for $M(n,\theta)$, we obtain an upper bound on $M(n,\theta)$. Here, we are thinking of $V_m(\mathbb{R}^n)$ as $S^{n-1}\times S^{n-2}\times\hdots\times S^{n-m}$ with measure given by the product of the uniform spherical measures on each of the factors. Note the following lemma that we are implicitly using in this claim.
\begin{lemma}\label{pointpair}$H(\boldsymbol{x},\boldsymbol{y})$ is a point pair invariant function, that is, it depends only on $\left<\boldsymbol{x},\boldsymbol{y}\right>$.
\end{lemma}
Therefore, $H$ may be thought of as a function of $t:=\left<\boldsymbol{x},\boldsymbol{y}\right>$. We abuse notation and write $H(t)=H(\boldsymbol{x},\boldsymbol{y})$. Observe that $m=1$ corresponds to the construction of~\cite{SZ} described in the previous subsection.\\
\\
Two natural questions that arise are the following. 1) Under what conditions do we have that the function $H$ satisfies the Delsarte linear programming conditions? 2) What specific choices coming from geometry lead to improved upper bounds on sizes of spherical codes and sphere packing densities? For positive semi-definiteness, we have the following lemma.
\begin{lemma}\label{positivity}Suppose $F:[-1,1]^m\rightarrow\mathbb{R}$ is integrable, we have $m$ orthonormal $\boldsymbol{z}_1,\hdots,\boldsymbol{z}_m\in S^{n-1}$, and $g$ is positive semi-definite for points in $S^{n-m-1}$. Then $h$ is a positive semi-definite function on $S^{n-1}$ almost everywhere, that is, for any finite subset $A:=\{\boldsymbol{x}_1,\hdots,\boldsymbol{x}_N\}\subset S^{n-1}\setminus\{\pm\boldsymbol{z}_1,\hdots,\pm\boldsymbol{z}_m\}$ and any $c_1,\hdots,c_N\in\mathbb{R}$, we have
\begin{equation*}\label{hposdef}
\sum_{i,j}c_ic_jh(\boldsymbol{x}_i,\boldsymbol{x}_j;\boldsymbol{z}_1,\hdots,\boldsymbol{z}_m)\geq 0.
\end{equation*}
In particular, $H$ is positive semi-definite on $S^{n-1}$.
\end{lemma}
\begin{proof}From the definition of $h$, we have
\[\sum_{i,j}c_ic_jh(\boldsymbol{x}_i,\boldsymbol{x}_j;\boldsymbol{z}_1,\hdots,\boldsymbol{z}_m)=\sum_{i,j}c_ia_ic_ja_jg\left(\left<\Pi_{m}\hdots\Pi_1(\boldsymbol{x}_i),\Pi_{m}\hdots\Pi_1(\boldsymbol{x}_j)\right>\right),\]
where
\[a_i:=F\left(\left<\boldsymbol{x}_i,\boldsymbol{z}_1\right>,\hdots,\left<\boldsymbol{x}_i,\boldsymbol{z}_m\right>\right).\]
The inequality follows from the fact that for each $i$, $\Pi_{m}\hdots\Pi_1(\boldsymbol{x}_i)$ may be thought of as a point in $S^{n-m-1}$, and $g$ is positive semi-definite on $S^{n-m-1}$.
\end{proof}
Schoenberg's theorem~\cite{Sch42} on positive semi-definiteness of functions on spheres then leads to the following corollary.
\begin{corollary}Suppose $F:[-1,1]^m\rightarrow\mathbb{R}$ is integrable, we have $m$ orthonormal $\boldsymbol{z}_1,\hdots,\boldsymbol{z}_m\in S^{n-1}$, and $g$ is positive semi-definite on $S^{n-m-1}$. Then $H$ satisfies for every $k\geq 0$,
\begin{equation*}
\int_{-1}^1H(t)p_k^{\frac{n-3}{2},\frac{n-3}{2}}(t)(1-t^2)^{\frac{n-3}{2}}dt\geq 0.
\end{equation*}
In particular, as soon as $H_0:=\frac{\int_{-1}^1H(t)(1-t^2)^{\frac{n-3}{2}}dt}{\int_{-1}^1(1-t^2)^{\frac{n-3}{2}}dt}>0$, the second Delsarte linear programming condition is satisfied with respect to the Jacobi measure~\eqref{Jacmeasure}.
\end{corollary}
As a result of this corollary, the difficulty is in providing interesting sufficient conditions for the first Delsarte linear programming condition, that is, $H(t)\leq 0$ for every $t\in[-1,\cos\theta]$.\\
\\
Aside from proving general results for general functions of the form~\eqref{hStiefel}, we will work in a specialization of the above general construction. We will take $m\geq 2$ and $g=g_{n-m,\theta'}$ to be Levenshtein's optimal polynomial giving the bound $M(n-m,\theta')\leq M_{\textup{Lev}}(n-m,\theta')$. Throughout this paper, we let
\begin{equation*}\boldsymbol{p}:=(0,\hdots,0,1)\in S^{n-1}.\end{equation*}
Given $0<\theta<\theta'\leq \frac{\pi}{2}$, let $r$ and $R$ be as in~\eqref{goodR} and~\eqref{bigRfirst}, respectively, as in Subsection~\ref{m1}. We also choose $\eta,\delta$, and take $F=\chi_{\cal{C}^{\theta,\theta'}_{\delta,\eta,m}}:[-1,1]^m\rightarrow\mathbb{R}_{\geq 0}$ to be the characteristic function of the region
\begin{equation*}\label{conicalregion}
\cal{C}^{\theta,\theta'}_{\delta,\eta,m}:=\left\{\boldsymbol{u}\in\mathbb{B}^m: r-\delta\leq|\boldsymbol{u}|\leq R\textup{ and }\left<\boldsymbol{u},\boldsymbol{p}\right>\geq |\boldsymbol{u}|\cos\left(\eta\right)\right\},
\end{equation*}
a subset of the $m$-dimensional closed unit ball $\mathbb{B}^m$ centered at $\boldsymbol{0}$. This choice of data is partly motivated by the fact that with $r,R$ as in~\eqref{goodR} and~\eqref{bigRfirst}, and $\delta=\eta=0$, if
\begin{equation}\label{Fneq0}F\left(\left<\boldsymbol{x},\boldsymbol{z}_1\right>,\hdots,\left<\boldsymbol{x},\boldsymbol{z}_m\right>\right)F\left(\left<\boldsymbol{y},\boldsymbol{z}_1\right>,\hdots,\left<\boldsymbol{y},\boldsymbol{z}_m\right>\right)\neq 0\end{equation}
and
\begin{equation}\label{Fneq1}\left<\boldsymbol{x},\boldsymbol{y}\right>\leq\cos\theta,\end{equation}
then
\begin{equation*}\label{Fneq0consequence}\left<\Pi_{m}\hdots\Pi_1(\boldsymbol{x}),\Pi_{m}\hdots\Pi_1(\boldsymbol{y})\right>\leq\cos\theta'\end{equation*}
and so $g\left(\left<\Pi_{m}\hdots\Pi_1(\boldsymbol{x}),\Pi_{m}\hdots\Pi_1(\boldsymbol{y})\right>\right)\leq 0$, implying that the integrand of $H$ is not positive for the above $g,F$ when~\eqref{Fneq0} and~\eqref{Fneq1} are true and $\delta=\eta=0$. We will prove in the next section that this is true by performing a change of coordinates. Therefore, by allowing $\delta,\eta$ to vary and $m\geq 2$, in addition to a radial thickening or shrinking of a good subset measured by $\delta$, we are also allowing an angular thickening determined by $\eta$. In the $m=1$ case, there is no room for angular thickenings. In the case $m\geq 2$, it is not true that the larger $\delta,\eta$ are, the better it is; $\delta$ and $\eta$ are not optimized independently. We will be able to take $\delta=O(n^{-1})$ and $\eta=O(n^{-\frac{1}{2}})$.
\section{Change of coordinates and properties of general functions}\label{changeofvar}
In order to study the functions $H$, we introduce coordinate systems that will be convenient for our calculations. Our integrals over $V_m(\mathbb{R}^n)$ defining $H$ become integrals over $\mathbb{B}^m\times\mathbb{B}^m$. We first treat the case of averaging over Stiefel manifolds, and then recall from~\cite{SZ} results pertaining to averaging over spheres for spherical codes and averaging over $\mathbb{R}^n$ for sphere packings.

\subsection{Averaging over Stiefel manifolds $V_m(\mathbb{R}^n)$} Recall the general construction of Subsection~\ref{mhigher}. Given points $\boldsymbol{x},\boldsymbol{y}\in S^{n-1}$ and an ordered set of orthonormal $\boldsymbol{z}_1,\hdots,\boldsymbol{z}_m\in S^{n-1}$, let
\[t:=\left<\boldsymbol{x},\boldsymbol{y}\right>,\]
and, for every $1\leq i\leq m$, let
\[u_i:=\left<\boldsymbol{x},\boldsymbol{z}_i\right>,\]
\[v_i:=\left<\boldsymbol{y},\boldsymbol{z}_i\right>.\]
Recall the definition of the projection maps $\Pi_i$ from Subsection~\ref{mhigher}. By the orthonormality of $\boldsymbol{z}_1,\hdots,\boldsymbol{z}_m$, we have
\begin{equation*}\Pi_i\cdots\Pi_1(\boldsymbol{x})=\frac{\boldsymbol{x}-\left<\boldsymbol{x},\boldsymbol{z}_1\right>\boldsymbol{z}_1-\hdots-\left<\boldsymbol{x},\boldsymbol{z}_i\right>\boldsymbol{z}_i}{|\boldsymbol{x}-\left<\boldsymbol{x},\boldsymbol{z}_1\right>\boldsymbol{z}_1-\hdots-\left<\boldsymbol{x},\boldsymbol{z}_i\right>\boldsymbol{z}_i|},\end{equation*}
and
\begin{equation*}\Pi_i\cdots\Pi_1(\boldsymbol{y})=\frac{\boldsymbol{y}-\left<\boldsymbol{y},\boldsymbol{z}_1\right>\boldsymbol{z}_1-\hdots-\left<\boldsymbol{y},\boldsymbol{z}_i\right>\boldsymbol{z}_i}{|\boldsymbol{y}-\left<\boldsymbol{y},\boldsymbol{z}_1\right>\boldsymbol{z}_1-\hdots-\left<\boldsymbol{y},\boldsymbol{z}_i\right>\boldsymbol{z}_i|}.\end{equation*}
Using these two formulas along with the orthonormality of the $\boldsymbol{z}_1,\hdots,\boldsymbol{z}_m$, we have, in the coordinates $t,u_1,v_1,\hdots,u_m,v_m$, that
\begin{equation}\label{ti}t_i:=\left<\Pi_{i}\hdots\Pi_1(\boldsymbol{x}),\Pi_{i}\hdots\Pi_1(\boldsymbol{y})\right>=\frac{t-u_1v_1-\hdots-u_iv_i}{\sqrt{(1-u_1^2-\hdots-u_i^2)(1-v_1^2-\hdots-v_i^2)}},
\end{equation}
\begin{equation}\label{projxz}
\left<\Pi_{i}\cdots\Pi_1(\boldsymbol{x}),\boldsymbol{z}_{i+1}\right>=\frac{u_{i+1}}{\sqrt{1-u_1^2-\hdots-u_i^2}},
\end{equation}
and
\begin{equation}\label{projyz}
\left<\Pi_{i}\cdots\Pi_1(\boldsymbol{y}),\boldsymbol{z}_{i+1}\right>=\frac{v_{i+1}}{\sqrt{1-v_1^2-\hdots-v_i^2}}.
\end{equation}
We now rewrite the general function
\begin{equation}\label{Ht}H(t)=\int_{V_m(\mathbb{R}^n)}F\left(\left<\boldsymbol{x},\boldsymbol{z}_1\right>,\hdots,\left<\boldsymbol{x},\boldsymbol{z}_m\right>)F(\left<\boldsymbol{y},\boldsymbol{z}_1\right>,\hdots,\left<\boldsymbol{y},\boldsymbol{z}_m\right>\right)g\left(\left<\Pi_{m}\hdots\Pi_1(\boldsymbol{x}),\Pi_{m}\hdots\Pi_1(\boldsymbol{y})\right>\right)d\boldsymbol{z}_1\cdots d\boldsymbol{z}_m\end{equation}
given in~\eqref{hStiefel} using these coordinates. Abstractly, if we let $\mu(u_1,v_1,\hdots,u_m,v_m;t)$ be the density function of the pushforward of the measure $d\boldsymbol{z}_1\cdots d\boldsymbol{z}_m$ on $V_m(\mathbb{R}^n)$ onto the coordinates $u_1,v_1,\hdots,u_m,v_m$ \textit{conditional} on fixed $t$, we obtain a measure on $\mathbb{B}^m\times \mathbb{B}^m$.  Indeed,
\begin{equation*}u_1^2+\hdots+u_m^2=\left<\boldsymbol{x},\boldsymbol{z}_1\right>^2+\hdots+\left<\boldsymbol{x},\boldsymbol{z}_m\right>^2\leq |\boldsymbol{x}|^2=1,\end{equation*}
and
\begin{equation*}v_1^2+\hdots+v_m^2=\left<\boldsymbol{y},\boldsymbol{z}_1\right>^2+\hdots+\left<\boldsymbol{y},\boldsymbol{z}_m\right>^2\leq |\boldsymbol{y}|^2=1,\end{equation*}
where the inequalities follow from the orthonormality of $\boldsymbol{z}_1,\hdots,\boldsymbol{z}_m$. Using the coordinates $t,u_1,v_1,\hdots,u_m,v_m$, this abstractly defined measure $\mu(u_1,v_1,\hdots,u_m,v_m;t)$ on $\mathbb{B}^m\times\mathbb{B}^m$ allows us to rewrite~\eqref{Ht} as
\begin{equation*}\label{Ht2}H(t)=\iint_{\mathbb{B}^m\times\mathbb{B}^m}g\left(\frac{t-\left<\boldsymbol{u},\boldsymbol{v}\right>}{\sqrt{(1-|\boldsymbol{u}|^2)(1-|\boldsymbol{v}|^2)}}\right)F\left(\boldsymbol{u})F(\boldsymbol{v}\right)\mu(\boldsymbol{u},\boldsymbol{v};t)d\boldsymbol{u}d\boldsymbol{v},\end{equation*}
where $\boldsymbol{u}=(u_1,\hdots,u_m)$, $\boldsymbol{v}=(v_1,\hdots,v_m)$, $\mu(\boldsymbol{u},\boldsymbol{v};t)=\mu(u_1,v_1,\hdots,u_m,v_m;t)$, and $d\boldsymbol{u}d\boldsymbol{v}$ is the uniform probability measure on $\mathbb{B}^m\times\mathbb{B}^m$. Here, we also used~\eqref{ti} for $i=m$. We make this implicit formula for $H(t)$ explicit by calculating $\mu(\boldsymbol{u},\boldsymbol{v};t)$ in the following lemma.
\begin{lemma}\label{conditionaldensity}Up to positive normalization constants, we have
\begin{equation*}\mu(\boldsymbol{u},\boldsymbol{v};t)=\begin{cases}\left((1-|\boldsymbol{u}|^2)(1-|\boldsymbol{v}|^2)-(t-\left<\boldsymbol{u},\boldsymbol{v}\right>)^2\right)^{\frac{n-m-3}{2}}&\textup{if }t\in(-1,1)\textup{ and }-1\leq\frac{t-\left<\boldsymbol{u},\boldsymbol{v}\right>}{\sqrt{(1-|\boldsymbol{u}|^2)(1-|\boldsymbol{v}|^2)}}\leq 1\\
\delta_{\boldsymbol{0}}(\boldsymbol{u}-\boldsymbol{v})\left(1-|\boldsymbol{u}|^2\right)^{\frac{n-m-2}{2}}&\textup{if }t=1\\
\delta_{\boldsymbol{0}}(\boldsymbol{u}+\boldsymbol{v})\left(1-|\boldsymbol{u}|^2\right)^{\frac{n-m-2}{2}}&\textup{if }t=-1\\
0&\textup{otherwise,}
\end{cases}\end{equation*}
where $\delta_{\boldsymbol{0}}$ is the delta function supported at $\boldsymbol{0}$.
\end{lemma}

\begin{proof}Note that the conditional density $\mu(\boldsymbol{u},\boldsymbol{v};t)$ is non-zero only if

\[-1\leq\frac{t-\left<\boldsymbol{u},\boldsymbol{v}\right>}{\sqrt{(1-|\boldsymbol{u}|^2)(1-|\boldsymbol{v}|^2)}}\leq 1\]
is satisfied. When this is satisfied and $t\in(-1,1)$, this conditional density $\mu(\boldsymbol{u},\boldsymbol{v};t)$ is proportional to
\[\lim_{\varepsilon\rightarrow 0^+}\frac{\vol(\cal{R}(u_1,v_1,\hdots,u_m,v_m,\varepsilon;\boldsymbol{x},\boldsymbol{y}))}{\varepsilon^{2m}},\]
where
\begin{eqnarray*}&&\cal{R}(u_1,v_1,\hdots,u_m,v_m,\varepsilon;\boldsymbol{x},\boldsymbol{y}):=\left\{(\boldsymbol{z}_1,\hdots,\boldsymbol{z}_m)\in V_m(\mathbb{R}^n)|\ \forall 1\leq i\leq m,\ 0\leq \left<\boldsymbol{x},\boldsymbol{z}_i\right>-u_i\leq\varepsilon\textup{ and }0\leq \left<\boldsymbol{y},\boldsymbol{z}_i\right>-v_i\leq\varepsilon\right\}\\
&=& \left\{(\boldsymbol{z}_1,\hdots,\boldsymbol{z}_m)\in V_m(\mathbb{R}^n)|\ \forall 1\leq i\leq m,\ \substack{0\leq \left<\Pi_{i-1}\hdots\Pi_1\boldsymbol{x},\boldsymbol{z}_i\right>-\frac{u_i}{\sqrt{1-u_1^2-\hdots-u_{i-1}^2}}\leq\frac{\varepsilon}{\sqrt{1-u_1^2-\hdots-u_{i-1}^2}}\\ 0\leq \left<\Pi_{i-1}\hdots\Pi_1\boldsymbol{y},\boldsymbol{z}_i\right>-\frac{v_i}{\sqrt{1-v_1^2-\hdots-v_{i-1}^2}}\leq\frac{\varepsilon}{\sqrt{1-v_1^2-\hdots-v_{i-1}^2}}}\right\}.
\end{eqnarray*}
The second equality is a consequence of equalities~\eqref{projxz} and~\eqref{projyz}. From this and the proof of Proposition 4.1 of~\cite{SZ}, we obtain that $\mu(\boldsymbol{u},\boldsymbol{v};t)$ is proportional to

\begin{eqnarray*}
\frac{\frac{\det\begin{bmatrix}1&t& u_1\\ t&1&v_1\\ u_1&v_1&1\end{bmatrix}^{\frac{n-4}{2}}}{(1-t^2)^{\frac{n-3}{2}}}\frac{\det\begin{bmatrix}1&t_1& \frac{u_2}{\sqrt{1-u_1^2}}\\ t_1&1&\frac{v_2}{\sqrt{1-v_1^2}}\\ \frac{u_2}{\sqrt{1-u_1^2}}&\frac{v_2}{\sqrt{1-v_1^2}}&1\end{bmatrix}^{\frac{n-5}{2}}}{(1-t_1^2)^{\frac{n-4}{2}}}\cdots\frac{\det\begin{bmatrix}1&t_{m-1}& \frac{u_m}{\sqrt{1-u_1^2-\hdots-u_{m-1}^2}}\\ t_{m-1}&1&\frac{v_m}{\sqrt{1-v_1^2-\hdots-v_{m-1}^2}}\\ \frac{u_m}{\sqrt{1-u_1^2-\hdots-u_{m-1}^2}}&\frac{v_m}{\sqrt{1-v_1^2-\hdots-v_{m-1}^2}}&1\end{bmatrix}^{\frac{n-m-3}{2}}}{(1-t_{m-1}^2)^{\frac{n-m-2}{2}}}}{\prod_{i=1}^{m-1}\sqrt{\left(1-u_1^2-\hdots-u_i^2\right)\left(1-v_1^2-\hdots-v_i^2\right)}}.
\end{eqnarray*}
This is equal to
\[\frac{\left((1-t_m^2)(1-u_1^2-\hdots-u_m^2)(1-v_1^2-\hdots-v_m^2)\right)^{\frac{n-m-3}{2}}}{(1-t^2)^{\frac{n-3}{2}}}=\frac{\left((1-|\boldsymbol{u}|^2)(1-|\boldsymbol{v}|^2)-(t-\left<\boldsymbol{u},\boldsymbol{v}\right>)^2\right)^{\frac{n-m-3}{2}}}{(1-t^2)^{\frac{n-3}{2}}}.\]
The cases $t=\pm 1$ follow in an easier manner. Indeed, $t=1$ is equivalent to $\boldsymbol{x}=\boldsymbol{y}$, implying that $\boldsymbol{u}=\boldsymbol{v}$. As a result, $\mu(\boldsymbol{u},\boldsymbol{v};1)$ is proportional to
\[\delta_0(\boldsymbol{u}-\boldsymbol{v})\lim_{\varepsilon\rightarrow 0^+}\frac{\vol(\cal{D}(u_1,\hdots,u_m,\varepsilon;\boldsymbol{x}))}{\varepsilon^{m}},\]
where
\begin{eqnarray*}&&\cal{D}(u_1,\hdots,u_m,\varepsilon;\boldsymbol{x}):=\left\{(\boldsymbol{z}_1,\hdots,\boldsymbol{z}_m)\in V_m(\mathbb{R}^n)|\ \forall 1\leq i\leq m,\ 0\leq \left<\boldsymbol{x},\boldsymbol{z}_i\right>-u_i\leq\varepsilon\right\}\\
&=& \left\{(\boldsymbol{z}_1,\hdots,\boldsymbol{z}_m)\in V_m(\mathbb{R}^n)|\ \forall 1\leq i\leq m,\ \substack{\\ \\ 0\leq \left<\Pi_{i-1}\hdots\Pi_1\boldsymbol{x},\boldsymbol{z}_i\right>-\frac{u_i}{\sqrt{1-u_1^2-\hdots-u_{i-1}^2}}\leq\frac{\varepsilon}{\sqrt{1-u_1^2-\hdots-u_{i-1}^2}}}\right\}.
\end{eqnarray*}
Therefore, $\mu(\boldsymbol{u},\boldsymbol{v};1)$ is proportional to
\[\delta_0(\boldsymbol{u}-\boldsymbol{v})(1-u_1^2)^{\frac{n-3}{2}}\frac{\left(1-\left(\frac{u_2}{\sqrt{1-u_1^2}}\right)^2\right)^{\frac{n-4}{2}}}{\sqrt{1-u_1^2}}\hdots\frac{\left(1-\left(\frac{u_m}{\sqrt{1-u_1^2-\hdots-u_{m-1}^2}}\right)^2\right)^{\frac{n-m-2}{2}}}{\sqrt{1-u_1^2-\hdots-u_{m-1}^2}}=\delta_0(\boldsymbol{u}-\boldsymbol{v})\left(1-|\boldsymbol{u}|^2\right)^{\frac{n-m-2}{2}},\]
as required. The case $t=-1$ is proved in a similar way.
 The conclusion follows.
\end{proof}
Let
\begin{equation*}\label{Rmtdef}\cal{R}^m_t:=\left\{(\boldsymbol{u},\boldsymbol{v})\in\mathbb{B}^m\times\mathbb{B}^m:-1\leq\frac{t-\left<\boldsymbol{u},\boldsymbol{v}\right>}{\sqrt{(1-|\boldsymbol{u}|^2)(1-|\boldsymbol{v}|^2)}}\leq 1\right\}.\end{equation*}
A recurring notation throughout this paper is the following. Since $H$ are used to bound sizes of spherical codes in $S^{n-1}$, we write
\begin{equation*}\label{H0def}
H_0:=\frac{\int_{-1}^1H(t)(1-t^2)^{\frac{n-3}{2}}dt}{\int_{-1}^1(1-t^2)^{\frac{n-3}{2}}dt}.
\end{equation*}
On the other hand, since $g$ will be chosen to relate to spherical codes in $S^{n-m-1}$, we write
\begin{equation*}\label{g0def}
g_0:=\frac{\int_{-1}^1g(t)(1-t^2)^{\frac{n-m-3}{2}}dt}{\int_{-1}^1(1-t^2)^{\frac{n-m-3}{2}}dt}.
\end{equation*}
As a consequence of Lemma~\ref{conditionaldensity}, we obtain the following.
\begin{proposition}\label{Hexplicit1}The general function $H(t)$ is a positive multiple of
\begin{equation}\label{Hrewritten}
\iint_{\cal{R}^m_t}g\left(\frac{t-\left<\boldsymbol{u},\boldsymbol{v}\right>}{\sqrt{(1-|\boldsymbol{u}|^2)(1-|\boldsymbol{v}|^2)}}\right)F\left(\boldsymbol{u})F(\boldsymbol{v}\right)\left((1-|\boldsymbol{u}|^2)(1-|\boldsymbol{v}|^2)-(t-\left<\boldsymbol{u},\boldsymbol{v}\right>)^2\right)^{\frac{n-m-3}{2}}d\boldsymbol{u}d\boldsymbol{v}
\end{equation}
for $t\in(-1,1)$. We also have 
\begin{equation}\label{Hrewrittengeneral}
H(1)=g(1)\left(\frac{\int_{\mathbb{B}^m}F(\boldsymbol{u})^2(1-|\boldsymbol{u}|^2)^{\frac{n-m-2}{2}}d\boldsymbol{u}}{\int_{\mathbb{B}^m}(1-|\boldsymbol{u}|^2)^{\frac{n-m-2}{2}}d\boldsymbol{u}}\right)
\end{equation}
and
\begin{equation}\label{H0rewrittengeneral}H_0=g_0\left(\frac{\int_{\mathbb{B}^m}F(\boldsymbol{u})(1-|\boldsymbol{u}|^2)^{\frac{n-m-2}{2}}d\boldsymbol{u}}{\int_{\mathbb{B}^m}(1-|\boldsymbol{u}|^2)^{\frac{n-m-2}{2}}d\boldsymbol{u}}\right)^2.
\end{equation}
\end{proposition}
\begin{proof}Both~\eqref{Hrewritten} and~\eqref{Hrewrittengeneral} are immediate consequences of Lemma~\ref{conditionaldensity} and its proof. We verify~\eqref{H0rewrittengeneral} as follows. It follows from~\eqref{Hrewritten} that
\begin{eqnarray*}&&H_0\\&=&\frac{\int_{-1}^1\iint_{\cal{R}^m_t}g\left(\frac{t-\left<\boldsymbol{u},\boldsymbol{v}\right>}{\sqrt{(1-|\boldsymbol{u}|^2)(1-|\boldsymbol{v}|^2)}}\right)F\left(\boldsymbol{u})F(\boldsymbol{v}\right)\left((1-|\boldsymbol{u}|^2)(1-|\boldsymbol{v}|^2)-(t-\left<\boldsymbol{u},\boldsymbol{v}\right>)^2\right)^{\frac{n-m-3}{2}}d\boldsymbol{u}d\boldsymbol{v}dt}{\int_{-1}^1\iint_{\cal{R}^m_t}\left((1-|\boldsymbol{u}|^2)(1-|\boldsymbol{v}|^2)-(t-\left<\boldsymbol{u},\boldsymbol{v}\right>)^2\right)^{\frac{n-m-3}{2}}d\boldsymbol{u}d\boldsymbol{v}dt}\\
&=& \frac{\iint_{\mathbb{B}^m\times\mathbb{B}^m}F\left(\boldsymbol{u})F(\boldsymbol{v}\right)\left((1-|\boldsymbol{u}|^2)(1-|\boldsymbol{v}|^2)\right)^{\frac{n-m-3}{2}}\int\limits_{\substack{-1\leq t\leq 1\\ -1\leq\frac{t-\left<\boldsymbol{u},\boldsymbol{v}\right>}{\sqrt{(1-|\boldsymbol{u}|^2)(1-|\boldsymbol{v}|^2)}}\leq 1}}g\left(\frac{t-\left<\boldsymbol{u},\boldsymbol{v}\right>}{\sqrt{(1-|\boldsymbol{u}|^2)(1-|\boldsymbol{v}|^2)}}\right)\left(1-\left(\frac{t-\left<\boldsymbol{u},\boldsymbol{v}\right>}{\sqrt{(1-|\boldsymbol{u}|^2)(1-|\boldsymbol{v}|^2)}}\right)^2\right)^{\frac{n-m-3}{2}}dtd\boldsymbol{u}d\boldsymbol{v}}{\iint_{\mathbb{B}^m\times\mathbb{B}^m}\left((1-|\boldsymbol{u}|^2)(1-|\boldsymbol{v}|^2)\right)^{\frac{n-m-3}{2}}\int\limits_{\substack{-1\leq t\leq 1\\ -1\leq\frac{t-\left<\boldsymbol{u},\boldsymbol{v}\right>}{\sqrt{(1-|\boldsymbol{u}|^2)(1-|\boldsymbol{v}|^2)}}\leq 1}}\left(1-\left(\frac{t-\left<\boldsymbol{u},\boldsymbol{v}\right>}{\sqrt{(1-|\boldsymbol{u}|^2)(1-|\boldsymbol{v}|^2)}}\right)^2\right)^{\frac{n-m-3}{2}}dtd\boldsymbol{u}d\boldsymbol{v}}.\end{eqnarray*}
For any fixed $\boldsymbol{u},\boldsymbol{v}\in\mathbb{B}^m$, when $t=1$ we deduce that
\[\frac{t-\left<\boldsymbol{u},\boldsymbol{v}\right>}{\sqrt{(1-|\boldsymbol{u}|^2)(1-|\boldsymbol{v}|^2)}}\geq \frac{1-|\boldsymbol{u}||\boldsymbol{v}|}{\sqrt{(1-|\boldsymbol{u}|^2)(1-|\boldsymbol{v}|^2)}}\geq 1,\]
where the first inequality is a consequence of the Cauchy--Schwarz inequality and the second inequality is equivalent to $(|\boldsymbol{u}|-|\boldsymbol{v}|)^2\geq 0$. When $t=-1$, we similarly obtain that
\[\frac{t-\left<\boldsymbol{u},\boldsymbol{v}\right>}{\sqrt{(1-|\boldsymbol{u}|^2)(1-|\boldsymbol{v}|^2)}}\leq -\frac{1-|\boldsymbol{u}||\boldsymbol{v}|}{\sqrt{(1-|\boldsymbol{u}|^2)(1-|\boldsymbol{v}|^2)}}\leq -1.\]
As a result, as $t$ varies in $[-1,1]$, $\frac{t-\left<\boldsymbol{u},\boldsymbol{v}\right>}{\sqrt{(1-|\boldsymbol{u}|^2)(1-|\boldsymbol{v}|^2)}}$ linearly covers $[-1,1]$. Consequently, rewriting the integral over $t$ in terms of $\frac{t-\left<\boldsymbol{u},\boldsymbol{v}\right>}{\sqrt{(1-|\boldsymbol{u}|^2)(1-|\boldsymbol{v}|^2)}}$, we get 
\begin{eqnarray*}H_0&=&\frac{\iint_{\mathbb{B}^m\times\mathbb{B}^m}F\left(\boldsymbol{u})F(\boldsymbol{v}\right)\left((1-|\boldsymbol{u}|^2)(1-|\boldsymbol{v}|^2)\right)^{\frac{n-m-2}{2}}d\boldsymbol{u}d\boldsymbol{v}\int_{-1}^1g(t)(1-t^2)^{\frac{n-m-3}{2}}dt}{\iint_{\mathbb{B}^m\times\mathbb{B}^m}\left((1-|\boldsymbol{u}|^2)(1-|\boldsymbol{v}|^2)\right)^{\frac{n-m-2}{2}}d\boldsymbol{u}d\boldsymbol{v}\int_{-1}^1(1-t^2)^{\frac{n-m-3}{2}}dt}\\
&=&g_0\left(\frac{\int_{\mathbb{B}^m}F(\boldsymbol{u})(1-|\boldsymbol{u}|^2)^{\frac{n-m-2}{2}}d\boldsymbol{u}}{\int_{\mathbb{B}^m}(1-|\boldsymbol{u}|^2)^{\frac{n-m-2}{2}}d\boldsymbol{u}}\right)^2,\end{eqnarray*}
as required.
\end{proof}

From Proposition~\ref{Hexplicit1}, the Delsarte linear programming method implies that if we choose $F$ a characteristic function and suitable $g$ and $m$ so that the associated function $H$ satisfies $H(t)\leq 0$ for every $t\in[-1,\cos\theta]$, then
\begin{equation}\label{generalchibound}M(n,\theta)\leq \frac{g(1)/g_0}{\left(\frac{\int_{\mathbb{B}^m}F(\boldsymbol{u})(1-|\boldsymbol{u}|^2)^{\frac{n-m-2}{2}}d\boldsymbol{u}}{\int_{\mathbb{B}^m}(1-|\boldsymbol{u}|^2)^{\frac{n-m-2}{2}}d\boldsymbol{u}}\right)}.\end{equation}
We prove the following lemma for future use.
\begin{lemma}\label{masscomputation}For $F$ the characteristic function of the region
\[\left\{\boldsymbol{u}\in\mathbb{B}^m:R\leq |\boldsymbol{u}|\leq S\textup{ and }\left<\boldsymbol{u},\boldsymbol{p}\right>\geq |\boldsymbol{u}|\cos\varphi\right\},\]
we have
\begin{equation*}\frac{\int_{\mathbb{B}^m}F(\boldsymbol{u})(1-|\boldsymbol{u}|^2)^{\frac{n-m-2}{2}}d\boldsymbol{u}}{\int_{\mathbb{B}^m}(1-|\boldsymbol{u}|^2)^{\frac{n-m-2}{2}}d\boldsymbol{u}}=\frac{\int_{\cos\varphi}^1(1-t^2)^{\frac{m-3}{2}}dt}{\int_{-1}^1(1-t^2)^{\frac{m-3}{2}}dt}\frac{\int_R^S(1-r^2)^{\frac{n-m-2}{2}}r^{m-1}dr}{\int_0^1(1-r^2)^{\frac{n-m-2}{2}}r^{m-1}dr}.\end{equation*}
\end{lemma}
\begin{proof}
First, using spherical coordinates, we obtain
\[\int_{\mathbb{B}^m}(1-|\boldsymbol{u}|^2)^{\frac{n-m-2}{2}}d\boldsymbol{u}=\vol(S^{m-1})\int_0^1(1-r^2)^{\frac{n-m-2}{2}}r^{m-1}dr,\]
where $\vol(S^{m-1})$ is the volume of the unit $(m-1)$-sphere. Similarly,
\[\int_{\mathbb{B}^m}F(\boldsymbol{u})(1-|\boldsymbol{u}|^2)^{\frac{n-m-2}{2}}d\boldsymbol{u}=\vol(S^{m-1}_{\varphi})\int_R^S(1-r^2)^{\frac{n-m-2}{2}}r^{m-1}dr,\]
where $\vol(S^{m-1}_{\varphi})$ is the volume of the subset of points $\boldsymbol{u}\in S^{m-1}$ satisfying $\left<\boldsymbol{u},\boldsymbol{p}\right>\geq\cos\varphi$. We know that
\[\frac{\vol(S^{m-1}_{\varphi})}{\vol(S^{m-1})}=\frac{\int_{\cos\varphi}^1(1-t^2)^{\frac{m-3}{2}}dt}{\int_{-1}^1(1-t^2)^{\frac{m-3}{2}}dt}.\]
Therefore,
\[\frac{\int_{\mathbb{B}^m}F(\boldsymbol{u})(1-|\boldsymbol{u}|^2)^{\frac{n-m-2}{2}}d\boldsymbol{u}}{\int_{\mathbb{B}^m}(1-|\boldsymbol{u}|^2)^{\frac{n-m-2}{2}}d\boldsymbol{u}}=\frac{\int_{\cos\varphi}^1(1-t^2)^{\frac{m-3}{2}}dt}{\int_{-1}^1(1-t^2)^{\frac{m-3}{2}}dt}\frac{\int_R^S(1-r^2)^{\frac{n-m-2}{2}}r^{m-1}dr}{\int_0^1(1-r^2)^{\frac{n-m-2}{2}}r^{m-1}dr},\]
as required.
\end{proof}
Though the coordinate system above was suitable for proving the above results, in order to verify for which choice of parameters we have $H(t)\leq 0$ for every $t\in[-1,\cos\theta]$, it is sometimes convenient to reduce to one-dimensional integrals, depending on how complicated $g$ is chosen to be. We do not use the following proposition in this paper; instead, we give sharp estimates of Levenshtein's optimal polynomials near their largest roots. See Section~\ref{Jacobiestimates}. Nonetheless, we state and prove the following proposition.
\begin{proposition}\label{Hexplicit2} For $m\geq 2$,
\begin{equation*}\iint_{\cal{R}^m_t}g\left(\frac{t-\left<\boldsymbol{u},\boldsymbol{v}\right>}{\sqrt{(1-|\boldsymbol{u}|^2)(1-|\boldsymbol{v}|^2)}}\right)F\left(\boldsymbol{u})F(\boldsymbol{v}\right)\left((1-|\boldsymbol{u}|^2)(1-|\boldsymbol{v}|^2)-(t-\left<\boldsymbol{u},\boldsymbol{v}\right>)^2\right)^{\frac{n-m-3}{2}}d\boldsymbol{u}d\boldsymbol{v}=\int_{-1}^1g(x)\nu_m(x;t,F)dx,\end{equation*}
where
\begin{equation*}\label{nugeneral}\nu_m(x;t,F):=\left(\frac{1-x^2}{x^2}\right)^{\frac{n-m-3}{2}}\int_{-1}^1\int_{\cal{R}^m_{t,\alpha,x}}F\left(\boldsymbol{u})F(\boldsymbol{v}\right)\frac{\left(t-|\boldsymbol{u}||\boldsymbol{v}|\alpha\right)^{n-m-3}}{|\grad f_{t,\alpha}|\sqrt{\left(\frac{1}{|\boldsymbol{u}|^2}+\frac{1}{|\boldsymbol{v}|^2}\right)}}dA_{x,\alpha}\frac{d\alpha}{\sqrt{1-\alpha^2}},\end{equation*}
$dA_{x,\alpha}$ is the volume measure induced from the pullback of the Euclidean metric on $\mathbb{B}^m\times\mathbb{B}^m$ to
\begin{equation*}\cal{R}^m_{t,\alpha,x}:=\left\{(\boldsymbol{u},\boldsymbol{v})\in(\mathbb{B}^m\setminus\{\boldsymbol{0}\})\times(\mathbb{B}^m\setminus\{\boldsymbol{0}\}):\frac{t-\left<\boldsymbol{u},\boldsymbol{v}\right>}{\sqrt{(1-|\boldsymbol{u}|^2)(1-|\boldsymbol{v}|^2)}}=x\textup{ and }\left<\frac{\boldsymbol{u}}{|\boldsymbol{u}|},\frac{\boldsymbol{v}}{|\boldsymbol{v}|}\right>=\alpha\right\},\end{equation*}
and
\begin{equation*}|\grad f_{t,\alpha}|=\sqrt{\frac{(|\boldsymbol{u}|t-|\boldsymbol{v}|\alpha)^2}{(1-|\boldsymbol{u}|^2)^3(1-|\boldsymbol{v}|^2)}+\frac{(|\boldsymbol{v}|t-|\boldsymbol{u}|\alpha)^2}{(1-|\boldsymbol{u}|^2)(1-|\boldsymbol{v}|^2)^3}},\end{equation*}
the norm of the gradient of the function
\begin{equation*}f_{t,\alpha}(\boldsymbol{u},\boldsymbol{v}):=\frac{t-|\boldsymbol{u}||\boldsymbol{v}|\alpha}{\sqrt{(1-|\boldsymbol{u}|^2)(1-|\boldsymbol{v}|^2)}}.\end{equation*}
\end{proposition}
\begin{proof}
First, consider the function $\pi:(\mathbb{B}^m\setminus\{\boldsymbol{0}\})\times(\mathbb{B}^m\setminus\{\boldsymbol{0}\})\rightarrow[-1,1]$ given by
\[\pi(\boldsymbol{u},\boldsymbol{v})=\left<\frac{\boldsymbol{u}}{|\boldsymbol{u}|},\frac{\boldsymbol{v}}{|\boldsymbol{v}|}\right>.\]
Then
\[|\grad \pi|^2=\left(\frac{1}{|\boldsymbol{u}|^2}+\frac{1}{|\boldsymbol{v}|^2}\right)\left(1-\left<\frac{\boldsymbol{u}}{|\boldsymbol{u}|},\frac{\boldsymbol{v}}{|\boldsymbol{v}|}\right>^2\right).\]
By the co-area formula,
\begin{eqnarray*}
&&\iint_{\cal{R}^m_t}g\left(\frac{t-\left<\boldsymbol{u},\boldsymbol{v}\right>}{\sqrt{(1-|\boldsymbol{u}|^2)(1-|\boldsymbol{v}|^2)}}\right)F\left(\boldsymbol{u})F(\boldsymbol{v}\right)\left((1-|\boldsymbol{u}|^2)(1-|\boldsymbol{v}|^2)-(t-\left<\boldsymbol{u},\boldsymbol{v}\right>)^2\right)^{\frac{n-m-3}{2}}d\boldsymbol{u}d\boldsymbol{v}\\
&=&\iint_{\cal{R}^m_t\cap((\mathbb{B}^m\setminus\{\boldsymbol{0}\})\times(\mathbb{B}^m\setminus\{\boldsymbol{0}\}))}g\left(\frac{t-|\boldsymbol{u}||\boldsymbol{v}|\left<\frac{\boldsymbol{u}}{|\boldsymbol{u}|},\frac{\boldsymbol{v}}{|\boldsymbol{v}|}\right>}{\sqrt{(1-|\boldsymbol{u}|^2)(1-|\boldsymbol{v}|^2)}}\right)F\left(\boldsymbol{u})F(\boldsymbol{v}\right)\left((1-|\boldsymbol{u}|^2)(1-|\boldsymbol{v}|^2)-\left(t-|\boldsymbol{u}||\boldsymbol{v}|\left<\frac{\boldsymbol{u}}{|\boldsymbol{u}|},\frac{\boldsymbol{v}}{|\boldsymbol{v}|}\right>\right)^2\right)^{\frac{n-m-3}{2}}d\boldsymbol{u}d\boldsymbol{v}\\
&=&\int_{-1}^1\int_{\cal{R}^m_{t,\alpha}}g\left(\frac{t-|\boldsymbol{u}||\boldsymbol{v}|\alpha}{\sqrt{(1-|\boldsymbol{u}|^2)(1-|\boldsymbol{v}|^2)}}\right)F\left(\boldsymbol{u})F(\boldsymbol{v}\right)\frac{\left((1-|\boldsymbol{u}|^2)(1-|\boldsymbol{v}|^2)-(t-|\boldsymbol{u}||\boldsymbol{v}|\alpha)^2\right)^{\frac{n-m-3}{2}}}{\sqrt{\left(\frac{1}{|\boldsymbol{u}|^2}+\frac{1}{|\boldsymbol{v}|^2}\right)}}d\boldsymbol{u}d\boldsymbol{v}\frac{d\alpha}{\sqrt{1-\alpha^2}},
\end{eqnarray*}
where
\[\cal{R}^m_{t,\alpha}:=\left\{(\boldsymbol{u},\boldsymbol{v})\in(\mathbb{B}^m\setminus\{\boldsymbol{0}\})\times(\mathbb{B}^m\setminus\{\boldsymbol{0}\}):-1\leq\frac{t-\left<\boldsymbol{u},\boldsymbol{v}\right>}{\sqrt{(1-|\boldsymbol{u}|^2)(1-|\boldsymbol{v}|^2)}}\leq 1\textup{ and }\left<\frac{\boldsymbol{u}}{|\boldsymbol{u}|},\frac{\boldsymbol{v}}{|\boldsymbol{v}|}\right>=\alpha\right\}.\]
Consider the function $f_{t,\alpha}:\mathbb{B}^m\times\mathbb{B}^m\rightarrow \mathbb{R}$ given by
\[f_{t,\alpha}(\boldsymbol{u},\boldsymbol{v}):=\frac{t-|\boldsymbol{u}||\boldsymbol{v}|\alpha}{\sqrt{(1-|\boldsymbol{u}|^2)(1-|\boldsymbol{v}|^2)}}.\]
The function $f_{t,\alpha}$ satisfies
\begin{eqnarray*}|\grad f_{t,\alpha}|^2&=&\sum_{i=1}^m\left(\frac{(u_i(|\boldsymbol{u}|t-|\boldsymbol{v}|\alpha))^2}{|\boldsymbol{u}|^2(1-|\boldsymbol{u}|^2)^3(1-|\boldsymbol{v}|^2)}+\frac{(v_i(|\boldsymbol{v}|t-|\boldsymbol{u}|\alpha))^2}{|\boldsymbol{v}|^2(1-|\boldsymbol{u}|^2)(1-|\boldsymbol{v}|^2)^3}\right)\\
&=&\frac{(|\boldsymbol{u}|t-|\boldsymbol{v}|\alpha)^2}{(1-|\boldsymbol{u}|^2)^3(1-|\boldsymbol{v}|^2)}+\frac{(|\boldsymbol{v}|t-|\boldsymbol{u}|\alpha)^2}{(1-|\boldsymbol{u}|^2)(1-|\boldsymbol{v}|^2)^3}.
\end{eqnarray*}
A second application of the co-area formula gives
\begin{eqnarray*}
&&\iint_{\cal{R}^m_t}g\left(\frac{t-\left<\boldsymbol{u},\boldsymbol{v}\right>}{\sqrt{(1-|\boldsymbol{u}|^2)(1-|\boldsymbol{v}|^2)}}\right)F\left(\boldsymbol{u})F(\boldsymbol{v}\right)\left((1-|\boldsymbol{u}|^2)(1-|\boldsymbol{v}|^2)-(t-\left<\boldsymbol{u},\boldsymbol{v}\right>)^2\right)^{\frac{n-m-3}{2}}d\boldsymbol{u}d\boldsymbol{v}\\
&=&\int_{-1}^1g\left(x\right)(1-x^2)^{\frac{n-m-3}{2}}\int_{-1}^1\int_{\cal{R}^m_{t,\alpha,x}}F\left(\boldsymbol{u})F(\boldsymbol{v}\right)\frac{\left((1-|\boldsymbol{u}|^2)(1-|\boldsymbol{v}|^2)\right)^{\frac{n-m-3}{2}}}{|\grad f_{t,\alpha}|\sqrt{\left(\frac{1}{|\boldsymbol{u}|^2}+\frac{1}{|\boldsymbol{v}|^2}\right)}}dA_{x,\alpha}\frac{d\alpha}{\sqrt{1-\alpha^2}}dx\\
&=&\int_{-1}^1g\left(x\right)\left(\frac{1-x^2}{x^2}\right)^{n-m-3}\int_{-1}^1\int_{\cal{R}^m_{t,\alpha,x}}F\left(\boldsymbol{u})F(\boldsymbol{v}\right)\frac{\left(t-|\boldsymbol{u}||\boldsymbol{v}|\alpha\right)^{n-m-3}}{|\grad f_{t,\alpha}|\sqrt{\left(\frac{1}{|\boldsymbol{u}|^2}+\frac{1}{|\boldsymbol{v}|^2}\right)}}dA_{x,\alpha}\frac{d\alpha}{\sqrt{1-\alpha^2}}dx,
\end{eqnarray*}
where
\[\cal{R}^m_{t,\alpha,x}:=\left\{(\boldsymbol{u},\boldsymbol{v})\in(\mathbb{B}^m\setminus\{\boldsymbol{0}\})\times(\mathbb{B}^m\setminus\{\boldsymbol{0}\}):\frac{t-\left<\boldsymbol{u},\boldsymbol{v}\right>}{\sqrt{(1-|\boldsymbol{u}|^2)(1-|\boldsymbol{v}|^2)}}=x\textup{ and }\left<\frac{\boldsymbol{u}}{|\boldsymbol{u}|},\frac{\boldsymbol{v}}{|\boldsymbol{v}|}\right>=\alpha\right\}\]
and $dA_{x,\alpha}$ is the volume measure on $\cal{R}^m_{t,\alpha,x}$ obtained from the pullback of the Euclidean metric on $\mathbb{B}^m\times\mathbb{B}^m$. 
We let
\[\nu_m(x;t,F):=\left(\frac{1-x^2}{x^2}\right)^{\frac{n-m-3}{2}}\int_{-1}^1\int_{\cal{R}^m_{t,\alpha,x}}F\left(\boldsymbol{u})F(\boldsymbol{v}\right)\frac{\left(t-|\boldsymbol{u}||\boldsymbol{v}|\alpha\right)^{n-m-3}}{|\grad f_{t,\alpha}|\sqrt{\left(\frac{1}{|\boldsymbol{u}|^2}+\frac{1}{|\boldsymbol{v}|^2}\right)}}dA_{x,\alpha}\frac{d\alpha}{\sqrt{1-\alpha^2}},\]
from which the conclusion follows.
\end{proof}

\subsection{Averaging over spheres}
In~\cite{SZ}, the author and Sardari worked with the specialization $m=1$, that is, we averaged over $V_1(\mathbb{R}^n)=S^{n-1}$. Since we will use this when proving Theorem~\ref{THMtwo}, we explicitly state the $m=1$ specialization of the previous subsection. We also use this opportunity to fix notation.\\
\\
Given points $\boldsymbol{x}$, $\boldsymbol{y}$, $\boldsymbol{z}\in S^{n-1}$, let
\[t:=\left<\boldsymbol{x},\boldsymbol{y}\right>,\]
\[u:=\left<\boldsymbol{x},\boldsymbol{z}\right>,\]
\[v:=\left<\boldsymbol{y},\boldsymbol{z}\right>.\]
$\Pi_{\boldsymbol{z}}$ be the projection $\Pi_1$ from the introduction for projection of points in $S^{n-1}\setminus\{\pm\boldsymbol{z}\}$ onto the orthogonal complement of $\boldsymbol{z}$ followed by normalization to a unit vector. We have
\begin{equation*}\Pi_{\boldsymbol{z}}(\boldsymbol{x})=\frac{\boldsymbol{x}-\left<\boldsymbol{x},\boldsymbol{z}\right>\boldsymbol{z}}{|\boldsymbol{x}-\left<\boldsymbol{x},\boldsymbol{z}\right>\boldsymbol{z}|},\end{equation*}
and
\begin{equation*}\Pi_{\boldsymbol{z}}(\boldsymbol{y})=\frac{\boldsymbol{y}-\left<\boldsymbol{y},\boldsymbol{z}\right>\boldsymbol{z}}{|\boldsymbol{y}-\left<\boldsymbol{y},\boldsymbol{z}\right>\boldsymbol{z}|}.\end{equation*}
Using these two formulas, we have, in the coordinates $t,u,v$, that
\begin{equation*}\left<\Pi_{\boldsymbol{z}}(\boldsymbol{x}),\Pi_{\boldsymbol{z}}(\boldsymbol{y})\right>=\frac{t-uv}{\sqrt{(1-u^2)(1-v^2)}}.
\end{equation*}
We may also rewrite the general function
\begin{equation*}H\left(\left<\boldsymbol{x},\boldsymbol{y}\right>\right)=\int_{S^{n-1}}F\left(\left<\boldsymbol{x},\boldsymbol{z}\right>)F(\left<\boldsymbol{y},\boldsymbol{z}\right>\right)g\left(\left<\Pi_{\boldsymbol{z}}(\boldsymbol{x}),\Pi_{\boldsymbol{z}}(\boldsymbol{y})\right>\right)d\boldsymbol{z}\end{equation*}
constructed in the introduction in these coordinates. Abstractly, if we let $\mu(u,v;t)$ be the density function of the projection of the measure $d\boldsymbol{z}$ on $S^{n-1}$ onto the coordinates $u,v$ \textit{conditional} on fixed $t$, we obtain a measure on $[-1,1]^2$. Using the coordinates $t,u,v$, this abstractly defined measure $\mu(u,v;t)$ on $[-1,1]^2$ allows us to rewrite
\begin{equation*}H(t)=\iint_{[-1,1]^2}g\left(\frac{t-uv}{\sqrt{(1-u^2)(1-v^2)}}\right)F\left(u)F(v\right)\mu(u,v;t)dudv.\end{equation*}
Let
\begin{equation*}\cal{R}_t:=\left\{(u,v)\in[-1,1]^2:-1\leq\frac{t-uv}{\sqrt{(1-u^2)(1-v^2)}}\leq 1\right\}.\end{equation*}
As a consequence of Proposition~\ref{Hexplicit1}, we obtain that for $t\in(-1,1)$, $H(t)$ is a positive multiple of
\begin{equation}\label{Hrewrittenm1}
\iint_{\cal{R}_t}g\left(\frac{t-uv}{\sqrt{(1-u^2)(1-v^2)}}\right)F(u)F(v)\left((1-u^2)(1-v^2)-(t-uv)^2\right)^{\frac{n-4}{2}}dudv.
\end{equation}
Furthermore, 
\begin{equation*}\label{Hrewrittenspecialm1}
H(1)=g(1)\left(\frac{\int_{-1}^1F(u)^2(1-u^2)^{\frac{n-3}{2}}du}{\int_{-1}^1(1-u^2)^{\frac{n-3}{2}}du}\right)
\end{equation*}
and
\begin{equation*}\label{H0generalm1}H_0=g_0\left(\frac{\int_{-1}^1F(u)(1-u^2)^{\frac{n-3}{2}}du}{\int_{-1}^1(1-u^2)^{\frac{n-3}{2}}du}\right)^2.\end{equation*}
If our parameters and functions are chosen so that $H(t)\leq 0$ for every $t\in[-1,\cos\theta]$, then
\begin{equation*}
M(n,\theta)\leq\frac{g(1)}{g_0}\cdot\frac{\left(\frac{\int_{-1}^1F(u)^2(1-u^2)^{\frac{n-3}{2}}du}{\int_{-1}^1(1-u^2)^{\frac{n-3}{2}}du}\right)}{\left(\frac{\int_{-1}^1F(u)(1-u^2)^{\frac{n-3}{2}}du}{\int_{-1}^1(1-u^2)^{\frac{n-3}{2}}du}\right)^2}.
\end{equation*}
Our function $F$ will be chosen to be a characteristic function, in which case, this inequality simplifies to
\begin{equation*}
M(n,\theta)\leq\frac{g(1)/g_0}{\left(\frac{\int_{-1}^1F(u)(1-u^2)^{\frac{n-3}{2}}du}{\int_{-1}^1(1-u^2)^{\frac{n-3}{2}}du}\right)}.
\end{equation*}
The analogue of Proposition~\ref{Hexplicit2} for $m=1$ is the following.
\begin{proposition}\label{Hexplicit2m1}
\begin{equation*}\iint_{\cal{R}_t}g\left(\frac{t-uv}{\sqrt{(1-u^2)(1-v^2)}}\right)F(u)F(v)\left((1-u^2)(1-v^2)-(t-uv)^2\right)^{\frac{n-4}{2}}dudv=\int_{-1}^1g(x)\nu_1(x;t,F)dx,\end{equation*}
where
\begin{equation*}\cal{R}_{t,x}:=\left\{(u,v)\in[-1,1]^2:\frac{t-uv}{\sqrt{(1-u^2)(1-v^2)}}=x\right\},\end{equation*}
and
\begin{equation*}\nu_1(x;t,F):=\left(\frac{1-x^2}{x^2}\right)^{\frac{n-4}{2}}\int_{\cal{R}_{t,x}}F(u)F(v)\frac{\left(t-uv\right)^{\frac{n-4}{2}}}{\sqrt{\frac{(ut-v)^2}{(1-u^2)^3(1-v^2)}+\frac{(vt-u)^2}{(1-u^2)(1-v^2)^3}}}d\ell_{x}\end{equation*}
with $d\ell_x$ the measure on $\cal{R}_{t,x}$ induced from the pullback of the Euclidean metric on $[-1,1]^2$.
\end{proposition}
\begin{proof}
Consider the function $f_t:[-1,1]^2\rightarrow [-1,1]$ given by
\[f_t(u,v):=\frac{t-uv}{\sqrt{(1-u^2)(1-v^2)}}.\]
The function $f_t$ satisfies
\begin{eqnarray*}|\grad f_t|^2&=&\frac{(ut-v)^2}{(1-u^2)^3(1-v^2)}+\frac{(vt-u)^2}{(1-u^2)(1-v^2)^3}.
\end{eqnarray*}
An application of the co-area formula gives
\begin{eqnarray*}
&&\iint_{\cal{R}_t}g\left(\frac{t-uv}{\sqrt{(1-u^2)(1-v^2)}}\right)F(u)F(v)\left((1-u^2)(1-v^2)-(t-uv)^2\right)^{\frac{n-4}{2}}d\boldsymbol{u}d\boldsymbol{v}\\
&=&\int_{-1}^1g\left(x\right)(1-x^2)^{\frac{n-4}{2}}\int_{\cal{R}_{t,x}}F(u)F(v)\frac{\left((1-u^2)(1-v^2)\right)^{\frac{n-4}{2}}}{|\grad f_t|}d\ell_{x}dx\\
&=&\int_{-1}^1g\left(x\right)\left(\frac{1-x^2}{x^2}\right)^{\frac{n-4}{2}}\int_{\cal{R}_{t,x}}F(u)F(v)\frac{\left(t-uv\right)^{\frac{n-4}{2}}}{|\grad f_t|}d\ell_{x}dx
\end{eqnarray*}
where
\[\cal{R}_{t,x}:=\left\{(u,v)\in[-1,1]^2:\frac{t-uv}{\sqrt{(1-u^2)(1-v^2)}}=x\right\}\]
and $d\ell_x$ is the area measure on $\cal{R}_{t,x}$ obtained from the pullback of the Euclidean metric on $[-1,1]$. 
We let
\[\nu_1(x;t,F):=\left(\frac{1-x^2}{x^2}\right)^{\frac{n-4}{2}}\int_{\cal{R}_{t,x}}F(u)F(v)\frac{\left(t-uv\right)^{\frac{n-4}{2}}}{|\grad f_t|}d\ell_{x},\]
from which the conclusion follows.
\end{proof}
For the special parameters chosen in~\cite{SZ}, we could estimate the conditional density $\nu_1(x;t,F)$ for $|x-s'|=o(n^{-\frac{1}{2}})$, where $s':=\cos(\theta')$. Also, let $s:=\cos(\theta)$. Since we will give two proofs of Theorem~\ref{THMtwo}, one that still uses the estimation of the conditional densities, we recall the estimate proved in~\cite{SZ}.\\
\\
For angles $0<\theta<\theta'\leq\frac{\pi}{2}$, we took $F=\chi_{[r-\delta,R]}$ for an appropriately chosen $\delta\in [0,r]$, where
\begin{equation*}\label{littler}r=\sqrt{\frac{s-s'}{1-s'}},\end{equation*}
and
\begin{equation*}\label{bigR}R=\cos\left(2\arctan\left(\frac{s}{\sqrt{(1-s)(s-s')}}\right)+\arccos(r)-\pi\right).\end{equation*}
$R>r$ and both $r,R$ are roots of the equation
\begin{equation*}\frac{s-rx}{\sqrt{(1-r^2)(1-x^2)}}=s'\end{equation*}
in $x$. In the above, $r,R$ are chosen so that 
\begin{equation*}\label{condition}\frac{s-uv}{\sqrt{(1-u^2)(1-v^2)}}\leq s'\end{equation*}
for every $r\leq u,v\leq R$. We will use the following estimate. The proof may be found in~\cite{SZ} as Proposition 4.2. There is also a version of this estimate for $\nu_m(x;t,F)$. However, we do not include it in this paper. In the following and the rest of the paper, for every $x\in\mathbb{C}$, we define
\[x^+:=\begin{cases}x&\textup{ if }x\in\mathbb{R}_{\geq 0}\\ 0 &\textup{ otherwise.}\end{cases}\]
\begin{proposition}\label{nuFestimate}
For $|x-s'|=o\left(n^{-\frac{1}{2}}\right)$ as $n\rightarrow\infty$, we have up to a positive multiplicative constant depending on $n,R,r,\delta,s$, that $\nu_1(x;s,\chi_{[r-\delta,R]})$ is equal to
\[(1+o(1))\left(\frac{1-x^2}{x^2}\right)^{\frac{n-4}{2}}\left(\delta+\sqrt{\frac{s-x}{1-x}}-r\right)^+e^{-\frac{2nr\left(\sqrt{\frac{s-x}{1-x}}-r\right)}{s-r^2}}.\]
\end{proposition}

\subsection{Averaging over $\mathbb{R}^n$}
In the case of sphere packings, we average over $\mathbb{R}^n$; see Proposition~\ref{38}. For this setting too, we will change coordinates. In this case, we are given an angle $0<\theta\leq\frac{\pi}{2}$, a function $g:[-1,1]\rightarrow\mathbb{R}$ satisfying the Delsarte linear programming conditions for $M(n,\theta)$, and a function $F:[-1,1]\rightarrow\mathbb{R}$. From this data, we defined the function $H:\mathbb{R}^n\times\mathbb{R}^n\rightarrow\mathbb{R}$ given by
\begin{equation}\label{Hspheresubsection}
H(\boldsymbol{x},\boldsymbol{y}):=\int_{\mathbb{R}^n}F(|\boldsymbol{x}-\boldsymbol{z}|)F(|\boldsymbol{y}-\boldsymbol{z}|)g\left(\left<\frac{\boldsymbol{x}-\boldsymbol{z}}{|\boldsymbol{x}-\boldsymbol{z}|},\frac{\boldsymbol{y}-\boldsymbol{z}}{|\boldsymbol{y}-\boldsymbol{z}|}\right>\right)d\boldsymbol{z}.
\end{equation}
Recall that $H$ is a point-pair invariant function, and so is only a function of $T:=|\boldsymbol{x}-\boldsymbol{y}|$. We could also view $H$ as radial function on $\mathbb{R}^n$. By Section 4.2 of~\cite{SZ}, in particular equation (33) and Lemmas 4.3 and 4.4 therein, if we let $r:=\frac{1}{2\sin(\theta/2)}$ for $\frac{\pi}{3}\leq\theta\leq\pi$ and $F=\chi_{[0,r+\delta]}$, then
\begin{equation}\label{HT1}
H(T)=T^{2n-1}\int_{-1}^1g(x)\mu(x;\chi_{[0,r+\delta]}(x/T))dx,
\end{equation} 
where $\mu$ is given by the following proposition.
\begin{proposition}\label{spheredensity} When $|x-\cos(\theta)|=o(n^{-1/2})$, $x\leq\frac{1}{2}$, and $\delta=o(n^{-1/2})$, there is a constant $C_{n,\delta,r}>0$ such that
\begin{equation*}
\mu(x;\chi_{[0,r+\delta]})=C_{n,\delta,r}(1+o(1))\frac{(1-x^2)^{\frac{n-4}{2}}}{(1-x)^{n-\frac{1}{2}}}\sqrt{1+\left(x-\frac{(1-x^2)r}{\sqrt{1-(1-x^2)r^2}}\right)^2}\left(r+\delta-\sqrt{1-(1-x^2)(r+\delta)^2}-x(r+\delta)\right)^+.
\end{equation*}
\end{proposition}
In this proposition, we took $T=1$. This is a minor modification of Proposition 4.4 of~\cite{SZ}.

\section{Estimating Levenshtein's optimal polynomials}\label{Jacobiestimates}
In order to prove our main results, we need to estimate Levenshtein's optimal polynomials. In Proposition 5.1 of~\cite{SZ}, the author and Sardari used the differential and interlacing properties of Jacobi polynomials to estimate Jacobi polynomials near their largest roots from a Taylor expansion. However, those estimates are not sharp enough for the results of this paper. We will use a different and sharp estimate of Jacobi polynomials to determine sufficient conditions under which the first Delsarte linear programming condition, $H(t)\leq 0$ for every $t\in[-1,\cos\theta]$, holds when $g$ is chosen to be Levenshtein's optimal polynomials.\\
\\
We first recall Jacobi polynomials and some of their properties. Given $\alpha,\beta>-1$, the Jacobi polynomials $\{p_d^{\alpha,\beta}(t)\}_{d\geq 0}$ of different degrees $d$ are orthogonal with respect to the inner product on $\mathbb{R}[t]$ given by
\begin{equation*}\left<f,g\right>_{L^2(\mu_{\alpha,\beta})}:=\int_{-1}^1f(t)g(t)\mu_{\alpha,\beta}(dt),\end{equation*}
where
\begin{equation*}\mu_{\alpha,\beta}(dt):=(1-t)^{\alpha}(1+t)^{\beta}dt.\end{equation*}
We denote the $L^2$-norm associated to this inner product by $\|\cdot\|_2$. We normalization our Jacobi polynomials so that
\begin{equation*}p_d^{\alpha,\beta}(1)=\binom{d+\alpha}{d}\end{equation*}
for every $d$. Throughout this paper,
\begin{equation*}\binom{a+b}{b}:=\frac{\Gamma(a+b+1)}{\Gamma(a+1)\Gamma(b+1)},\end{equation*}
where
\begin{equation*}\Gamma(z):=\int_0^{\infty}t^{z-1}e^{-t}dt,\ \textup{Re}(z)>0,\end{equation*}
is the Gamma function. By the general theory of orthogonal polynomials, the Jacobi polynomials $p_d^{\alpha,\beta}$ have $d$ simple real roots $t_{1,d}^{\alpha,\beta}>t_{2,d}^{\alpha,\beta}>\hdots>t_{d,d}^{\alpha,\beta}$. We will use the following estimate on Jacobi polynomials, a result due to Krasikov and a part of Theorem 1 of his~\cite{IK}.
\begin{lemma}[Krasikov~\cite{IK}]\label{Krasikovlemma}Suppose $d\geq 2$ is an even integer and $\alpha\geq\sqrt{\frac{7}{6}}$. Let
\begin{equation}\label{fda}f_{d,\alpha}(x):=\left((1-q_{d,\alpha}-x^2)(d+\alpha)(d+\alpha+1)\right)^{\frac{1}{4}}(1-x^2)^{\frac{\alpha}{2}}p_d^{\alpha,\alpha}(x),\end{equation}
where $q_{d,\alpha}:=\frac{\alpha^2-1}{(d+\alpha)(d+\alpha+1)}$. For $0\leq x\leq\sqrt{1-q_{d,\alpha}}$, we have
\begin{equation*}\label{fdaequal}f_{d,\alpha}(x)=f_{d,\alpha}(0)(\cos(\omega_{d,\alpha}(x))+R_{d,\alpha}(x)),\end{equation*}
where
\begin{equation}\label{omegada}\omega_{d,\alpha}(x):=\sqrt{(d+\alpha)(d+\alpha+1)}\left(\arccos\sqrt{\frac{1-q_{d,\alpha}-x^2}{1-q_{d,\alpha}}}-\sqrt{q_{d,\alpha}}\arccos\sqrt{\frac{1-q_{d,\alpha}-x^2}{(1-q_{d,\alpha})(1-x^2)}}\right)\end{equation}
and
\begin{equation}\label{Rda}|R_{d,\alpha}(x)|<\begin{cases}\frac{2(1-x^2)x}{(1-q_{d,\alpha})(1-q_{d,\alpha}-x^2)^{\frac{3}{2}}\sqrt{(d+\alpha)(d+\alpha+1)}}&\textup{ if }q_{d,\alpha}\in\left[0,\frac{1}{2}\right)\\\frac{(1+q_{d,\alpha})x}{4(1-q_{d,\alpha}-x^2)^{\frac{3}{2}}\sqrt{(d+\alpha)(d+\alpha+1)}}&\textup{ if }q_{d,\alpha}\in\left[\frac{1}{2},1\right).\end{cases}\end{equation}
Furthermore,
\[f_{d,\alpha}(0)=\left(\frac{-1}{4}\right)^{\frac{d}{2}}\binom{d+\alpha}{\frac{d}{2}}(d^2+2d\alpha+d+\alpha+1)^{\frac{1}{4}}.\]
\end{lemma}
The polynomials $p_d^{\alpha,\alpha}$ satisfy the recurrence relation
\begin{equation}\label{recurrence}
xp_d^{\alpha,\alpha}(x)=\frac{(d+1)p_{d+1}^{\alpha,\alpha}(x)+(d+2\alpha)p_{d-1}^{\alpha,\alpha}(x)}{2d+2\alpha+1},
\end{equation}
while the polynomials $p_d^{\alpha+1,\alpha}$ satisfy the recurrence relation
\begin{equation}\label{recurrence2}
p_d^{\alpha+1,\alpha}(x)=\frac{(d+2\alpha+2)p_d^{\alpha+1,\alpha+1}(x)+(d+\alpha+1)p_{d-1}^{\alpha+1,\alpha+1}(x)}{2(d+\alpha+1)}.
\end{equation}
As a result, one could use Lemma~\ref{Krasikovlemma} for even $d+1$ and $d-1$ along with recurrence relation~\eqref{recurrence} to estimate $p_d^{\alpha,\alpha}(x)$ for odd $d$. One could then use~\eqref{recurrence2} to estimate $p_d^{\alpha+1,\alpha}$.\\
\\
The arguments within~\cite{IK} are equally important to us. A couple important facts are that
\begin{equation}\label{omegader}
\omega'_{d,\alpha}(x)=b_{d,\alpha}(x):=\frac{\sqrt{1-q_{d,\alpha}-x^2}\sqrt{(d+\alpha)(d+\alpha+1)}}{1-x^2},
\end{equation}
and the remainder satisfies
\begin{equation}\label{Rintegral} R_{d,\alpha}(x)=\frac{1}{f_{d,\alpha}(0)}\int_0^x\varepsilon_{d,\alpha}(t)b_{d,\alpha}(t)f_{d,\alpha}(t)\sin(\omega_{d,\alpha}(x)-\omega_{d,\alpha}(t))dt,\end{equation}
where
\begin{equation*}
\varepsilon_{d,\alpha}(x):=\frac{3(b'_{d,\alpha}(x))^2-2b_{d,\alpha}(x)b''_{d,\alpha}(x)}{4b^4_{d,\alpha}(x)}=-\frac{x^6+6q_{d,\alpha}x^4-3x^2+4q_{d,\alpha}^2-6q_{d,\alpha}+2}{4(d+\alpha)(d+\alpha+1)(1-q_{d,\alpha}-x^2)^3}
\end{equation*}
as given in equations (11) and (14) of~\cite{IK}. We also define
\begin{equation}\label{Tint}
T_{d,\alpha}(x):=\frac{1}{f_{d,\alpha}(0)}\int_0^{x}\varepsilon_{d,\alpha}(t)b_{d,\alpha}(t)f_{d,\alpha}(t)\cos(\omega_{d,\alpha}(x)-\omega_{d,\alpha}(t))dt
\end{equation}
to be used later. The following lemma is important to us.
\begin{lemma}\label{diffroots}
Let $d\geq 5$ and $\alpha>0$. Then
\begin{equation*}t_{1,d}^{\alpha,\alpha}<\sqrt{1-q_{d,\alpha}}\left(1-\frac{3(\alpha+1)^{\frac{4}{3}}}{2d^{\frac{2}{3}}(d+\alpha+1)^{\frac{2}{3}}(d+2\alpha+1)^{\frac{2}{3}}}\right).\end{equation*}
\end{lemma}
\begin{proof}
I. Krasikov proved sharp estimates of $t_{1,d}^{\alpha,\alpha}$ in~\cite{IK3} and~\cite{IK4}. He proved that when $\alpha>-1$ and $d\geq 5$,
\[t_{1,d}^{\alpha,\alpha}=\sqrt{\frac{d(d+2\alpha+1)}{(d+\alpha+1)^2-d}}-\frac{3(\alpha+1)^{\frac{4}{3}}(1+2\psi)}{2d^{\frac{1}{6}}(d+2\alpha+1)^{\frac{1}{6}}((d+\alpha+1)^2-d)^{\frac{5}{6}}}\]
for some $\psi\in(0,1)$. In particular,
\[t_{1,d}^{\alpha,\alpha}<\sqrt{\frac{d(d+2\alpha+1)}{(d+\alpha+1)^2-d}}\left(1-\frac{3(\alpha+1)^{\frac{4}{3}}}{2d^{\frac{2}{3}}(d+\alpha+1)^{\frac{2}{3}}(d+2\alpha+1)^{\frac{2}{3}}}\right).\]
Furthermore,
\[\sqrt{\frac{d(d+2\alpha+1)}{(d+\alpha+1)^2-d}}< \sqrt{1-q_{d,\alpha}}\]
follows from
\[1-q_{d,\alpha}-\frac{d(d+2\alpha+1)}{(d+\alpha+1)^2-d}> 1-q_{d,\alpha}-\frac{d(d+2\alpha+1)}{(d+\alpha)(d+\alpha+1)}=\frac{\alpha}{(d+\alpha)(d+\alpha+1)}>0,\]
where the last inequality is a consequence of $\alpha>0$. The conclusion follows.
\end{proof}
Lemma~\ref{diffroots} bounds from below the distance between $\sqrt{1-q_{d,\alpha}}$ and $t_{1,d}^{\alpha,\alpha}$. For the application that we have in mind, $\alpha$ grows linearly in $d$, with both $\sqrt{1-q_{d,\alpha}}$ and $t_{1,d}^{\alpha,\alpha}$ converging to the constant $\cos(\theta')>0$. Therefore, the estimate in Lemma~\ref{Krasikovlemma} is valid $\Theta(d^{-2/3})$ to the right of $t_{1,d}^{\alpha,\alpha}$.\\
\\
We will also use the following result of Krasikov in our argument for the optimality of $\frac{1}{e}$ in Theorems~\ref{THMone} and~\ref{THMtwo} as well as extend the domain of validity of our estimates of Levenshtien's optimal polynomials.
\begin{lemma}[Krasikov~\cite{IK2}]\label{Krasikovupperbound}For $d\geq 6$ and $\alpha,\beta\geq \frac{1+\sqrt{2}}{4}$,
\begin{equation*}\max_{x\in[-1,1]}(1-x)^{\alpha+\frac{1}{2}}(1+x)^{\beta+\frac{1}{2}}(p_d^{\alpha,\beta}(x))^2< 3\|p_d^{\alpha,\beta}\|_2^2\alpha^{\frac{1}{3}}\left(1+\frac{\alpha}{d}\right)^{\frac{1}{6}}.\end{equation*}
\end{lemma}
This is a sharpening of such a result first proved by Erd\'elyi--Magnus--Nevai~\cite{EMN} from 1994. Their bound was $\frac{e\|p_d^{\alpha,\beta}\|_2^2(2+\sqrt{\alpha^2+\beta^2})}{\pi}$ but valid for $d\geq 0$ and $\alpha,\beta\geq-\frac{1}{2}$. For our proof of optimality, we could replace Lemma~\ref{Krasikovupperbound} with this inequality of~\cite{EMN}. For optimality, we also record the following comparison.
\begin{lemma}\label{magnitudecomp}As $\alpha,d\rightarrow\infty$ with $d$ even,
\begin{equation*}\frac{f_{d,\alpha}(0)^2}{\|p_d^{\alpha,\alpha}\|_2^2}=(1+o(1))\frac{\sqrt{(d+\alpha)(d+\alpha+1)}(2d+2\alpha+1)}{\pi\sqrt{d(d+2\alpha)}}.\end{equation*}
\end{lemma}
\begin{proof}
By Lemma~\ref{Krasikovlemma},
\[f_{d,\alpha}(0)^2=\left(\frac{1}{4}\right)^{d}\binom{d+\alpha}{\frac{d}{2}}^2((d+\alpha)(d+\alpha+1))^{\frac{1}{2}}.\]
On the other hand, it is well-known that
\[\|p_d^{\alpha,\beta}\|_2^2=\frac{2^{\alpha+\beta+1}}{2d+\alpha+\beta+1}\frac{\Gamma(d+\alpha+1)\Gamma(d+\beta+1)}{\Gamma(d+\alpha+\beta+1)d!},\]
from which we obtain that as $\alpha,d\rightarrow\infty$,
\[\|p_d^{\alpha,\alpha}\|_2^2=(1+o(1))\frac{\pi\sqrt{d(d+2\alpha)}}{2d+2\alpha+1}\frac{1}{4^d}\binom{d+\alpha}{\frac{d}{2}}^2.\]
These are based on the normalization $p_d^{\alpha,\beta}(1)=\binom{d+\alpha}{d}$. The estimate is a consequence of Stirling's formula giving
\[\binom{2n}{n}=(1+o(1))\frac{2^{2n}}{\sqrt{\pi n}}\textup{ as }n\rightarrow\infty.\]
Therefore, as $\alpha,d\rightarrow\infty$ with $d$ even,
\[\frac{f_{d,\alpha}(0)^2}{\|p_d^{\alpha,\alpha}\|_2^2}=(1+o(1))\frac{\sqrt{(d+\alpha)(d+\alpha+1)}(2d+2\alpha+1)}{\pi\sqrt{d(d+2\alpha)}}.\]
\end{proof}
\begin{corollary}\label{corollary:magnitudecomp}
If $\frac{\alpha}{d}\rightarrow\frac{1}{2\rho}$ as $d\rightarrow\infty$ with $d$ even, then 
\begin{equation*}\frac{f_{d,\alpha}(0)^2}{d}=(1+o(1))\frac{2\|p_d^{\alpha,\alpha}\|_2^2\left(1+\frac{1}{2\rho}\right)^2}{\pi\sqrt{1+\frac{1}{\rho}}}.\end{equation*}
\end{corollary}
Later, we will estimate Levenshtein's optimal polynomials near their largest roots. For that, we will apply Lemma~\ref{Krasikovlemma} to estimate Jacobi polynomials near $t_{1,d}^{\alpha,\alpha}$. The following lemma will be used in our estimation.
\begin{lemma}\label{omegaTaylor}Assume the conditions and notations of Lemma~\ref{Krasikovlemma}. Then for $s'=t_{1,d}^{\alpha,\alpha}$ and $\alpha,d=O(n)$,
\begin{equation*}\omega_{d,\alpha}(x)-\omega_{d,\alpha}(s')=\frac{\sqrt{(d+\alpha)(d+\alpha+1)}\sqrt{1-q_{d,\alpha}-s'^2}}{1-s'^2}(x-s')+\frac{s'\sqrt{(d+\alpha)(d+\alpha+1)}\left(1-2q_{d,\alpha}-s'^2\right)}{2\sqrt{1-q_{d,\alpha}-s'^2} \left(1-s'^2\right)^{2}}(x-s')^2+O(n^2(x-s')^3).\end{equation*}
\end{lemma} 
\begin{proof}
By~\eqref{omegada},
\[\omega_{d,\alpha}(x)=\sqrt{(d+\alpha)(d+\alpha+1)}\left(\arccos\sqrt{\frac{1-q_{d,\alpha}-x^2}{1-q_{d,\alpha}}}-\sqrt{q_{d,\alpha}}\arccos\sqrt{\frac{1-q_{d,\alpha}-x^2}{(1-q_{d,\alpha})(1-x^2)}}\right).\]
We write the Taylor expansion of part of this expression around $s'$ to get
\begin{eqnarray*}&&\left(\arccos\sqrt{\frac{1-q_{d,\alpha}-x^2}{1-q_{d,\alpha}}}-\sqrt{q_{d,\alpha}}\arccos\sqrt{\frac{1-q_{d,\alpha}-x^2}{(1-q_{d,\alpha})(1-x^2)}}\right)-\left(\arccos\sqrt{\frac{1-q_{d,\alpha}-s'^2}{1-q_{d,\alpha}}}-\sqrt{q_{d,\alpha}}\arccos\sqrt{\frac{1-q_{d,\alpha}-s'^2}{(1-q_{d,\alpha})(1-s'^2)}}\right)\\&=& \frac{\sqrt{1-q_{d,\alpha}-s'^2}}{1-s'^2}(x-s')+\frac{s' \left(1-2q_{d,\alpha}-s'^2\right)}{2\sqrt{1-q_{d,\alpha}-s'^2} \left(1-s'^2\right)^{2}}(x-s')^2+O(n(x-s')^3),
\end{eqnarray*}
from which we obtain
\[\omega_{d,\alpha}(x)-\omega_{d,\alpha}(s')=\frac{\sqrt{(d+\alpha)(d+\alpha+1)}\sqrt{1-q_{d,\alpha}-s'^2}}{1-s'^2}(x-s')+\frac{s'\sqrt{(d+\alpha)(d+\alpha+1)}\left(1-2q_{d,\alpha}-s'^2\right)}{2\sqrt{1-q_{d,\alpha}-s'^2} \left(1-s'^2\right)^{2}}(x-s')^2+O(n^2(x-s')^3),\]
as required.
\end{proof}
In the following, we estimate the remainder term $R_{d,\alpha}$ in Lemma~\ref{Krasikovlemma} near $t^{\alpha,\alpha}_{1,d}$ .
\begin{lemma}\label{RTaylor}Assume the conditions and notations of Lemma~\ref{Krasikovlemma}, and let $s'=t_{1,d}^{\alpha,\alpha}$. Recall the definitions of $b_{d,\alpha}(x)$ and $T_{d,\alpha}(x)$ given, respectively, in~\eqref{omegader} and~\eqref{Tint}. Assume furthermore that $\alpha,d$ grow linearly in $n$ with $\frac{\alpha}{d}\rightarrow\frac{1}{2\rho}$. Then for $|x-s'|=o(n^{-2/3})$,
\begin{equation*}\label{Rdaestimate}R_{d,\alpha}(x)-R_{d,\alpha}(s')=(x-s')b_{d,\alpha}(s')(T_{d,\alpha}(s')+o(1)).\end{equation*}
\end{lemma}
\begin{proof}We use Taylor's theorem to estimate $R_{d,\alpha}(x)$ near $x=s'$. From~\eqref{Rintegral}, we know that
\begin{equation*}
R_{d,\alpha}(x)=\frac{1}{f_{d,\alpha}(0)}\int_0^x\varepsilon_{d,\alpha}(t)b_{d,\alpha}(t)f_{d,\alpha}(t)\sin(\omega_{d,\alpha}(x)-\omega_{d,\alpha}(t))dt.
\end{equation*}
This along with $\omega_{d,\alpha}'(x)=b_{d,\alpha}(x)$ imply that
\begin{equation}\label{Rprimeint}
R'_{d,\alpha}(x)=\frac{b_{d,\alpha}(x)}{f_{d,\alpha}(0)}\int_0^x\varepsilon_{d,\alpha}(t)b_{d,\alpha}(t)f_{d,\alpha}(t)\cos(\omega_{d,\alpha}(x)-\omega_{d,\alpha}(t))dt.
\end{equation}
By Taylor's theorem,
\begin{equation}\label{TaylorexpansionR}R_{d,\alpha}(x)-R_{d,\alpha}(s')=R'_{d,\alpha}(s')(x-s')+\frac{R''_{d,\alpha}(\xi)}{2}(x-s')^2\end{equation}
for some $\xi$ between $x$ and $s'$.
Furthermore,
\begin{eqnarray*}
R''_{d,\alpha}(x)&=&\frac{b'_{d,\alpha}(x)}{f_{d,\alpha}(0)}\int_0^x\varepsilon_{d,\alpha}(t)b_{d,\alpha}(t)f_{d,\alpha}(t)\cos(\omega_{d,\alpha}(x)-\omega_{d,\alpha}(t))dt\\&&+\frac{b_{d,\alpha}(x)}{f_{d,\alpha}(0)}\left(\varepsilon_{d,\alpha}(x)b_{d,\alpha}(x)f_{d,\alpha}(x)-b_{d,\alpha}(x)\int_0^x\varepsilon_{d,\alpha}(t)b_{d,\alpha}(t)f_{d,\alpha}(t)\sin(\omega_{d,\alpha}(x)-\omega_{d,\alpha}(t))dt\right)\\
&=& \frac{b'_{d,\alpha}(x)}{f_{d,\alpha}(0)}\int_0^x\varepsilon_{d,\alpha}(t)b_{d,\alpha}(t)f_{d,\alpha}(t)\cos(\omega_{d,\alpha}(x)-\omega_{d,\alpha}(t))dt+\frac{b^2_{d,\alpha}(x)}{f_{d,\alpha}(0)}\varepsilon_{d,\alpha}(x)f_{d,\alpha}(x)-b^2_{d,\alpha}(x)R_{d,\alpha}(x)\\
&=&b'_{d,\alpha}(x)T_{d,\alpha}(x)+\frac{b^2_{d,\alpha}(x)}{f_{d,\alpha}(0)}\varepsilon_{d,\alpha}(x)f_{d,\alpha}(x)-b^2_{d,\alpha}(x)R_{d,\alpha}(x).
\end{eqnarray*}

By Lemma~\ref{Krasikovlemma}, for $|x-s'|=o(n^{-2/3})$, $|R_{d,\alpha}(x)|=O(1)$. The bound on the remainder $R_{d,\alpha}$ in Lemma~\ref{Krasikovlemma} was proved in~\cite{IK} by first applying the triangle inequality to obtain
\[|R_{d,\alpha}(x)|\leq \frac{1}{|f_{d,\alpha}(0)|}\int_0^x|\varepsilon_{d,\alpha}(t)b_{d,\alpha}(t)f_{d,\alpha}(t)|dt.\]
As a result, we also have $|T_{d,\alpha}(x)|=O(1)$ for $|x-s'|=o(n^{-2/3})$. This,~\eqref{Rprimeint}, and the above calculation imply that $|R''_{d,\alpha}(x)|=O(n^{4/3})$ for $|x-s'|=o(n^{-2/3})$. Applying~\eqref{TaylorexpansionR} we obtain that for $|x-s'|=o(n^{-2/3})$,
\[R_{d,\alpha}(x)-R_{d,\alpha}(s')=(x-s')b_{d,\alpha}(s')\left(T_{d,\alpha}(s')+o(1)\right),\]
as required.
\end{proof}
For the calculations of improvement factors for functions constructed using Stiefel manifolds, we will apply the above estimates to the following setting. We will take $g=g_{n-m,\theta'}$ to be Levenshtein's optimal polynomial for $M(n-m,\theta')$ given by
\begin{equation}\label{gnthetageneral}
g(x)=\begin{cases}
\frac{(x+1)^2}{(x-t^{\alpha+1,\alpha+1}_{1,d})}\left(  p^{\alpha+1,\alpha+1}_{d}(x) \right)^2  &\text{ if } t^{\alpha+1,\alpha}_{1,d}< s' \leq t^{\alpha+1,\alpha+1}_{1,d} ,
\\
\frac{(x+1)}{(x-t^{\alpha+1,\alpha}_{1,d})}\left(  p^{\alpha+1,\alpha}_{d}(x) \right)^2  &\text{ if } t^{\alpha+1,\alpha+1}_{1,d-1}< s' \leq t^{\alpha+1,\alpha}_{1,d}.
\end{cases}
\end{equation}
By the works of Levenshtein~\cite{Lev79},~\cite{Lev98}, this satisfies the Delsarte linear programming conditions for $M(n-m,\theta')$ and gives the bound $M(n-m,\theta')\leq M_{\textup{Lev}}(n-m,\theta')$. We address the second case; the first case may be addressed in a similar way. Therefore, we take
\begin{equation}\label{gchoice}g(x)=\frac{(x+1)^2}{(x-s')}\left(  p^{\alpha+1,\alpha+1}_{d}(x) \right)^2\end{equation}
with $s'=t^{\alpha+1,\alpha+1}_{1,d}$. Assuming the notations in the construction of $H$, we will have $\alpha=\frac{n-m-3}{2}$, and $d$ uniquely determined by the inequality $t^{\alpha+1,\alpha+1}_{1,d-1}<\cos\theta'\leq t^{\alpha+1,\alpha+1}_{1,d}$. In fact, we will require $\frac{\alpha}{d}\rightarrow \frac{1}{2\rho}$ as $n\rightarrow\infty$, where $\rho:=\frac{1-\sin(\theta')}{2\sin(\theta')}$, in which case $t^{\alpha+1,\alpha+1}_{1,d}\rightarrow\cos(\theta')$ as $n\rightarrow\infty$. By taking $s'=t^{\alpha+1,\alpha+1}_{1,d}$ with $\frac{\alpha}{d}\rightarrow \frac{1}{2\rho}$,
\begin{equation}\label{qlim}q_{d,\alpha+1}=\frac{\left(\frac{\alpha}{d}\right)^2-\frac{1}{d^2}}{\left(1+\frac{\alpha}{d}\right)\left(1+\frac{\alpha}{d}+\frac{1}{d}\right)}\rightarrow \frac{1}{(1+2\rho)^2}=\sin^2(\theta')\textup{ as } n\rightarrow\infty,\end{equation}
from which it follows that $\sqrt{1-q_{d,\alpha+1}}\rightarrow\cos(\theta')$ as $n\rightarrow\infty$. We will use the above with $m=1$ in our proof of Theorem~\ref{THMtwo}. For Theorem~\ref{THMone}, we again use the above but with $m=0$ and $\theta'$ replaced by $\theta\in\left[\frac{\pi}{3},\frac{\pi}{2}\right)$. Applying the above lemmas to this setting, we have the following useful proposition.
\begin{proposition}\label{gcorollary}Assuming that $\alpha,d,n,\rho$ are chosen as above with $d$ even and $s'=t_{1,d}^{\alpha+1,\alpha+1}$ with $\frac{\alpha}{d}\rightarrow\frac{1}{2\rho}$ as $n\rightarrow\infty$, we have for $|x-s'|=o(n^{-2/3})$,
\begin{equation*}(1-x^2)^{\alpha}g(x)=(1+o(1))\frac{(1+s')}{(1-s')\left(1-q_{d,\alpha+1}-s'^2\right)^{1/2}}\frac{2\|p_d^{\alpha+1,\alpha+1}\|_2^2\left(1+\frac{1}{2\rho}\right)}{\pi\sqrt{(1+\frac{1}{\rho})}}b^2_{d,\alpha+1}(s')\left(\sin\left(\omega_{d,\alpha+1}(s')\right)-T_{d,\alpha+1}(s')+o(1)\right)^2(x-s')\end{equation*}
\end{proposition}
\begin{proof}By definition~\eqref{gchoice} of $g$, and by definition~\eqref{fda},
\begin{eqnarray*}(1-x^2)^{\alpha}g(x)&=&\frac{(x+1)^2}{(x-s')}(1-x^2)^{\alpha}\left(p^{\alpha+1,\alpha+1}_{d}(x)\right)^2\\
&=&\frac{(x+1)^2}{(x-s')(1-x^2)(1-q_{d,\alpha+1}-x^2)^{1/2}(d+\alpha+1)^{1/2}(d+\alpha+2)^{1/2}}(f_{d,\alpha+1}(x))^2\\
&=&\frac{(x+1)^2}{d(x-s')(1-x^2)\left(1-q_{d,\alpha+1}-x^2\right)^{1/2}\left(1+\frac{\alpha+1}{d}\right)^{1/2}\left(1+\frac{\alpha+2}{d}\right)^{1/2}}(f_{d,\alpha+1}(x))^2,
\end{eqnarray*}
where $q_{d,\alpha+1}=\frac{(\alpha+1)^2-1}{(d+\alpha+1)(d+\alpha+2)}$. Using $\frac{\alpha}{d}=\frac{1}{2\rho}(1+o(1))$, $|x-s'|=o(n^{-2/3})$, and Lemma~\ref{diffroots}, we obtain
\[(1-x^2)^{\alpha}g(x)=(1+o(1))\frac{(1+s')}{d(1-s')\left(1+\frac{1}{2\rho}\right)\left(1-q_{d,\alpha+1}-s'^2\right)^{1/2}}(x-s')\left(\frac{f_{d,\alpha+1}(x)}{x-s'}\right)^2.\]
By Lemma~\ref{Krasikovlemma}, since $d$ is even, for every $0\leq x\leq\sqrt{1-q_{d,\alpha+1}}$,
\[f_{d,\alpha+1}(x)=f_{d,\alpha+1}(0)(\cos(\omega_{d,\alpha+1}(x))+R_{d,\alpha+1}(x)).\]
Therefore,
\begin{equation}\label{modifiedg}(1-x^2)^{\alpha}g(x)=(1+o(1))\frac{(1+s')f_{d,\alpha+1}(0)^2}{d(1-s')\left(1+\frac{1}{2\rho}\right)\left(1-q_{d,\alpha+1}-s'^2\right)^{1/2}}(x-s')\left(\frac{(\cos(\omega_{d,\alpha+1}(x))+R_{d,\alpha+1}(x))}{x-s'}\right)^2.\end{equation}
Note that $g(s')=0$ and so $\cos(\omega_{d,\alpha+1}(s'))+R_{d,\alpha+1}(s')=0$. Therefore, we rewrite
\begin{eqnarray*}
&&\frac{\cos(\omega_{d,\alpha+1}(x))+R_{d,\alpha+1}(x)}{x-s'}\\&=&\frac{(\cos(\omega_{d,\alpha+1}(x))-\cos(\omega_{d,\alpha+1}(s')))+(R_{d,\alpha+1}(x)-R_{d,\alpha+1}(s'))}{x-s'}\\
&=& \frac{-2\sin\left(\frac{\omega_{d,\alpha+1}(x)+\omega_{d,\alpha+1}(s')}{2}\right)\sin\left(\frac{\omega_{d,\alpha+1}(x)-\omega_{d,\alpha+1}(s')}{2}\right)+(R_{d,\alpha+1}(x)-R_{d,\alpha+1}(s'))}{x-s'}.
\end{eqnarray*}
Applying Lemmas~\ref{omegaTaylor} and~\ref{RTaylor}, we get
\[\frac{R_{d,\alpha+1}(x)-R_{d,\alpha+1}(s')}{x-s'}=b_{d,\alpha+1}(s')(T_{d,\alpha+1}(s')+o(1)),\]
\[\frac{\sin\left(\frac{\omega_{d,\alpha+1}(x)-\omega_{d,\alpha+1}(s')}{2}\right)}{x-s'}=(1+o(1))\frac{\sqrt{(d+\alpha+1)(d+\alpha+2)}\sqrt{1-q_{d,\alpha+1}-s'^2}}{2(1-s'^2)}=(1+o(1))\frac{b_{d,\alpha+1}(s')}{2},\]
and
\[\sin\left(\frac{\omega_{d,\alpha+1}(x)+\omega_{d,\alpha+1}(s')}{2}\right)=(1+o(1))\sin(\omega_{d,\alpha+1}(s')).\]
As a result,
\begin{equation}\label{quotcalc}\frac{\cos(\omega_{d,\alpha+1}(x))+R_{d,\alpha+1}(x)}{x-s'}=-(1+o(1))b_{d,\alpha+1}(s')\left(\sin\left(\omega_{d,\alpha+1}(s')\right)-T_{d,\alpha+1}(s')+o(1)\right).\end{equation}
Equations~\eqref{modifiedg} and~\eqref{quotcalc} combine to give that for $|x-s'|=o(n^{-2/3})$,
\begin{equation}\label{modifiedg2}(1-x^2)^{\alpha}g(x)=(1+o(1))\frac{(1+s')f_{d,\alpha+1}(0)^2}{d(1-s')\left(1+\frac{1}{2\rho}\right)\left(1-q-s'^2\right)^{1/2}}b^2_{d,\alpha+1}(s')\left(\sin\left(\omega_{d,\alpha+1}(s')\right)-T_{d,\alpha+1}(s')+o(1)\right)^2(x-s').\end{equation}
Applying Corollary~\ref{corollary:magnitudecomp} completes the proof.
\end{proof}
A similar result could be obtained for odd $d$ using the recurrence relation~\eqref{recurrence}, and then for $p_d^{\alpha+1,\alpha}$ by using recurrence relation~\eqref{recurrence2}.\\
\\
A comment regarding the main term in this proposition is in order. By~\eqref{Rda}, if $q_{d,\alpha}\in\left[\frac{1}{2},1\right)$,
\begin{equation*}|R_{d,\alpha+1}(x)|<\frac{(1+q_{d,\alpha+1})x}{4(1-q_{d,\alpha+1}-x^2)^{\frac{3}{2}}\sqrt{(d+\alpha+1)(d+\alpha+2)}}.\end{equation*}
By Lemma~\ref{diffroots}, for $|x-s'|=o(n^{-2/3})$, we have 
\begin{eqnarray*}|R_{d,\alpha+1}(x)|&\leq &\frac{(1+q_{d,\alpha+1})}{4(1-q_{d,\alpha+1})\left(\frac{3(\alpha+2)^{\frac{4}{3}}}{d^{\frac{2}{3}}(d+\alpha+2)^{\frac{2}{3}}(d+2\alpha+3)^{\frac{2}{3}}}\right)^{\frac{3}{2}}\sqrt{(d+\alpha+1)(d+\alpha+2)}}(1+o(1))\\
&=& \frac{\rho(1+2\rho)(2-s'^2)}{2\cdot 3^{\frac{3}{2}}s'^2}(1+o(1)).
\end{eqnarray*}
When $\theta'>0.2432...$ ($=13.937...^{\circ}$), which is guaranteed by $q_{d,\alpha+1}\in\left[\frac{1}{2},1\right)$ and~\eqref{qlim}, $|R_{d,\alpha+1}(x)|<\frac{1}{\sqrt{2}}$ for $|x-s'|=o(n^{-2/3})$. This bound also applies to $T_{d,\alpha+1}(s')$ since the proof of Lemma~\ref{Krasikovlemma} in~\cite{IK} passed through applying the triangle inequality to~\eqref{Rintegral}. Furthermore, $\cos(\omega_{d,\alpha+1}(s'))=-R_{d,\alpha+1}(s')$, and so $|\sin\left(\omega_{d,\alpha+1}(s')\right)-T_{d,\alpha+1}(s')|$ is bounded away from $0$ for sufficiently large $n$.\\
\\
On the other hand, if $q_{d,\alpha+1}\in\left[0,\frac{1}{2}\right)$, then~\eqref{Rda} implies that
\begin{equation*}|R_{d,\alpha+1}(x)|<\frac{2(1-x^2)x}{(1-q_{d,\alpha+1})(1-q_{d,\alpha+1}-x^2)^{\frac{3}{2}}\sqrt{(d+\alpha+1)(d+\alpha+2)}}.\end{equation*}
Again, by Lemma~\ref{diffroots}, for $|x-s'|=o(n^{-2/3})$, we have
\begin{eqnarray*}|R_{d,\alpha+1}(x)|&\leq&
\frac{2(1-s'^2)s'}{(1-q_{d,\alpha+1})^{\frac{5}{2}}\left(\frac{3(\alpha+2)^{\frac{4}{3}}}{d^{\frac{2}{3}}(d+\alpha+2)^{\frac{2}{3}}(d+2\alpha+3)^{\frac{2}{3}}}\right)^{\frac{3}{2}}\sqrt{(d+\alpha+1)(d+\alpha+2)}}(1+o(1))\\
 &=&\frac{8\rho(1+2\rho)(1-s'^2)}{2\cdot 3^{\frac{3}{2}}s'^4}(1+o(1)).
\end{eqnarray*}
In this case, we have $|R_{d,\alpha+1}(x)|<\frac{1}{\sqrt{2}}$ if $\theta'<0.9954...$. This is guaranteed by $q_{d,\alpha+1}\in\left[0,\frac{1}{2}\right)$ and~\eqref{qlim}. Therefore, $|\sin\left(\omega_{d,\alpha+1}(s')\right)-T_{d,\alpha+1}(s')|$ is bounded away from $0$ for sufficiently large $n$ by a universal constant.\\
\\
Though we do not use the following lemma in the proofs of our main results, it tells us that we have concavity to the right of the maximal root to $\sqrt{1-q_{d,\alpha}}$.
\begin{lemma}\label{lemma:concavity}
For $\alpha=\frac{n-3}{2}$, $(1-t^2)^{\frac{\alpha+1}{2}}p_d^{\alpha,\alpha}(t)$ is concave in the interval $[s',\sqrt{1-q_{d,\alpha}}]$. 
\end{lemma}
\begin{proof}
$y:=(1-t^2)^{\frac{\alpha+1}{2}}p_d^{\alpha,\alpha}(t)$ satisfies the normal form of the Jacobi differential equation $y''+b_{d,\alpha}^2y=0$, where $b_{d,\alpha}$ is as in~\eqref{omegader}. Since $s'$ is the largest zero of $y$ and $y(t)\geq 0$ for every $t\geq s'$, we obtain that $y''\leq 0$ on $[s',\sqrt{1-q_{d,\alpha}}]$, as required.
\end{proof}
We also do not use the following lemma, but it gives us information about the rest of the interval.
\begin{lemma}\label{lemma:decreasing}For $n\geq 3$ and $\alpha=\frac{n-3}{2}$,
$(1-t^2)^{\alpha}(p_d^{\alpha,\alpha}(t))^2$ is decreasing on $[\sqrt{1-q_{d,\alpha}},1]$.
\end{lemma}
\begin{proof}
By definition, $q_{d,\alpha}=\frac{\alpha^2-1}{(d+\alpha)(d+\alpha+1)}$. Note that
\[(n-1)^2t^2-4(1-t^2)d(d+n-2)-4t^2> 0\iff t^2> \frac{4d(d+n-2)}{(2d+n-3)(2d+n-1)+2(n-3)}.\]
For $t\in[\sqrt{1-q_{d,\alpha}},1]$,
\[t^2\geq 1-q_{d,\alpha}=1-\frac{(n-5)(n-1)}{(2d+n-3)(2d+n-1)}> \frac{4d(d+n-2)}{(2d+n-3)(2d+n-1)+2(n-3)}.\]
The last inequality follows from $n\geq 3$ and
\begin{eqnarray*}&&1-\frac{(n-5)(n-1)}{(2d+n-3)(2d+n-1)}- \frac{4d(d+n-2)}{(2d+n-3)(2d+n-1)+2(n-3)}\\
&&\geq 1-\frac{(n-5)(n-1)}{(2d+n-3)(2d+n-1)}- \frac{4d(d+n-2)}{(2d+n-3)(2d+n-1)}\\
&&=\frac{2(n-1)}{(2d+n-3)(2d+n-1)}>0.
\end{eqnarray*}
The analysis of~\cite{Barg} implies that
\[\frac{d}{dt}\log |p_d^{\alpha,\alpha}(t)|<\frac{(n-1)t-\sqrt{(n-1)^2t^2-4(1-t^2)d(d+n-2)}}{2(1-t^2)}.\]
The above imply that for $t\in[\sqrt{1-q_{d,\alpha}},1]$ and $n\geq 3$,
\[\frac{d}{dt}\log |p_d^{\alpha,\alpha}(t)|<\frac{(n-3)t}{2(1-t^2)},\]
and so
\[\frac{d}{dt}\log\left((1-t^2)^{\alpha}(p_d^{\alpha,\alpha}(t))^2\right)<0,\]
as desired.
\end{proof}
Combining Lemmas~\ref{lemma:concavity} and~\ref{lemma:decreasing}, if $n\geq 3$ and $\alpha=\frac{n-3}{2}$, on the interval $[s',1]$ with $s'=t_{1,d}^{\alpha,\alpha}$, we may bound $(1-t^2)^{\frac{\alpha+1}{2}}p_d^{\alpha,\alpha}(t)$ from above by its tangent line at $s'$. As a result, we may extend our estimate of $(1-x^2)^{\alpha}g(x)$ to hold on all of $[s',1]$ and $o(n^{-2/3})$ to the left of $s'$ at the expense of an innocuous inequality on $[s',1]$. This paper does not use this extension of our estimates; we leave the details to the reader.
\section{Proof of Theorem~\ref{THMone}}\label{section:m1b}
In this section, we prove Theorem~\ref{THMone}, that is, if $\frac{\pi}{3}\leq\theta\leq\frac{\pi}{2}$, then
\begin{equation*}\delta_n\leq\frac{1+o(1)}{e}\cdot\delta_n^{\textup{KL}}(\theta)\textup{ as }n\rightarrow\infty,\end{equation*}
and that $\frac{1}{e}$ cannot be improved if we use the construction of Subsection~\ref{m1b} applied to Levenshtein's optimal polynomials. The result has advantages over that given in~\cite{SZ}. The first is that $\frac{1}{e}$ is better than $0.4325$. The second is that $\frac{1}{e}$ is optimal in the above sense, and that the specific value of $\theta$ does not play a role in our calculations, and so we see that the improvement factor $\frac{1}{e}$ is, in contrast to~\cite{SZ}, independent of $\theta$. In particular, our argument does not use $\theta^*$ and its value. These lead to the cleaner and stronger conclusion above and explain the convergence of the entries in the table of Section 6 of~\cite{SZ} to what appeared to be $\frac{1}{e}$.\\
\\
Let $s:=\cos(\theta)$. We take $g=g_{n,\theta}$ to be Levenshtein's optimal polynomial for $M(n,\theta)$, assume that $s=t_{1,d}^{\alpha+1,\alpha+1}$, and let $F=\chi_{[0,r+\delta]}$, where $r=\frac{1}{2\sin(\theta/2)}$ and $\delta=O(n^{-1})$ is to be found. We also assume that $d$ is even. The odd case and the case where $s=t_{1,d}^{\alpha+1,\alpha}$ may be treated similarly. We associated the function $H$ given in equality~\eqref{Hspheresubsection}, which led in~\eqref{HT1} to the Cohn--Elkies linear programming conditions being satisfied if $H(T)\leq 0$ when $|T|\geq 1$, where
\begin{equation}\label{HTee}
H(T)=T^{2n-1}\int_{-1}^1g(x)\mu(x;\chi_{[0,r+\delta]}(x/T))dx.
\end{equation}
Here, we used Proposition~\ref{38}. We first find $\delta$ so that $H(1)\leq 0$. By Proposition~\ref{spheredensity}, it suffices that
\begin{equation}\label{1H}
\int_{-1}^1g_{n,\theta}(x)\mu(x;\chi_{[0,r+\delta]})dx<0
\end{equation}
for sufficiently large $n$, where, when $|x-s|=o(n^{-1/2})$,
\begin{equation*}
\mu(x;\chi_{[0,r+\delta]})=C_{n,\delta,r}(1+o(1))\frac{(1-x^2)^{\frac{n-4}{2}}}{(1-x)^{n-\frac{1}{2}}}\sqrt{1+\left(x-\frac{(1-x^2)r}{\sqrt{1-(1-x^2)r^2}}\right)^2}\left(r+\delta-\sqrt{1-(1-x^2)(r+\delta)^2}-x(r+\delta)\right)^+
\end{equation*}
as $n\rightarrow\infty$ with $C_{n,\delta,r}>0$ constants. Inequality~\eqref{1H} is equivalent to
\begin{equation*}\label{1Hunravel}
\int_{s}^1g_{n,\theta}(x)\mu(x;\chi_{[0,r+\delta]})dx<\int_{-1}^s|g_{n,\theta}(x)|\mu(x;\chi_{[0,r+\delta]})dx.
\end{equation*}
Observe that
\begin{equation*}
\left(r+\delta-\sqrt{1-(1-x^2)(r+\delta)^2}-x(r+\delta)\right)^+=\begin{cases}r+\delta-\sqrt{1-(1-x^2)(r+\delta)^2}-x(r+\delta)& \textup{ if }-1\leq x\leq 1-\frac{1}{2(r+\delta)^2}\\
0&\textup{otherwise}.\end{cases}
\end{equation*}
A Taylor expansion around $\delta=0$ gives
\begin{equation*}
1-\frac{1}{2(r+\delta)^2}=s+\frac{\delta}{r^3}+O(n^{-2}).
\end{equation*}
Furthermore, for $|x-s|=o(1)$, a Taylor expansion around $x=s$ and $\delta=0$ gives
\begin{equation*}
r+\delta-\sqrt{1-(1-x^2)(r+\delta)^2}-x(r+\delta)=2\delta-\frac{r}{1-s}(x-s)+o(1)=2\delta\left(1-\frac{r^3}{\delta}(x-s)\right)+o(1)
\end{equation*}
We also have
\begin{equation*}
\sqrt{1+\left(x-\frac{(1-x^2)r}{\sqrt{1-(1-x^2)r^2}}\right)^2}=\sqrt{2}(1+o(1))
\end{equation*}
and
\begin{equation*}
\frac{1}{(1-x)^{n-\frac{1}{2}}}=\frac{1+o(1)}{(1-s)^{n-\frac{1}{2}}}e^{\frac{n(x-s)}{1-s}}.
\end{equation*}
Therefore,
\begin{eqnarray*}
\int_{s}^1g_{n,\theta}(x)\mu(x;\chi_{[0,r+\delta]})dx=\frac{2\delta C_{n,\delta,r}\sqrt{2}(1+o(1))}{\sqrt{1-s^2}(1-s)^{n-\frac{1}{2}}}\int_s^{s+\frac{\delta}{r^3}}e^{\frac{n(x-s)}{1-s}}\left(1-\frac{r^3}{\delta}(x-s)\right)(1-x^2)^{\frac{n-3}{2}}g_{n,\theta}(x)dx.
\end{eqnarray*}
By Proposition~\ref{gcorollary}, there is a constant $K_{n,s}>0$ such that for $|x-s|=o(n^{-2/3})$,
\begin{equation*}
(1-x^2)^{\frac{n-3}{2}}g_{n,\theta}(x)=(1+o(1))K_{n,s}(x-s).
\end{equation*}
As a result, there is a constant $D_{n,\delta,s}>0$ such that
\begin{eqnarray*}
\int_{s}^1g_{n,\theta}(x)\mu(x;\chi_{[0,r+\delta]})dx=D_{n,\delta,s}(1+o(1))\int_0^{\frac{\delta}{r^3}}e^{2r^2nz}\left(1-\frac{r^3}{\delta}z\right)zdz
\end{eqnarray*}
as $n\rightarrow\infty$. Similarly, we have
\begin{eqnarray*}
\int_{-1}^s|g_{n,\theta}(x)|\mu(x;\chi_{[0,r+\delta]})dx=D_{n,\delta,s}(1+o(1))\int_0^{\frac{1}{n^{2/3}\log(n)}}e^{-2r^2nz}\left(1+\frac{r^3}{\delta}z\right)zdz
\end{eqnarray*}
as $n\rightarrow\infty$. The reason that this second estimate is an equality and not an inequality is that the ratio of  $K_{n,s}(x-s)$ to $(1-x^2)^{\frac{n-3}{2}}g_{n,\theta}(x)$ is bounded by a polynomial. This is a consequence of Lemmas~\ref{Krasikovupperbound} and~\ref{magnitudecomp} and Proposition~\ref{gcorollary}. Furthermore, away from $|x-s|=O(n^{-2/3})$, the integrand has exponential decay due to the concentration of the mass of the intersection of two balls in $\mathbb{R}^n$ near its boundary.\\
\\
Therefore, we want $\delta$ to satisfy
\begin{equation*}
\int_0^{\infty}e^{-2z}\left(1+\frac{r}{n\delta}z\right)zdz>\int_0^{\frac{n\delta}{r}}e^{2z}\left(1-\frac{r}{n\delta}z\right)zdz,
\end{equation*}
which is true exactly when
\begin{equation*}
\int_{-\frac{n\delta}{r}}^{\infty}e^{-2z}\left(1+\frac{r}{n\delta}z\right)zdz>0,
\end{equation*}
that is $\frac{n\delta}{r}<1$. Under this condition, we obtain that $H(1)\leq 0$. If $|T|-1\neq O(n^{-1})$, then it is clear that $H(T)\leq 0$. On the other hand, if $|T|-1=O(n^{-1})$, then the same argument as above combined with~\eqref{HTee} shows that if $\frac{n\delta}{r}<1$, then $H(T)\leq 0$ whenever $|T|\geq 1$. An application of Proposition~\ref{38} with $\frac{n\delta}{r}=1-o(1)$ improves $\delta_n\leq \frac{M_{\textup{Lev}}(n,\theta)}{(2r)^n}=\delta_n^{\textup{KL}}(\theta)$ to
\begin{equation*}
\delta_n\leq (1+o(1))\frac{M_{\textup{Lev}}(n,\theta)}{(2(r+\delta))^n}=\frac{(1+o(1))}{\left(1+\frac{\delta}{r}\right)^n}\delta_n^{\textup{KL}}(\theta)=\frac{1+o(1)}{e}\delta_n^{\textup{KL}}(\theta),
\end{equation*}
concluding the proof of Theorem~\ref{THMone}.
\section{Proofs of Theorem~\ref{THMtwo}}\label{section:m1}
In this section, we prove Theorem~\ref{THMtwo}, an analogue of Theorem~\ref{THMone} for spherical codes. Recall that it states that for angles $0<\theta<\theta'\leq\frac{\pi}{2}$, we have
\[M(n,\theta)\leq\frac{1+o(1)}{e}\frac{M_{\textup{Lev}}(n-1,\theta')}{\mu_n(\theta,\theta')}\textup{ as }n\rightarrow\infty.\]
Furthermore, by considering $V_1(\mathbb{R}^n)$ with Levenshtein's optimal polynomial, $\frac{1}{e}$ cannot be improved. Note that this improves the spherical codes theorem of~\cite{SZ} in two directions. First, the improvement factor is $\frac{1}{e}$ which is better than $0.4325$. Additionally, we claim optimality. Furthermore, $\theta'=\theta^*$ is not important; the optimization constant is always $\frac{1}{e}(1+o(1))$ and the proof, as we will see, does not make use of $\theta^*$ and its value, in contrast to the argument we gave in~\cite{SZ}.\\
\\
We give two proofs of this result. One proof uses the estimate on conditional densities for averaging over spheres, which was important to our calculations in~\cite{SZ}. However, we replace our estimation of Levenshtein's optimal polynomials in~\cite{SZ} with Proposition~\ref{gcorollary}. In the second proof, we show that this new estimate allows us to circumvent estimating conditional densities altogether and directly obtain inequality~\eqref{constantimprovement} in a cleaner way.
\subsection{Proof using conditional densities estimate}
In this first proof, we use our estimate of conditional densities in the case of $V_1(\mathbb{R}^n)$, that is, Proposition~\ref{nuFestimate}, to find the maximal $\delta\in[0,r]$ such that we have $H(t)\leq 0$ for every $t\in[-1,s]$, where $s:=\cos(\theta)$. Let $s':=\cos(\theta')$.\\
\\
First, note that it suffices to assume that $|x-s'|=o\left(n^{-2/3}\right)$ and so $|t-s|=o\left(n^{-2/3}\right)$. As in~\cite{SZ}, it suffices to consider $t=s$. Given $\theta'$, $d$ is uniquely determined by $t^{\alpha+1,\alpha+1}_{1,d-1}< \cos(\theta') \leq t^{\alpha+1,\alpha+1}_{1,d}$, where $\alpha=\frac{n-4}{2}$, and $g=g_{n-1,\theta'}$ is Levenshtein's optimal polynomial for bounding $M(n-1,\theta')$. Without loss of generality, we may assume that  $s':=\cos(\theta')=t^{\alpha+1,\alpha+\varepsilon}_{1,d}$ for some $\varepsilon\in\{0,1\}$. We address the case when $\varepsilon=1$ and $d$ is even; the other cases are resolved in a similar way. We therefore have
\[g(x)=\frac{(x+1)^2}{(x-s')}\left(  p^{\alpha+1,\alpha+1}_{d}(x) \right)^2.\]
See the discussion around equation~\eqref{gchoice}. As a result of equation~\eqref{Hrewrittenm1} and Proposition~\ref{Hexplicit2m1}, $H(s)\leq 0$ is equivalent to the following inequality being valid.
\[0\geq\int_{-1}^1g(x)\nu_1(x;s,\chi_{[r-\delta,R]})dx=\int_{-1}^{s'}g(x)\nu_1(x;s,\chi_{[r-\delta,R]})dx+\int_{s'}^1g(x)\mu(x,s)dx=-\int_{-1}^{s'}|g(x)|\nu_1(x;s,\chi_{[r-\delta,R]})dx+\int_{s'}^1g(x)\mu(x,s)dx,\]
that is,
\begin{equation}\label{gineqm1}\int_{s'}^1g(x)\nu_1(x;s,\chi_{[r-\delta,R]})dx\leq \int_{-1}^{s'}|g(x)|\nu_1(x;s,\chi_{[r-\delta,R]})dx.\end{equation}
Note that by construction, $g(x)\leq 0$ for every $x\in[-1,s']$. We estimate both sides of this inequality and find conditions equivalent to the validity of the first Delsarte linear programming condition for $M(n,\theta)$. By Proposition~\ref{nuFestimate}, for $|x-s'|=o(n^{-1/2})$ and $\delta=o(n^{-1/2})$,
\[\nu_1(x;s,\chi_{[r-\delta,R]})=C_{n,R,r,\delta,s}(1+o(1))\left(\frac{1-x^2}{x^2}\right)^{\frac{n-4}{2}}\left(\delta+\sqrt{\frac{s-x}{1-x}}-r\right)^+e^{-\frac{2nr\left(\sqrt{\frac{s-x}{1-x}}-r\right)}{s-r^2}}\]
for some constant $C_{n,R,r,\delta,s}>0$. As a result,
\[C_{n,R,r,\delta,s}^{-1}\int_{s'}^1g(x)\nu_1(x;s,\chi_{[r-\delta,R]})dx=(1+o(1))\int_{s'}^1\frac{1}{x^{n-4}}\left(\delta+\sqrt{\frac{s-x}{1-x}}-r\right)^+e^{-\frac{2nr\left(\sqrt{\frac{s-x}{1-x}}-r\right)}{s-r^2}}g(x)(1-x^2)^{\frac{n-4}{2}}dx.\]
Let $z:=x-s'$ with $|z|\leq\frac{1}{n^{2/3}\log(n)}=o(n^{-2/3})$. We have the Taylor expansions
\begin{equation*}
\sqrt{\frac{s-x}{1-x}}-r=-\frac{z(1-s)}{2(s-s')^{\frac{1}{2}}(1-s')^{\frac{3}{2}}}(1+o(1))
\end{equation*}
and
\begin{equation*}\frac{1}{x}=\frac{1}{s'}\left(1-\frac{z}{s'}\right)+o\left(\frac{1}{n}\right)\implies \frac{1}{x^{n-4}}=(1+o(1))\frac{1}{s'^{n-4}}e^{-n\frac{z}{s'}}.\end{equation*}
Therefore,
\begin{equation}\label{upperint}C_{n,R,r,\delta,s}^{-1}s'^{n-4}\int_{s'}^1g(x)\nu_1(x;s,\chi_{[r-\delta,R]})dx=(1+o(1))\int_{0}^{\frac{2\delta(s-s')^{\frac{1}{2}}(1-s')^{\frac{3}{2}}}{(1-s)}}e^{\frac{nz}{1-s'}}\left(\delta-\frac{z(1-s)}{2(s-s')^{\frac{1}{2}}(1-s')^{\frac{3}{2}}}\right)g(s'+z)(1-(s'+z)^2)^{\frac{n-4}{2}}dz.\end{equation}
Similarly,
\begin{equation}\label{lowerint}C_{n,R,r,\delta,s}^{-1}s'^{n-4}\int_{-1}^{s'}|g(x)|\nu_1(x;s,\chi_{[r-\delta,R]})dx\geq (1+o(1))\int_{0}^{\frac{1}{n^{2/3}\log(n)}}e^{\frac{-nz}{1-s'}}\left(\delta+\frac{z(1-s)}{2(s-s')^{\frac{1}{2}}(1-s')^{\frac{3}{2}}}\right)|g(s'-z)|(1-(s'-z)^2)^{\frac{n-4}{2}}dz.\end{equation}
By Proposition~\ref{gcorollary}, for $|x-s'|=o(n^{-2/3})$,
\begin{equation}\label{gK}(1-x^2)^{\frac{n-4}{2}}g(x)=(1+o(1))K_{\alpha,d,s'}(x-s')\end{equation}
for a constant $K_{\alpha,d,s'}>0$. Combining~\eqref{gK} with~\eqref{upperint} and~\eqref{lowerint}, inequality~\eqref{gineqm1} is satisfied if for large enough $n$, $\delta$ satisfies
\begin{eqnarray*}
&&\int_{0}^{\frac{1}{n^{2/3}\log(n)}}e^{\frac{-nz}{1-s'}}\left(\delta+\frac{z(1-s)}{2(s-s')^{\frac{1}{2}}(1-s')^{\frac{3}{2}}}\right)zdz\\&>&\int_{0}^{\frac{2\delta(s-s')^{\frac{1}{2}}(1-s')^{\frac{3}{2}}}{(1-s)}}e^{\frac{nz}{1-s'}}\left(\delta-\frac{z(1-s)}{2(s-s')^{\frac{1}{2}}(1-s')^{\frac{3}{2}}}\right)zdz.
\end{eqnarray*}
Changing variables, this is equivalent to
\begin{eqnarray*}
&&\int_{0}^{\frac{n^{1/3}}{\log(n)}}e^{\frac{-v}{1-s'}}\left(1+\frac{v(1-s)}{2n\delta(s-s')^{\frac{1}{2}}(1-s')^{\frac{3}{2}}}\right)vdv\\&>&\int_{0}^{\frac{2\delta n(s-s')^{\frac{1}{2}}(1-s')^{\frac{3}{2}}}{(1-s)}}e^{\frac{v}{1-s'}}\left(1-\frac{v(1-s)}{2n\delta(s-s')^{\frac{1}{2}}(1-s')^{\frac{3}{2}}}\right)vdv.
\end{eqnarray*}
Let $\Lambda:=\frac{2\delta n(s-s')^{\frac{1}{2}}(1-s')^{\frac{3}{2}}}{(1-s)}$. Therefore, we may equivalently maximize $\Lambda>0$ such that
\begin{eqnarray*}
\int_{0}^{\infty}e^{\frac{-v}{1-s'}}\left(1+\frac{v}{\Lambda}\right)vdv>\int_{0}^{\Lambda}e^{\frac{v}{1-s'}}\left(1-\frac{v}{\Lambda}\right)vdv
\end{eqnarray*}
for sufficiently large $n$, that is,
\begin{eqnarray*}
\int_{-\Lambda}^{\infty}e^{\frac{-v}{1-s'}}\left(1+\frac{v}{\Lambda}\right)vdv>0.
\end{eqnarray*}
This is true if and only if $\Lambda<2(1-s')$. Letting $\delta=\frac{\gamma}{n}$ for some constant $\gamma>0$,
\[\Lambda=\frac{2\gamma(s-s')^{\frac{1}{2}}(1-s')^{\frac{3}{2}}}{(1-s)}< 2(1-s')\]
that is, we can take
\[\gamma=(1-o(1))\frac{(1-s)}{(s-s')^{\frac{1}{2}}(1-s')^{\frac{1}{2}}}=(1-o(1))\frac{1-r^2}{r}.\]
We obtain that the improvement factor is
\[\left(\sqrt{\frac{1-(r-\delta)^2}{1-r^2}}\right)^{-n}=\left(1+\frac{r\gamma}{n(1-r^2)}+o\left(\frac{1}{n}\right)\right)^{-n}=e^{-1}(1+o(1)),\]
establising inequality~\eqref{constantimprovement}. We comment on optimality in the second proof; the argument is similar to the sphere packings case given in the previous section.
\subsection{Proof without using conditional densities estimate}\label{subsection:m1b}
By~\eqref{Hrewrittenm1}, for $t\in(-1,1)$, $H(t)$ is a positive multiple of
\begin{equation}\label{Hrewrittenm1b}
\iint_{\cal{R}_t}g\left(\frac{t-uv}{\sqrt{(1-u^2)(1-v^2)}}\right)\chi_{[r-\delta,R]}(u)\chi_{[r-\delta,R]}(v)\left((1-u^2)(1-v^2)-(t-uv)^2\right)^{\frac{n-4}{2}}dudv,
\end{equation}
where
\[\cal{R}_t:=\left\{(u,v)\in[-1,1]^2:-1\leq\frac{t-uv}{\sqrt{(1-u^2)(1-v^2)}}\leq 1\right\}.\]
Similar to the first proof, it suffices to prove that $H(s)\leq 0$. Since $g(x)\leq 0$ for $x\in[-1,s']$, $H(s)\leq 0$ if
\begin{equation}\label{Hm1new}
\int_{r-\delta}^{r+\frac{1}{n^{2/3}\log(n)}}\int_{r-\delta}^{r+\frac{1}{n^{2/3}\log(n)}}\left(1-\left(\frac{s-uv}{\sqrt{(1-u^2)(1-v^2)}}\right)^2\right)^{\frac{n-4}{2}}g\left(\frac{s-uv}{\sqrt{(1-u^2)(1-v^2)}}\right)(1-u^2)^{\frac{n-4}{2}}(1-v^2)^{\frac{n-4}{2}}dudv\leq 0.
\end{equation} 
Note that for $r-\delta\leq u,v\leq r+\frac{1}{n^{2/3}\log(n)}$ with $\delta=o(1)$,
\[\frac{s-uv}{\sqrt{(1-u^2)(1-v^2)}}=s'+o(1),\]
and so $(u,v)\in\cal{R}_t$ is automatically satisfied. From Proposition~\ref{gcorollary}, inequality~\eqref{Hm1new} is satisfied for sufficiently large $n$ if and only if
\begin{equation}\label{Hm1newbineq}
\int_{r-\frac{\gamma}{n}}^{r+\frac{1}{n^{2/3}\log(n)}}\int_{r-\frac{\gamma}{n}}^{r+\frac{1}{n^{2/3}\log(n)}}\left(\frac{s-uv}{\sqrt{(1-u^2)(1-v^2)}}-s'\right)(1-u^2)^{\frac{n-4}{2}}(1-v^2)^{\frac{n-4}{2}}dudv< 0,
\end{equation}
where $\gamma>0$ is a constant such that $\delta=\frac{\gamma}{n}$. In order to find the largest $\gamma>0$ such that inequality~\eqref{Hm1newbineq} is true, we calculate the asymptotics of the integral. Expanding the integral gives
\begin{eqnarray*}\label{Hm1newb}
&&\int_{r-\frac{\gamma}{n}}^{r+\frac{1}{n^{2/3}\log(n)}}\int_{r-\frac{\gamma}{n}}^{r+\frac{1}{n^{2/3}\log(n)}}\left(\frac{s-uv}{\sqrt{(1-u^2)(1-v^2)}}-s'\right)(1-u^2)^{\frac{n-4}{2}}(1-v^2)^{\frac{n-4}{2}}dudv\\
&=&s\left(\int_{r-\frac{\gamma}{n}}^{r+\frac{1}{n^{2/3}\log(n)}}(1-u^2)^{\frac{n-5}{2}}du\right)^2-\left(\int_{r-\frac{\gamma}{n}}^{r+\frac{1}{n^{2/3}\log(n)}}u(1-u^2)^{\frac{n-5}{2}}du\right)^2-s'\left(\int_{r-\frac{\gamma}{n}}^{r+\frac{1}{n^{2/3}\log(n)}}(1-u^2)^{\frac{n-4}{2}}du\right)^2.
\end{eqnarray*}
We have
\begin{eqnarray*}
&&\int_{r-\frac{\gamma}{n}}^{r+\frac{1}{n^{2/3}\log(n)}}(1-u^2)^{\frac{n-5}{2}}du\\ 
%&=& \int_{-\frac{\gamma}{n}}^{\frac{\beta}{n}}(1-(r+t)^2)^{\frac{n-5}{2}}dt\\
&=&(1-r^2)^{\frac{n-5}{2}}\int_{-\frac{\gamma}{n}}^{\frac{1}{n^{2/3}\log(n)}}\left(1-\frac{2rt+t^2}{1-r^2}\right)^{\frac{n-5}{2}}dt\\
&=&\frac{(1-r^2)^{\frac{n-5}{2}}}{n}\int_{-\gamma}^{\frac{n^{1/3}}{\log(n)}}\left(1-\frac{2r\frac{t}{n}+\frac{t^2}{n^2}}{1-r^2}\right)^{\frac{n-5}{2}}dt.
%&=&(1+o(1))(1-r^2)^{\frac{n-5}{2}}\int_{-\frac{\gamma}{n}}^{\frac{1}{n^{2/3}\log(n)}}e^{-\frac{n-5}{2}\left(\frac{2rt}{1-r^2}\right)}dt\\
%&=&(1+o(1))\frac{(1-r^2)^{\frac{n-5}{2}}}{n}\int_{-\gamma}^{\frac{n^{1/3}}{\log(n)}}e^{-\left(\frac{rt}{1-r^2}\right)}dt.
\end{eqnarray*}
Similarly,
\begin{eqnarray*}
\int_{r-\frac{\gamma}{n}}^{r+\frac{1}{n^{2/3}\log(n)}}(1-u^2)^{\frac{n-4}{2}}du=\frac{(1-r^2)^{\frac{n-4}{2}}}{n}\int_{-\gamma}^{\frac{n^{1/3}}{\log(n)}}\left(1-\frac{2r\frac{t}{n}+\frac{t^2}{n^2}}{1-r^2}\right)^{\frac{n-4}{2}}dt
\end{eqnarray*}
and
\[\int_{r-\frac{\gamma}{n}}^{r+\frac{1}{n^{2/3}\log(n)}}u(1-u^2)^{\frac{n-5}{2}}du=\frac{(1-r^2)^{\frac{n-5}{2}}}{n}\int_{-\gamma}^{\frac{n^{1/3}}{\log(n)}}\left(r+\frac{t}{n}\right)\left(1-\frac{2r\frac{t}{n}+\frac{t^2}{n^2}}{1-r^2}\right)^{\frac{n-5}{2}}dt.\]
From these and the fact that $s-r^2-s'(1-r^2)=0$, we deduce that the desired inequality is equivalent to
\begin{eqnarray*}
0&>&s\left(\int_{-\gamma}^{\frac{n^{1/3}}{\log(n)}}\left(1-\frac{2r\frac{t}{n}+\frac{t^2}{n^2}}{1-r^2}\right)^{\frac{n-5}{2}}dt\right)^2-\left(\int_{-\gamma}^{\frac{n^{1/3}}{\log(n)}}\left(r+\frac{t}{n}\right)\left(1-\frac{2r\frac{t}{n}+\frac{t^2}{n^2}}{1-r^2}\right)^{\frac{n-5}{2}}dt\right)^2-s'(1-r^2)\left(\int_{-\gamma}^{\frac{n^{1/3}}{\log(n)}}\left(1-\frac{2r\frac{t}{n}+\frac{t^2}{n^2}}{1-r^2}\right)^{\frac{n-4}{2}}dt\right)^2\\
&=& -\left(\frac{2r}{n}\left(\int_{-\gamma}^{\frac{n^{1/3}}{\log(n)}}t\left(1-\frac{2r\frac{t}{n}+\frac{t^2}{n^2}}{1-r^2}\right)^{\frac{n-5}{2}}dt\right)\left(\int_{-\gamma}^{\frac{n^{1/3}}{\log(n)}}t\left(1-\frac{2r\frac{t}{n}+\frac{t^2}{n^2}}{1-r^2}\right)^{\frac{n-5}{2}}dt\right)+\frac{1}{n^2}\left(\int_{-\gamma}^{\frac{n^{1/3}}{\log(n)}}t\left(1-\frac{2r\frac{t}{n}+\frac{t^2}{n^2}}{1-r^2}\right)^{\frac{n-5}{2}}dt\right)^2\right)\\
&&+\frac{2rs'}{n}\left(\int_{-\gamma}^{\frac{n^{1/3}}{\log(n)}}\left(1-\frac{2r\frac{t}{n}+\frac{t^2}{n^2}}{1-r^2}\right)^{\frac{n-5}{2}}dt\right)\left(\int_{-\gamma}^{\frac{n^{1/3}}{\log(n)}}t\left(1-\frac{2r\frac{t}{n}+\frac{t^2}{n^2}}{1-r^2}\right)^{\frac{n-5}{2}}dt\right)+O(n^{-2})
\end{eqnarray*}
This is negative for sufficiently large $n$ exactly when
\begin{equation*}\label{ineq:1}\int_{-\gamma}^{\infty}te^{-\frac{rt}{1-r^2}}dt>0,\end{equation*}
from which we would obtain that $H(s)\leq 0$ for sufficiently large $n$ provided that $\frac{r\gamma}{1-r^2}<1$. Note that if $|t-s|\neq O(n^{-1})$ and $t<s$, then it is clear that $H(t)<0$. On the other hand, if $|t-s|=O(n^{-1})$, we could simply replace $s$ with $t$ in the above argument. Choosing $\gamma=(1-o(1))\frac{1-r^2}{r}$ leads to the same improvement factor $\frac{1}{e}(1+o(1))$ as in the first proof.\\
\\
As for the optimality of this improvement factor, note that applying Lemmas~\ref{diffroots} and~\ref{Krasikovupperbound} and Proposition~\ref{gcorollary}, we see that the bound of Lemma~\ref{Krasikovupperbound} and the estimate of Proposition~\ref{gcorollary} differ by a multiplicative polynomial factor. However, the rest of the integrand $((1-u^2)(1-v^2))^{\frac{n-4}{2}}$ is exponentially smaller in the $\max\{u,v\}\neq r+O(n^{-1})$ regime compared to the $\max\{u,v\}=r+O(n^{-1})$ regime. Therefore, up to a multiple of order $1+o(1)$, there is no loss in considering~\eqref{Hm1new}, equivalently~\eqref{Hm1newbineq}, instead of~\eqref{Hrewrittenm1b}. Similarly, in the first proof, the inequality sign in~\eqref{lowerint} may be replaced with an equality.\\
\\
By~\cite{KL}, we have as $n\rightarrow\infty$
\begin{equation*}
\frac{1}{n}\log_2(M(n,\theta'))\leq (1+o(1))\left(\frac{1+\sin(\theta')}{2\sin(\theta')}\log_2\left(\frac{1+\sin(\theta')}{2\sin(\theta')}\right)-\frac{1-\sin(\theta')}{2\sin(\theta')}\log_2\left(\frac{1-\sin(\theta')}{2\sin(\theta')}\right)\right).
\end{equation*}
As a result, on geometric average, the presence of $\frac{M_{\textup{Lev}}(n-1,\theta')}{M_{\textup{Lev}}(n+1,\theta')}$ leads to an improvement by a factor of
\begin{equation*}
\frac{1+o(1)}{e}2^{-2\left(\frac{1+\sin(\theta')}{2\sin(\theta')}\log_2\left(\frac{1+\sin(\theta')}{2\sin(\theta')}\right)-\frac{1-\sin(\theta')}{2\sin(\theta')}\log_2\left(\frac{1-\sin(\theta')}{2\sin(\theta')}\right)\right)}.
\end{equation*}
When $\theta'=\theta^*$, which is the best comparison angle, we obtain an improvement by a factor of $0.2304...$ on geometric average, completing the proof of Theorem~\ref{THMtwo}.
\section{Improvement factors for $V_2(\mathbb{R}^n)$}\label{section:m2}
We now determine the improvement factor obtained when integrating over the Stiefel manifold $V_2(\mathbb{R}^n)$, that is, when $m=2$. Though it does not lead to a bound better than
\[M(n,\theta)\leq \frac{1+o(1)}{e}\frac{M_{\textup{Lev}}(n-1,\theta^*)}{\mu_n(\theta,\theta^*)}\]
whenever $0<\theta<\theta^*$, the improvement factors obtained by comparing $M(n,\theta)$ to $M(n-2,\theta')$ are better if $\theta'$ is small.
We do not use an estimate of conditional densities for averaging over Stiefel manifolds. We give a proof in the spirit of the second proof in the $V_1(\mathbb{R}^n)$ case given in Subsection~\ref{subsection:m1b}.\\
\\
The setup is as follows. Given angles $0<\theta<\theta'\leq\frac{\pi}{2}$, and $\delta\in[0,\theta'-\theta]$ with $\eta\in[0,\frac{\pi}{2}]$, consider the region
\[\cal{C}^{\theta,\theta'}_{\delta,\eta,2}:=\left\{\boldsymbol{u}\in\mathbb{B}^2: r-\delta\leq|\boldsymbol{u}|\leq R\textup{ and }\left<\boldsymbol{u},\boldsymbol{p}\right>\geq |\boldsymbol{u}|\cos(\eta)\right\}\]
where $r$ and $R$ are as in~\eqref{goodR} and~\eqref{bigRfirst}. We take $F=\chi_{\cal{C}^{\theta,\theta'}_{\delta,\eta,2}}:[-1,1]^2\rightarrow\mathbb{R}$ the characteristic function of this region, and $g=g_{n-2,\theta'}$ Levenshtein's optimal polynomial for $M(n-2,\theta')$. Using this setup, with appropriately chosen $\delta,\eta$, we prove the following.
\begin{theorem}For angles $0<\theta<\theta'\leq\frac{\pi}{2}$, the function $H$ obtained by averaging over $V_2(\mathbb{R}^n)$ gives the bound
\[M(n,\theta)\leq (1+o(1))\sqrt{\frac{\pi}{6e(1-s')}}\cdot\frac{M_{\textup{Lev}}(n-2,\theta')}{M_{\textup{Lev}}(n+1,\theta')}\cdot\underbrace{\frac{\sqrt{2\pi n}rM_{\textup{Lev}}(n+1,\theta')}{(1-r^2)^{\frac{n-1}{2}}}}_{\textup{KL bound}}.\]
\end{theorem}
\begin{proof}
By Proposition~\ref{Hexplicit1}, we have $H(t)\leq 0$ if and only if
\begin{equation*}\label{Hexplicitm2app}\iint_{\cal{R}^2_t}g_{n-2,\theta'}\left(\frac{t-\left<\boldsymbol{u},\boldsymbol{v}\right>}{\sqrt{(1-|\boldsymbol{u}|^2)(1-|\boldsymbol{v}|^2)}}\right)\left(1-\left(\frac{t-\left<\boldsymbol{u},\boldsymbol{v}\right>}{\sqrt{(1-|\boldsymbol{u}|^2)(1-|\boldsymbol{v}|^2)}}\right)^2\right)^{\frac{n-5}{2}}\chi_{\cal{C}^{\theta,\theta'}_{\delta,\eta,2}}(\boldsymbol{u})\chi_{\cal{C}^{\theta,\theta'}_{\delta,\eta,2}}(\boldsymbol{v})\left((1-|\boldsymbol{u}|^2)(1-|\boldsymbol{v}|^2)\right)^{\frac{n-5}{2}}d\boldsymbol{u}d\boldsymbol{v}\leq 0.\end{equation*}
As in the $V_1(\mathbb{R}^n)$ case, by Proposition~\ref{gcorollary}, if $\delta=o(n^{-2/3})$ and $1-\cos(\eta)=o(n^{-2/3})$, we may reduce this to showing that for sufficiently large $n$,
\begin{equation}\label{m2ineq}
\iint_{(\cal{C}^{\theta,\theta'}_{\delta,\eta,2}\times\cal{C}^{\theta,\theta'}_{\delta,\eta,2})\cap(\mathbb{B}^2_{r+\frac{1}{n^{2/3}\log(n)}}\times\mathbb{B}^2_{r+\frac{1}{n^{2/3}\log(n)}})}\left(\frac{s-\left<\boldsymbol{u},\boldsymbol{v}\right>}{\sqrt{(1-|\boldsymbol{u}|^2)(1-|\boldsymbol{v}|^2)}}-s'\right)(1-|\boldsymbol{u}|^2)^{\frac{n-5}{2}}(1-|\boldsymbol{v}|^2)^{\frac{n-5}{2}}dudv<0.
\end{equation}
Here, $\mathbb{B}^2_{r+\frac{1}{n^{2/3}\log(n)}}$ is the closed ball centered at the origin of radius $r+\frac{1}{n^{2/3}\log(n)}$. It suffices to show that this quantity is negative for appropriately chosen constant $\gamma>0$ with $\delta=\frac{\gamma}{n}$ and appropriately chosen $\eta$. In polar coordinates, the integral in inequality~\eqref{m2ineq} is equal to
\begin{eqnarray*}
&&\iint\limits_{\substack{r-\delta\leq\rho_1\leq r+\frac{1}{n^{2/3}\log(n)}\\ -\eta\leq \theta_1\leq \eta}}\iint\limits_{\substack{r-\delta\leq\rho_2\leq r+\frac{1}{n^{2/3}\log(n)}\\ -\eta\leq \theta_2\leq \eta}}\left(\frac{s-\rho_1\rho_2\cos(\theta_1-\theta_2)}{\sqrt{(1-\rho_1^2)(1-\rho_2^2)}}-s'\right)(1-\rho_1^2)^{\frac{n-5}{2}}(1-\rho_2^2)^{\frac{n-5}{2}}\rho_1\rho_2d\theta_2d\rho_2d\theta_1 d\rho_1\\
&=&4s\eta^2\left(\int_{r-\delta}^{r+\frac{1}{n^{2/3}\log(n)}}u(1-u^2)^{\frac{n-6}{2}}du\right)^2-\left(\iint_{[-\eta,\eta]^2}\cos(\theta_1-\theta_2)d\theta_1d\theta_2\right)\left(\int_{r-\delta}^{r+\frac{1}{n^{2/3}\log(n)}}u^2(1-u^2)^{\frac{n-6}{2}}du\right)^2\\&&-4s'\eta^2\left(\int_{r-\delta}^{r+\frac{1}{n^{2/3}\log(n)}}u(1-u^2)^{\frac{n-5}{2}}du\right)^2\\
&=&4s\eta^2\left(\int_{r-\frac{\gamma}{n}}^{r+\frac{1}{n^{2/3}\log(n)}}u(1-u^2)^{\frac{n-6}{2}}du\right)^2-4\sin^2(\eta)\left(\int_{r-\frac{\gamma}{n}}^{r+\frac{1}{n^{2/3}\log(n)}}u^2(1-u^2)^{\frac{n-6}{2}}du\right)^2-4s'\eta^2\left(\int_{r-\frac{\gamma}{n}}^{r+\frac{1}{n^{2/3}\log(n)}}u(1-u^2)^{\frac{n-5}{2}}du\right)^2.
\end{eqnarray*}
We want to choose $\gamma,\eta$ such that the last integral is negative for sufficiently large $n$. As in Subsection~\ref{subsection:m1b},
\[\int_{r-\frac{\gamma}{n}}^{r+\frac{1}{n^{2/3}\log(n)}}u(1-u^2)^{\frac{n-6}{2}}du=\frac{(1-r^2)^{\frac{n-6}{2}}}{n}\int_{-\gamma}^{\frac{n^{1/3}}{\log(n)}}\left(r+\frac{t}{n}\right)\left(1-\frac{2r\frac{t}{n}+\frac{t^2}{n^2}}{1-r^2}\right)^{\frac{n-6}{2}}dt,\]
and
\[\int_{r-\frac{\gamma}{n}}^{r+\frac{1}{n^{2/3}\log(n)}}u(1-u^2)^{\frac{n-5}{2}}du=\frac{(1-r^2)^{\frac{n-5}{2}}}{n}\int_{-\gamma}^{\frac{n^{1/3}}{\log(n)}}\left(r+\frac{t}{n}\right)\left(1-\frac{2r\frac{t}{n}+\frac{t^2}{n^2}}{1-r^2}\right)^{\frac{n-5}{2}}dt.\]
Similarly,
\[\int_{r-\frac{\gamma}{n}}^{r+\frac{1}{n^{2/3}\log(n)}}u^2(1-u^2)^{\frac{n-6}{2}}du=\frac{(1-r^2)^{\frac{n-6}{2}}}{n}\int_{-\gamma}^{\frac{n^{1/3}}{\log(n)}}\left(r+\frac{t}{n}\right)^2\left(1-\frac{2r\frac{t}{n}+\frac{t^2}{n^2}}{1-r^2}\right)^{\frac{n-6}{2}}dt.\]
Therefore, it suffices to require $\gamma,\eta$ to be such that
\begin{eqnarray*}
&&s\left(\int_{-\gamma}^{\frac{n^{1/3}}{\log(n)}}\left(r+\frac{t}{n}\right)\left(1-\frac{2r\frac{t}{n}+\frac{t^2}{n^2}}{1-r^2}\right)^{\frac{n-6}{2}}dt\right)^2-\left(r\int_{-\gamma}^{\frac{n^{1/3}}{\log(n)}}\left(r+\frac{t}{n}\right)\left(1-\frac{2r\frac{t}{n}+\frac{t^2}{n^2}}{1-r^2}\right)^{\frac{n-6}{2}}dt+\frac{1}{n}\int_{-\gamma}^{\frac{n^{1/3}}{\log(n)}}t\left(r+\frac{t}{n}\right)\left(1-\frac{2r\frac{t}{n}+\frac{t^2}{n^2}}{1-r^2}\right)^{\frac{n-6}{2}}dt\right)^2\\&&-s'(1-r^2)\left(\int_{-\gamma}^{\frac{n^{1/3}}{\log(n)}}\left(r+\frac{t}{n}\right)\left(1-\frac{2r\frac{t}{n}+\frac{t^2}{n^2}}{1-r^2}\right)^{\frac{n-5}{2}}dt\right)^2+\left(1-\frac{\sin^2(\eta)}{\eta^2}\right)\left(\int_{-\gamma}^{\frac{n^{1/3}}{\log(n)}}\left(r+\frac{t}{n}\right)^2\left(1-\frac{2r\frac{t}{n}+\frac{t^2}{n^2}}{1-r^2}\right)^{\frac{n-6}{2}}dt\right)^2\\
&=&-\frac{2r}{n}\left(\int_{-\gamma}^{\frac{n^{1/3}}{\log(n)}}\left(r+\frac{t}{n}\right)\left(1-\frac{2r\frac{t}{n}+\frac{t^2}{n^2}}{1-r^2}\right)^{\frac{n-6}{2}}dt\right)\left(\int_{-\gamma}^{\frac{n^{1/3}}{\log(n)}}t\left(r+\frac{t}{n}\right)\left(1-\frac{2r\frac{t}{n}+\frac{t^2}{n^2}}{1-r^2}\right)^{\frac{n-6}{2}}dt\right)-\frac{1}{n^2}\left(\int_{-\gamma}^{\frac{n^{1/3}}{\log(n)}}t\left(r+\frac{t}{n}\right)\left(1-\frac{2r\frac{t}{n}+\frac{t^2}{n^2}}{1-r^2}\right)^{\frac{n-6}{2}}dt\right)^2\\&&+\left(1-\frac{\sin^2(\eta)}{\eta^2}\right)\left(\int_{-\gamma}^{\frac{n^{1/3}}{\log(n)}}\left(r+\frac{t}{n}\right)^2\left(1-\frac{2r\frac{t}{n}+\frac{t^2}{n^2}}{1-r^2}\right)^{\frac{n-6}{2}}dt\right)^2+s'(1-r^2)\left(\int_{-\gamma}^{\frac{n^{1/3}}{\log(n)}}\left(r+\frac{t}{n}\right)\left(1-\frac{2r\frac{t}{n}+\frac{t^2}{n^2}}{1-r^2}\right)^{\frac{n-6}{2}}dt\right)^2\\
&&-s'(1-r^2)\left(\int_{-\gamma}^{\frac{n^{1/3}}{\log(n)}}\left(r+\frac{t}{n}\right)\left(1-\frac{2r\frac{t}{n}+\frac{t^2}{n^2}}{1-r^2}\right)^{\frac{n-5}{2}}dt\right)^2
\end{eqnarray*}
is negative for sufficiently large $n$. In the equality, the identity $s-r^2-s'(1-r^2)=0$ is used. Note that
\begin{eqnarray*}
&&s'(1-r^2)\left(\int_{-\gamma}^{\frac{n^{1/3}}{\log(n)}}\left(r+\frac{t}{n}\right)\left(1-\frac{2r\frac{t}{n}+\frac{t^2}{n^2}}{1-r^2}\right)^{\frac{n-6}{2}}dt\right)^2-s'(1-r^2)\left(\int_{-\gamma}^{\frac{n^{1/3}}{\log(n)}}\left(r+\frac{t}{n}\right)\left(1-\frac{2r\frac{t}{n}+\frac{t^2}{n^2}}{1-r^2}\right)^{\frac{n-5}{2}}dt\right)^2\\
&=& s'(1-r^2)\left(\int_{-\gamma}^{\frac{n^{1/3}}{\log(n)}}\left(r+\frac{t}{n}\right)\left(1-\frac{2r\frac{t}{n}+\frac{t^2}{n^2}}{1-r^2}\right)^{\frac{n-6}{2}}\left(1-\sqrt{1-\frac{2r\frac{t}{n}+\frac{t^2}{n^2}}{1-r^2}}\right)dt\right)\left(\int_{-\gamma}^{\frac{n^{1/3}}{\log(n)}}\left(r+\frac{t}{n}\right)\left(1-\frac{2r\frac{t}{n}+\frac{t^2}{n^2}}{1-r^2}\right)^{\frac{n-6}{2}}\left(1+\sqrt{1-\frac{2r\frac{t}{n}+\frac{t^2}{n^2}}{1-r^2}}\right)dt\right)\\
&=& \frac{2s'r}{n}\left(\int_{-\gamma}^{\frac{n^{1/3}}{\log(n)}}\left(r+\frac{t}{n}\right)\left(1-\frac{2r\frac{t}{n}+\frac{t^2}{n^2}}{1-r^2}\right)^{\frac{n-6}{2}}dt\right)\left(\int_{-\gamma}^{\frac{n^{1/3}}{\log(n)}}t\left(r+\frac{t}{n}\right)\left(1-\frac{2r\frac{t}{n}+\frac{t^2}{n^2}}{1-r^2}\right)^{\frac{n-6}{2}}dt\right)+O(n^{-2})
\end{eqnarray*}
We take $\gamma$ to satisfy $\frac{r\gamma}{1-r^2}<1$ so that $\int_{-\gamma}^{\infty}te^{-\frac{rt}{1-r^2}}dt>0$. Then it suffices to take
\[\eta^2<\frac{6(1-s')}{nr}\frac{\left(\int_{-\gamma}^{\infty}te^{-\frac{rt}{1-r^2}}dt\right)}{\left(\int_{-\gamma}^{\infty}e^{-\frac{rt}{1-r^2}}dt\right)}=\frac{6(1-s')}{nr}\left(\frac{1-r^2}{r}-\gamma\right)\]
in order to have the negativity of the integral for sufficiently large $n$. Consequently, for each $\gamma$ satisfying $\frac{r\gamma}{1-r^2}<1$ we could take $\eta=\frac{\kappa}{\sqrt{n}}$ with $\kappa>0$ such that
\[\kappa=(1-o(1))\sqrt{\frac{6(1-s')}{r}\left(\frac{1-r^2}{r}-\gamma\right)}.\]
We now calculate the bound obtained in this setup and optimize over $\gamma$. By inequality~\eqref{generalchibound},
\begin{eqnarray*}M(n,\theta)&\leq& \frac{M_{\textup{Lev}}(n-2,\theta')}{\left(\frac{\int_{\cal{C}^{\theta,\theta'}_{\delta,\eta,2}}(1-|\boldsymbol{u}|^2)^{\frac{n-4}{2}}d\boldsymbol{u}}{\int_{\mathbb{B}^2}(1-|\boldsymbol{u}|^2)^{\frac{n-4}{2}}d\boldsymbol{u}}\right)}\\&\leq&  \frac{M_{\textup{Lev}}(n-2,\theta')}{\left(\frac{\int_{-\eta}^{\eta}\int_{r-\frac{\gamma}{n}}^{r+\frac{1}{n^{2/3}\log(n)}}u(1-u^2)^{\frac{n-4}{2}}dud\theta}{\int_{0}^{2\pi}\int_{0}^{1}u(1-u^2)^{\frac{n-4}{2}}dud\theta}\right)}\\&=&(1+o(1))\frac{\pi M_{\textup{Lev}}(n-2,\theta')}{\eta(1-r^2)^{\frac{n-2}{2}}e^{\frac{r\gamma}{1-r^2}}}\\&=&(1+o(1))\frac{\sqrt{\pi} \sqrt{1-r^2}M_{\textup{Lev}}(n-2,\theta')}{r\sqrt{\frac{12(1-s')}{r}\left(\frac{1-r^2}{r}-\gamma\right)} M_{\textup{Lev}}(n+1,\theta')e^{\frac{r\gamma}{1-r^2}}}\cdot\underbrace{\frac{\sqrt{2\pi n}rM_{\textup{Lev}}(n+1,\theta')}{(1-r^2)^{\frac{n-1}{2}}}}_{\textup{KL bound}}.
\end{eqnarray*}
As a result of this, our improvement factor for any given $\gamma$ such that $\frac{r\gamma}{1-r^2}<1$ is
\begin{equation*}\label{improvementm2}
(1+o(1))\frac{\sqrt{\pi}\sqrt{1-r^2}}{r\sqrt{\frac{12(1-s')}{r}\left(\frac{1-r^2}{r}-\gamma\right)}e^{\frac{r\gamma}{1-r^2}}}\cdot\frac{M_{\textup{Lev}}(n-2,\theta')}{M_{\textup{Lev}}(n+1,\theta')}.
\end{equation*}
This is minimized when $\gamma=\frac{1-r^2}{2r}$ with minimum value
\begin{equation*}\label{improvementm2opt}
(1+o(1))\sqrt{\frac{\pi}{6e(1-s')}}\cdot\frac{M_{\textup{Lev}}(n-2,\theta')}{M_{\textup{Lev}}(n+1,\theta')}.
\end{equation*}
\end{proof}
By~\cite{KL}, we have as $n\rightarrow\infty$
\begin{equation*}
\frac{1}{n}\log_2(M(n,\theta'))\leq (1+o(1))\left(\frac{1+\sin(\theta')}{2\sin(\theta')}\log_2\left(\frac{1+\sin(\theta')}{2\sin(\theta')}\right)-\frac{1-\sin(\theta')}{2\sin(\theta')}\log_2\left(\frac{1-\sin(\theta')}{2\sin(\theta')}\right)\right).
\end{equation*}
As a result, on geometric average, the presence of $\frac{M_{\textup{Lev}}(n-2,\theta')}{M_{\textup{Lev}}(n+1,\theta')}$ leads to an improvement by a factor of
\begin{equation*}
\sqrt{\frac{\pi}{6e(1-s')}}2^{-3\left(\frac{1+\sin(\theta')}{2\sin(\theta')}\log_2\left(\frac{1+\sin(\theta')}{2\sin(\theta')}\right)-\frac{1-\sin(\theta')}{2\sin(\theta')}\log_2\left(\frac{1-\sin(\theta')}{2\sin(\theta')}\right)\right)}.
\end{equation*}
This tends to $0$ as $\theta'\rightarrow 0^+$; however, it leads to the improvement factor $0.2944$ on geometric average when $\theta'=\theta^*$, which is not better than $0.2304...$ obtained on geometric average from the $V_1(\mathbb{R}^n)$ case with $\theta'=\theta^*$.
\section{Improvement factor for $V_m(\mathbb{R}^n)$ with $m\geq 3$}\label{section:mhigher}
We determine the improvement factor obtained for an arbitrary $m\geq 3$. This is the setting in which we average over the Stiefel manifold $V_m(\mathbb{R}^n)$. In this setup, we have angles $0<\theta<\theta'\leq\frac{\pi}{2}$ as before, $\delta=\frac{\gamma}{n}$ for some constant $\gamma$, and $\eta>0$. As before, we define $r,R$ as in~\eqref{goodR} and~\eqref{bigRfirst}, and let
\[\cal{C}^{\theta,\theta'}_{\delta,\eta,m}:=\left\{\boldsymbol{u}\in\mathbb{B}^m:r-\delta\leq|\boldsymbol{u}|\leq R\textup{ and }\left<\boldsymbol{u},\boldsymbol{p}\right>\geq |\boldsymbol{u}|\cos(\eta)\right\}.\]
We then take $F=\chi_{\cal{C}^{\theta,\theta'}_{\delta,\eta,m}}$ to be the characteristic function of this region, and $g=g_{n-m,\theta'}$ Levenshtein's optimal polynomial for $M(n-m,\theta')$. Associated to this data, we define the function $H$ as before by averaging over $V_m(\mathbb{R}^n)$. We want to find parameters $\gamma,\eta$ such that $H(t)\leq 0$ for every $t\in[-1,s]$ when $n$ is large enough and such that $\gamma,\eta$ are optimized so that we obtain the best possible upper bound to $M(n,\theta)$ from $H$. By Proposition~\ref{gcorollary}, since $\delta=o(n^{-2/3})$, for $t=s$, it suffices to require that for large enough $n$,
\begin{equation}\label{mhigherineq}
\iint_{(\cal{C}^{\theta,\theta'}_{\delta,\eta,m}\times\cal{C}^{\theta,\theta'}_{\delta,\eta,m})\cap(\mathbb{B}^m_{r+\frac{1}{n^{2/3}\log(n)}}\times\mathbb{B}^m_{r+\frac{1}{n^{2/3}\log(n)}})}\left(\frac{s-\left<\boldsymbol{u},\boldsymbol{v}\right>}{\sqrt{(1-|\boldsymbol{u}|^2)(1-|\boldsymbol{v}|^2)}}-s'\right)(1-|\boldsymbol{u}|^2)^{\frac{n-m-3}{2}}(1-|\boldsymbol{v}|^2)^{\frac{n-m-3}{2}}d\boldsymbol{u}d\boldsymbol{v}<0.
\end{equation}
We rewrite the integral above using spherical coordinates. Points in $\mathbb{R}^m$ may be parametrized using spherical coordinates $\rho,\varphi_1,\hdots,\varphi_{m-1}$ such that $\rho\in[0,\infty)$, $\varphi_1,\hdots,\varphi_{m-2}\in[0,\pi]$, and $\varphi_{m-1}\in[0,2\pi)$, and in which cartesian coordinates are represented by
\[\begin{cases}x_1=\rho\cos(\varphi_1)\\
x_2=\rho\sin(\varphi_1)\cos(\varphi_2),\\
x_3=\rho\sin(\varphi_1)\sin(\varphi_2)\cos(\varphi_3),\\
\vdots\\
x_{m-1}=\rho\sin(\varphi_1)\hdots\sin(\varphi_{m-2})\cos(\varphi_{m-1}),\\
x_m=\rho\sin(\varphi_1)\hdots\sin(\varphi_{m-2})\sin(\varphi_{m-1}).\end{cases}\]
The volume form in spherical coordinates is given by $\rho^{m-1}\sin^{m-2}(\varphi_1)\sin^{m-3}(\varphi_2)\hdots\sin(\varphi_{m-1})d\rho d\varphi_1\hdots d\varphi_{m-1}$. In spherical coordinates, the inequality~\eqref{mhigherineq} is equivalent to the negativity of
\begin{equation}\label{mhigherineq2}
s\left(\int_{r-\frac{\gamma}{n}}^{r+\frac{1}{n^{2/3}\log(n)}}u^{m-1}(1-u^2)^{\frac{n-m-4}{2}}\right)^2-\frac{\sin^{2(m-1)}(\eta)}{(m-1)^2\left(\int_{0}^{\eta}\sin^{m-2}(\varphi)d\varphi\right)^2}\left(\int_{r-\frac{\gamma}{n}}^{r+\frac{1}{n^{2/3}\log(n)}}u^{m}(1-u^2)^{\frac{n-m-4}{2}}\right)^2-s'\left(\int_{r-\frac{\gamma}{n}}^{r+\frac{1}{n^{2/3}\log(n)}}u^{m-1}(1-u^2)^{\frac{n-m-3}{2}}\right)^2.
\end{equation}
By a linear change of variables, 
\[\int_{r-\frac{\gamma}{n}}^{r+\frac{1}{n^{2/3}\log(n)}}u^{m-1}(1-u^2)^{\frac{n-m-4}{2}}du=\frac{(1-r^2)^{\frac{n-m-4}{2}}}{n}\int_{-\gamma}^{\frac{n^{1/3}}{\log(n)}}\left(r+\frac{t}{n}\right)^{m-1}\left(1-\frac{2r\frac{t}{n}+\frac{t^2}{n^2}}{1-r^2}\right)^{\frac{n-m-4}{2}}dt,\]
and
\[\int_{r-\frac{\gamma}{n}}^{r+\frac{1}{n^{2/3}\log(n)}}u^{m}(1-u^2)^{\frac{n-m-4}{2}}du=\frac{(1-r^2)^{\frac{n-m-4}{2}}}{n}\int_{-\gamma}^{\frac{n^{1/3}}{\log(n)}}\left(r+\frac{t}{n}\right)^m\left(1-\frac{2r\frac{t}{n}+\frac{t^2}{n^2}}{1-r^2}\right)^{\frac{n-m-4}{2}}dt.\]
Similarly,
\[\int_{r-\frac{\gamma}{n}}^{r+\frac{1}{n^{2/3}\log(n)}}u^{m-1}(1-u^2)^{\frac{n-m-3}{2}}du=\frac{(1-r^2)^{\frac{n-m-3}{2}}}{n}\int_{-\gamma}^{\frac{n^{1/3}}{\log(n)}}\left(r+\frac{t}{n}\right)^{m-1}\left(1-\frac{2r\frac{t}{n}+\frac{t^2}{n^2}}{1-r^2}\right)^{\frac{n-m-3}{2}}dt.\]
Inequality~\eqref{mhigherineq2} is then equivalent to
\begin{small}
\begin{eqnarray*}
0&>&s\left(\int_{-\gamma}^{\frac{n^{1/3}}{\log(n)}}\left(r+\frac{t}{n}\right)^{m-1}\left(1-\frac{2r\frac{t}{n}+\frac{t^2}{n^2}}{1-r^2}\right)^{\frac{n-m-4}{2}}dt\right)^2-\frac{\sin^{2(m-1)}(\eta)}{(m-1)^2\left(\int_{0}^{\eta}\sin^{m-2}(\varphi)d\varphi\right)^2}\left(\int_{-\gamma}^{\frac{n^{1/3}}{\log(n)}}\left(r+\frac{t}{n}\right)^m\left(1-\frac{2r\frac{t}{n}+\frac{t^2}{n^2}}{1-r^2}\right)^{\frac{n-m-4}{2}}dt\right)^2\\&&
-s'(1-r^2)\left(\int_{-\gamma}^{\frac{n^{1/3}}{\log(n)}}\left(r+\frac{t}{n}\right)^{m-1}\left(1-\frac{2r\frac{t}{n}+\frac{t^2}{n^2}}{1-r^2}\right)^{\frac{n-m-3}{2}}dt\right)^2\\
&=&-\frac{2r}{n}\left(\int_{-\gamma}^{\frac{n^{1/3}}{\log(n)}}t\left(r+\frac{t}{n}\right)^{m-1}\left(1-\frac{2r\frac{t}{n}+\frac{t^2}{n^2}}{1-r^2}\right)^{\frac{n-m-4}{2}}dt\right)\left(\int_{-\gamma}^{\frac{n^{1/3}}{\log(n)}}\left(r+\frac{t}{n}\right)^{m-1}\left(1-\frac{2r\frac{t}{n}+\frac{t^2}{n^2}}{1-r^2}\right)^{\frac{n-m-4}{2}}dt\right)\\
&&+\left(1-\frac{\sin^{2(m-1)}(\eta)}{(m-1)^2\left(\int_{0}^{\eta}\sin^{m-2}(\varphi)d\varphi\right)^2}\right)\left(\int_{-\gamma}^{\frac{n^{1/3}}{\log(n)}}\left(r+\frac{t}{n}\right)^m\left(1-\frac{2r\frac{t}{n}+\frac{t^2}{n^2}}{1-r^2}\right)^{\frac{n-m-4}{2}}dt\right)\\
&&+s'(1-r^2)\left(\int_{-\gamma}^{\frac{n^{1/3}}{\log(n)}}\left(r+\frac{t}{n}\right)^{m-1}\left(1-\frac{2r\frac{t}{n}+\frac{t^2}{n^2}}{1-r^2}\right)^{\frac{n-m-4}{2}}dt\right)^2-s'(1-r^2)\left(\int_{-\gamma}^{\frac{n^{1/3}}{\log(n)}}\left(r+\frac{t}{n}\right)^{m-1}\left(1-\frac{2r\frac{t}{n}+\frac{t^2}{n^2}}{1-r^2}\right)^{\frac{n-m-3}{2}}dt\right)^2.
\end{eqnarray*}
\end{small}
Note that
\begin{small}
\begin{eqnarray*}&&\left(\int_{-\gamma}^{\frac{n^{1/3}}{\log(n)}}\left(r+\frac{t}{n}\right)^{m-1}\left(1-\frac{2r\frac{t}{n}+\frac{t^2}{n^2}}{1-r^2}\right)^{\frac{n-m-4}{2}}dt\right)^2-\left(\int_{-\gamma}^{\frac{n^{1/3}}{\log(n)}}\left(r+\frac{t}{n}\right)^{m-1}\left(1-\frac{2r\frac{t}{n}+\frac{t^2}{n^2}}{1-r^2}\right)^{\frac{n-m-3}{2}}dt\right)^2\\&=&\left(\int_{-\gamma}^{\frac{n^{1/3}}{\log(n)}}\left(r+\frac{t}{n}\right)^{m-1}\left(1-\frac{2r\frac{t}{n}+\frac{t^2}{n^2}}{1-r^2}\right)^{\frac{n-m-4}{2}}\left(1-\sqrt{1-\frac{2r\frac{t}{n}+\frac{t^2}{n^2}}{1-r^2}}\right)dt\right)\left(\int_{-\gamma}^{\frac{n^{1/3}}{\log(n)}}\left(r+\frac{t}{n}\right)^{m-1}\left(1-\frac{2r\frac{t}{n}+\frac{t^2}{n^2}}{1-r^2}\right)^{\frac{n-m-4}{2}}\left(1+\sqrt{1-\frac{2r\frac{t}{n}+\frac{t^2}{n^2}}{1-r^2}}\right)dt\right)\\
&=&\frac{2r}{n(1-r^2)}\left(\int_{-\gamma}^{\frac{n^{1/3}}{\log(n)}}\left(r+\frac{t}{n}\right)^{m-1}\left(1-\frac{2r\frac{t}{n}+\frac{t^2}{n^2}}{1-r^2}\right)^{\frac{n-m-4}{2}}dt\right)\left(\int_{-\gamma}^{\frac{n^{1/3}}{\log(n)}}t\left(r+\frac{t}{n}\right)^{m-1}\left(1-\frac{2r\frac{t}{n}+\frac{t^2}{n^2}}{1-r^2}\right)^{\frac{n-m-4}{2}}dt\right)+o(n^{-1}).\end{eqnarray*}
\end{small}
Therefore, it is sufficient to require that for large enough $n$,
\begin{small}
\begin{eqnarray*}
&&-\frac{2r}{n}\left(\int_{-\gamma}^{\frac{n^{1/3}}{\log(n)}}t\left(r+\frac{t}{n}\right)^{m-1}\left(1-\frac{2r\frac{t}{n}+\frac{t^2}{n^2}}{1-r^2}\right)^{\frac{n-m-4}{2}}dt\right)\left(\int_{-\gamma}^{\frac{n^{1/3}}{\log(n)}}\left(r+\frac{t}{n}\right)^{m-1}\left(1-\frac{2r\frac{t}{n}+\frac{t^2}{n^2}}{1-r^2}\right)^{\frac{n-m-4}{2}}dt\right)\\
&&+\left(1-\frac{\sin^{2(m-1)}(\eta)}{(m-1)^2\left(\int_{0}^{\eta}\sin^{m-2}(\varphi)d\varphi\right)^2}\right)\left(\int_{-\gamma}^{\frac{n^{1/3}}{\log(n)}}\left(r+\frac{t}{n}\right)^m\left(1-\frac{2r\frac{t}{n}+\frac{t^2}{n^2}}{1-r^2}\right)^{\frac{n-m-4}{2}}dt\right)^2\\
&&+\frac{2rs'}{n}\left(\int_{-\gamma}^{\frac{n^{1/3}}{\log(n)}}\left(r+\frac{t}{n}\right)^{m-1}\left(1-\frac{2r\frac{t}{n}+\frac{t^2}{n^2}}{1-r^2}\right)^{\frac{n-m-4}{2}}dt\right)\left(\int_{-\gamma}^{\frac{n^{1/3}}{\log(n)}}t\left(r+\frac{t}{n}\right)^{m-1}\left(1-\frac{2r\frac{t}{n}+\frac{t^2}{n^2}}{1-r^2}\right)^{\frac{n-m-4}{2}}dt\right)<0.
\end{eqnarray*}
\end{small}
For $m=o(n)$, this is satisfied if 
\begin{eqnarray*}
-\frac{2}{n}\left(\int_{-\gamma}^{\frac{n^{1/3}}{\log(n)}}te^{-\frac{rt}{1-r^2}}dt\right)+r\left(1-\frac{\sin^{2(m-1)}(\eta)}{(m-1)^2\left(\int_{0}^{\eta}\sin^{m-2}(\varphi)d\varphi\right)^2}\right)\left(\int_{-\gamma}^{\frac{n^{1/3}}{\log(n)}}e^{-\frac{rt}{1-r^2}}dt\right)+\frac{2s'}{n}\left(\int_{-\gamma}^{\frac{n^{1/3}}{\log(n)}}te^{-\frac{rt}{1-r^2}}dt\right)<0
\end{eqnarray*}
for sufficiently large $n$. As a result, it suffices to require $\frac{r\gamma}{1-r^2}<1$ and, for large enough $n$,
\begin{equation}\label{largemcondition}\left(1-\frac{\sin^{2(m-1)}(\eta)}{(m-1)^2\left(\int_{0}^{\eta}\sin^{m-2}(\varphi)d\varphi\right)^2}\right)<\frac{2(1-s')}{nr}\frac{\left(\int_{-\gamma}^{\infty}te^{-\frac{rt}{1-r^2}}dt\right)}{\left(\int_{-\gamma}^{\infty}e^{-\frac{rt}{1-r^2}}dt\right)}=\frac{2(1-s')}{nr}\left(\frac{1-r^2}{r}-\gamma\right).\end{equation}
Inequality~\eqref{largemcondition} is equivalent to having, for sufficiently large $n$,
\begin{equation*}
\eta^2<\left(\frac{m+1}{m-1}\right)\frac{2(1-s')}{nr}\left(\frac{1-r^2}{r}-\gamma\right).
\end{equation*}
We choose $\eta^2=(1-o(1))\left(\frac{m+1}{m-1}\right)\frac{2(1-s')}{nr}\left(\frac{1-r^2}{r}-\gamma\right)$.
Any $\gamma,\eta$ satisfying this inequality gives us the bound

\begin{equation*}
M(n,\theta)\leq\frac{M_{\textup{Lev}}(n-m,\theta')}{\left(\frac{\int_{\cal{C}^{\theta,\theta'}_{\delta,\eta,m}}(1-|\boldsymbol{u}|^2)^{\frac{n-m-2}{2}}d\boldsymbol{u}}{\int_{\mathbb{B}^m}(1-|\boldsymbol{u}|^2)^{\frac{n-m-2}{2}}d\boldsymbol{u}}\right)}=\frac{M_{\textup{Lev}}(n-m,\theta')}{\left(\frac{\int_{r-\frac{\gamma}{n}}^{R}u^{m-1}(1-u^2)^{\frac{n-m-2}{2}}du}{\int_{0}^{1}u^{m-1}(1-u^2)^{\frac{n-m-2}{2}}du}\cdot \frac{\int_{\cos(\eta)}^1(1-t^2)^{\frac{m-3}{2}}dt}{\int_{-1}^1(1-t^2)^{\frac{m-3}{2}}dt}\right)},
\end{equation*}
where the equality is a consequence of Proposition~\ref{masscomputation}. Since
\[\int_{r-\frac{\gamma}{n}}^{r+\frac{1}{n^{2/3}\log(n)}}u^{m-1}(1-u^2)^{\frac{n-m-2}{2}}du=\frac{(1-r^2)^{\frac{n-m-2}{2}}}{n}\int_{-\gamma}^{\frac{n^{1/3}}{\log(n)}}\left(r+\frac{t}{n}\right)^{m-1}\left(1-\frac{2r\frac{t}{n}+\frac{t^2}{n^2}}{1-r^2}\right)^{\frac{n-m-2}{2}}dt=(1+o(1))\frac{r^{m-1}(1-r^2)^{\frac{n-m-2}{2}}}{n}\int_{-\gamma}^{\infty}e^{-\frac{rt}{1-r^2}}dt\]
and
\[\int_{0}^{1}u^{m-1}(1-u^2)^{\frac{n-m-2}{2}}du=\frac{1}{2}\int_{0}^{1}v^{\frac{m-2}{2}}(1-v)^{\frac{n-m-2}{2}}dv=\frac{\Gamma\left(\frac{m}{2}\right)\Gamma\left(\frac{n-m}{2}\right)}{2\Gamma\left(\frac{n}{2}\right)},\]
we obtain
\[\frac{\int_{r-\frac{\gamma}{n}}^{R}u^{m-1}(1-u^2)^{\frac{n-m-2}{2}}du}{\int_{0}^{1}u^{m-1}(1-u^2)^{\frac{n-m-2}{2}}du}\geq (1+o(1))\frac{\frac{r^{m-1}(1-r^2)^{\frac{n-m-2}{2}}}{n}\int_{-\gamma}^{\infty}e^{-\frac{rt}{1-r^2}}dt}{\frac{\Gamma\left(\frac{m}{2}\right)\Gamma\left(\frac{n-m}{2}\right)}{2\Gamma\left(\frac{n}{2}\right)}}=(1+o(1))\frac{\frac{r^{m-2}(1-r^2)^{\frac{n-m}{2}}}{n}e^{\frac{r\gamma}{1-r^2}}}{\frac{\Gamma\left(\frac{m}{2}\right)\Gamma\left(\frac{n-m}{2}\right)}{2\Gamma\left(\frac{n}{2}\right)}}.\]
Specializing to $m=3$ and using Stirling's formula, we get
\[\frac{\int_{r-\frac{\gamma}{n}}^{R}u^{2}(1-u^2)^{\frac{n-5}{2}}du}{\int_{0}^{1}u^{2}(1-u^2)^{\frac{n-5}{2}}du}\geq (1+o(1))\frac{\frac{r(1-r^2)^{\frac{n-3}{2}}}{n}e^{\frac{r\gamma}{1-r^2}}}{\frac{\frac{\sqrt{\pi}}{2}\left(\frac{n-5}{2e}\right)^{\frac{n-5}{2}}}{2\left(\frac{n-2}{2e}\right)^{\frac{n-2}{2}}}}=(1+o(1))\frac{2\sqrt{n}r(1-r^2)^{\frac{n-3}{2}}e^{\frac{r\gamma}{1-r^2}}}{\sqrt{2\pi}}.\]
When $m=3$, we also have
\[\frac{\int_{\cos(\eta)}^1(1-t^2)^{\frac{m-3}{2}}dt}{\int_{-1}^1(1-t^2)^{\frac{m-3}{2}}dt}=\sin^2(\eta/2)= (1+o(1))\frac{(1-s')}{nr}\left(\frac{1-r^2}{r}-\gamma\right),\]
where the last equality is a consequence of our choice of $\eta$, depending on $\gamma$. As a result,
\[\frac{\int_{r-\frac{\gamma}{n}}^{R}u^{2}(1-u^2)^{\frac{n-5}{2}}du}{\int_{0}^{1}u^{2}(1-u^2)^{\frac{n-5}{2}}du}\cdot \frac{\int_{\cos(\eta)}^1(1-t^2)^{\frac{m-3}{2}}dt}{\int_{-1}^1(1-t^2)^{\frac{m-3}{2}}dt}\geq (1+o(1))\frac{2(1-s')(1-r^2)^{\frac{n-3}{2}}e^{\frac{r\gamma}{1-r^2}}}{\sqrt{2\pi n}}\left(\frac{1-r^2}{r}-\gamma\right).\]
This is maximized when $\gamma=0$, leading to

\begin{equation*}
M(n,\theta)\leq (1+o(1))\frac{1}{\frac{2(1-s')(1-r^2)^{\frac{n-3}{2}}}{\sqrt{2\pi n}}\left(\frac{1-r^2}{r}\right)}M_{\textup{Lev}}(n-3,\theta')=(1+o(1))\frac{M_{\textup{Lev}}(n-3,\theta')}{2(1-s')M_{\textup{Lev}}(n+1,\theta')}\frac{r\sqrt{2\pi n}M_{\textup{Lev}}(n+1,\theta')}{(1-r^2)^{\frac{n-1}{2}}}
\end{equation*}
When $m=4$, we obtain
\[\frac{\int_{r-\frac{\gamma}{n}}^{R}u^{3}(1-u^2)^{\frac{n-6}{2}}du}{\int_{0}^{1}u^{3}(1-u^2)^{\frac{n-6}{2}}du}\geq (1+o(1))\frac{nr^{2}(1-r^2)^{\frac{n-4}{2}}}{2}e^{\frac{r\gamma}{1-r^2}}\]
and
\[\frac{\int_{\cos(\eta)}^1(1-t^2)^{1/2}dt}{\int_{-1}^1(1-t^2)^{1/2}dt}=\eta\frac{1-\frac{\sin(2\eta)}{2\eta}}{\pi}=\frac{2}{3\pi}(1+o(1))\eta^3= \frac{2}{3\pi}(1+o(1))\left(\frac{10(1-s')}{3nr}\left(\frac{1-r^2}{r}-\gamma\right)\right)^{\frac{3}{2}}.\]
This is maximized when $\gamma=-\frac{1-r^2}{2r}$. As a result,
\begin{eqnarray*}
M(n,\theta)&\leq& (1+o(1))\frac{M_{\textup{Lev}}(n-4,\theta')}{\frac{\int_{r-\frac{\gamma}{n}}^{R}u^{3}(1-u^2)^{\frac{n-6}{2}}du}{\int_{0}^{1}u^{3}(1-u^2)^{\frac{n-6}{2}}du}\frac{\int_{\cos(\eta)}^1(1-t^2)^{1/2}dt}{\int_{-1}^1(1-t^2)^{1/2}dt}}\\&=&(1+o(1))\frac{3\sqrt{e\pi} }{\sqrt{2}\left(5(1-s')\right)^{\frac{3}{2}}}\frac{M_{\textup{Lev}}(n-4,\theta')}{M_{\textup{Lev}}(n+1,\theta')}\cdot\underbrace{\frac{r\sqrt{2\pi n}M_{\textup{Lev}}(n+1,\theta')}{(1-r^2)^{\frac{n-1}{2}}}}_{\textup{old bound}}.
\end{eqnarray*}
As before, using~\cite{KL}, on geometric average, this gives us the improvement factor $0.4267...$.\\
\\
We now address the case when $m\rightarrow\infty$ with $m=o(n)$. In this case,
\[\frac{\int_{\cos(\eta)}^1(1-t^2)^{\frac{m-3}{2}}dt}{\int_{-1}^1(1-t^2)^{\frac{m-3}{2}}dt}=(1+o(1))\sqrt{\frac{m}{2\pi}}\int_0^{\frac{\eta^2}{2}}(2v)^{\frac{m-3}{2}}dv=(1+o(1))\frac{\eta^{m-1}}{m-1}\sqrt{\frac{m}{2\pi}}.\]
As a result, for $m,n\rightarrow\infty$ with $m=o(n)$,
\begin{equation*}
M(n,\theta)\leq(1+o(1))\frac{M_{\textup{Lev}}(n-m,\theta')}{\frac{\frac{r^{m-2}(1-r^2)^{\frac{n-m}{2}}}{n}e^{\frac{r\gamma}{1-r^2}}}{\frac{\Gamma\left(\frac{m}{2}\right)\Gamma\left(\frac{n-m}{2}\right)}{2\Gamma\left(\frac{n}{2}\right)}}\cdot \frac{\eta^{m-1}}{m-1}\sqrt{\frac{m}{2\pi}}}=(1+o(1))\frac{\sqrt{n}M_{\textup{Lev}}(n-m,\theta')(1-r^2)^{\frac{m-1}{2}}}{\sqrt{m}M_{\textup{Lev}}(n+1,\theta')\frac{r^{m-1}e^{\frac{r\gamma}{1-r^2}}}{\frac{\Gamma\left(\frac{m}{2}\right)\Gamma\left(\frac{n-m}{2}\right)}{2\Gamma\left(\frac{n}{2}\right)}}\cdot \frac{\eta^{m-1}}{m-1}}\cdot\underbrace{\frac{r\sqrt{2\pi n}M_{\textup{Lev}}(n+1,\theta')}{(1-r^2)^{\frac{n-1}{2}}}}_{\textup{old bound}}.
\end{equation*}
Using Stirling's formula, we obtain
\begin{eqnarray*}
\frac{\sqrt{n}(1-r^2)^{\frac{m-1}{2}}}{\sqrt{m}\frac{r^{m-1}e^{\frac{r\gamma}{1-r^2}}}{\frac{\Gamma\left(\frac{m}{2}\right)\Gamma\left(\frac{n-m}{2}\right)}{2\Gamma\left(\frac{n}{2}\right)}}\cdot \frac{\eta^{m-1}}{m-1}}
=(1+o(1))\frac{\sqrt{\pi}(1-r^2)^{\frac{m-1}{2}}}{r^{m-1}e^{\frac{r\gamma}{1-r^2}}}\cdot \frac{m^{\frac{m}{2}}e^{\frac{m^2}{8n}(1+o(1))}}{e^{\frac{m}{2}}\left(\left(\frac{m+1}{m-1}\right)\frac{2(1-s')}{r}\left(\frac{1-r^2}{r}-\gamma\right)\right)^{\frac{m-1}{2}}}.
\end{eqnarray*}
This is minimized when $\gamma=-\frac{(m-3)(1-r^2)}{2r}$, in which case we obtain the improvement factor
\begin{eqnarray*}&&(1+o(1))\frac{\sqrt{\pi}(1-r^2)^{\frac{m-1}{2}}e^{\frac{m-3}{2}}}{r^{m-1}}\cdot \frac{m^{\frac{m}{2}}e^{\frac{m^2}{8n}(1+o(1))}}{e^{\frac{m}{2}}\left(\left(\frac{m+1}{m-1}\right)\frac{2(1-s')}{r}\left((m-1)\frac{1-r^2}{2r}\right)\right)^{\frac{m-1}{2}}}\frac{M_{\textup{Lev}}(n-m,\theta')}{M_{\textup{Lev}}(n+1,\theta')}\\
&=& (1+o(1))\cdot \frac{\sqrt{\pi m}e^{\frac{m^2}{8n}(1+o(1))}}{e^{2}\left(1-s'\right)^{\frac{m-1}{2}}}\frac{M_{\textup{Lev}}(n-m,\theta')}{M_{\textup{Lev}}(n+1,\theta')}
\end{eqnarray*}
\section{Improvement factor for $V_m(\mathbb{R}^n)$ with $n-m$ constant}\label{section:constant}
In this section, we sketch the improvement factor obtained by considering Stiefel manifolds $V_{n-k}(\mathbb{R}^n)$ with $k\geq 2$ constant. In this situation, our function $H$ is constructed as follows. We fix angles $0<\theta<\theta'\leq\frac{\pi}{2}$, and take $\delta\in[0,r]$, and $\eta_1,\eta_2\in (0,\pi]$ such that $\eta_1<\eta_2$, where
\[r=\sqrt{\frac{s-s'}{1-s'}}\]
with $s:=\cos(\theta)$ and $s':=\cos(\theta')$ as before. We also let $R=\cos\left(2\arctan\left(\frac{s}{(1-s)(s-s')}\right)+\arccos(r)-\pi\right)>r$ as before. Using these parameters, we define
\[\cal{D}^{\theta,\theta'}_{\delta,\eta_1,\eta_2,n-k}=\left\{\boldsymbol{u}\in\mathbb{B}^{n-k}:r\leq|\boldsymbol{u}|\leq R+\delta\textup{ and }|\boldsymbol{u}|\cos(\eta_2)\leq\left<\boldsymbol{u},\boldsymbol{p}\right>\leq|\boldsymbol{u}|\cos(\eta_1)\right\}.\]
We let $\delta=\frac{\gamma}{n}$, where $\gamma>0$ is a constant to be determined later. Taking $g=g_{k,\theta'}$ Levenshtein's optimal polynomial for bounding $M(k,\theta')$, and $F=\chi_{\cal{D}^{\theta,\theta'}_{\delta,\eta_1,\eta_2,n-k}}:[-1,1]^{n-k}\rightarrow\mathbb{R}$ the characteristic function of $\cal{D}^{\theta,\theta'}_{\delta,\eta_1,\eta_2,n-k}$, we construct $H$ as before. By Proposition~\ref{Hexplicit1}, it suffices to require $\gamma,\eta_2$ to be such that for sufficiently large $n$, we have that for every $t\in[-1,s]$,
\begin{equation*}\label{constantineq}\iint\limits_{\substack{\cal{D}^{\theta,\theta'}_{\delta,\eta_1,\eta_2,n-k}\times\cal{D}^{\theta,\theta'}_{\delta,\eta_1,\eta_2,n-k}\\ -1\leq \frac{t-\left<\boldsymbol{u},\boldsymbol{v}\right>}{\sqrt{(1-|\boldsymbol{u}|^2)(1-|\boldsymbol{v}|^2)}}\leq 1}}g_{k,\theta'}\left(\frac{t-\left<\boldsymbol{u},\boldsymbol{v}\right>}{\sqrt{(1-|\boldsymbol{u}|^2)(1-|\boldsymbol{v}|^2)}}\right)\left(1-\left(\frac{t-\left<\boldsymbol{u},\boldsymbol{v}\right>}{\sqrt{(1-|\boldsymbol{u}|^2)(1-|\boldsymbol{v}|^2)}}\right)^2\right)^{\frac{k-3}{2}}(1-|\boldsymbol{u}|^2)^{\frac{k-3}{2}}(1-|\boldsymbol{v}|^2)^{\frac{k-3}{2}}d\boldsymbol{u}d\boldsymbol{v}\leq 0.\end{equation*}
Note that in this construction, the region was expanded from above, not from below, the reason being that the main contribution to $H$ for $n-m=k$, $k$ a constant and $n\rightarrow\infty$, comes from the parts with largest radii, that is, near radius $R$. Since $k$ is constant in this setting and we want to find \textit{constants} $\gamma,\eta_2$ (independent of $n$), we cannot use the estimates on Levenshtein's optimal polynomials. We focus attention on $t=s$.\\
\\
Since the main contribution to $H$ comes from near $|\boldsymbol{u}|=|\boldsymbol{v}|=R$ as $n\rightarrow\infty$, we focus attention on
\[-1\leq \frac{s-\alpha R^2}{1-R^2}\leq 1\textup{ and } -1\leq \alpha\leq 1\]
being satisfied. This is equivalent to
\[\max\left\{-1,1-\frac{1-s}{R^2}\right\}\leq\alpha\leq \min\left\{1,-1+\frac{1+s}{R^2}\right\}.\]
From these, one obtains that $H(s)\leq 0$ if for sufficiently large $n$,
\begin{equation}\label{constantineq2}
\int_{\cos(\eta_2)}^{\cos(\eta_1)}\int_{\cos(\eta_2)}^{\cos(\eta_1)}\int_{-1}^1g_{k,\theta'}\left(\frac{s-\alpha R^2}{1-R^2}\right)\left(\max\left\{0,1-\left(\frac{s-\alpha R^2}{1-R^2}\right)^2\right\}\right)^{\frac{k-3}{2}}\max\{0,(1-x^2)(1-y^2)-(\alpha-xy)^2)\}^{\frac{n-k-4}{2}}d\alpha dxdy<0.
\end{equation}
Since $n$ is large and $k$ constant, by concentration around $x=y=\cos(\eta_2)$, we may replace inequality~\eqref{constantineq2} with
\begin{equation}\label{constantineq3}
\int_{\cos(2\eta_2)}^1g_{k,\theta'}\left(\frac{s-\alpha R^2}{1-R^2}\right)\left(\max\left\{0,1-\left(\frac{s-\alpha R^2}{1-R^2}\right)^2\right\}\right)^{\frac{k-3}{2}}\left(1-\left(\frac{\alpha-\cos^2(\eta_2)}{\sin^2(\eta_2)}\right)^2\right)^{\frac{n}{2}}d\alpha<0
\end{equation}
for sufficiently large $n$. Generically, $\alpha=\cos^2(\eta_2)$ has the greatest contribution to the integrand, and so inequality~\eqref{constantineq3} is satisfied for sufficiently large $n$ if 
\[g_{k,\theta'}\left(\frac{s-\cos^2(\eta_2) R^2}{1-R^2}\right)\left(\max\left\{0,1-\left(\frac{s-\cos^2(\eta_2) R^2}{1-R^2}\right)^2\right\}\right)^{\frac{k-3}{2}}<0.\]
Since $g_{k,\theta'}(s')=0$, this inequality is satisfied if $\frac{s-\cos^2(\eta_2) R^2}{1-R^2}=s'(1-o(1)),$ that is,
\[\cos(\eta_2)=\sqrt{\frac{s-s'(1-R^2)}{R^2}}(1+o(1)).\]
It is then easy to see that the upper bound on $M(n,\theta)$ obtained from this is
\[2^{n(1+o(1))\log_2\left(\frac{\sin(\theta/2)}{R\left(1-\frac{s-s'(1-R^2)}{R^2}\right)^{\frac{1}{2}}}\right)}\geq 2^{-\frac{n}{2}(1+o(1))},\]
a bound exponentially worse than the bound of Kabatyanskii--Levenshtein~\cite{KL}.

\end{document}